\documentclass[reqno]{amsart}
\usepackage[left = 2.5cm, right=2.5cm]{geometry}
\usepackage{amssymb, amsmath, amsfonts}

\usepackage{natbib}
\usepackage{lscape,comment}
\usepackage{color}
\usepackage{mathtools}
\usepackage{dsfont}
\usepackage{mathrsfs}
\usepackage{bm}
\usepackage{graphicx}
\usepackage{wrapfig}
\usepackage{caption}
\usepackage{tikz}
\usepackage{pgfplots}
\usetikzlibrary{arrows,decorations.markings,arrows.meta}
\usetikzlibrary{calc}

\usepackage[pdftex,plainpages=false,colorlinks,hyperindex,bookmarksopen,linkcolor=red,citecolor=blue,urlcolor=blue]{hyperref}
\DeclareMathAlphabet{\mathpzc}{OT1}{pzc}{m}{it}

\newcommand{\bt}[1]{{\textcolor{red}{#1}}}
\bibpunct{[}{]}{;}{n}{,}{,}

\newtheorem{thm}{Theorem}[section]
\newtheorem{defin}[thm]{Definition}
\newtheorem{rmk}[thm]{Remark}
\newtheorem{prop}[thm]{Proposition}
\newtheorem{lem}[thm]{Lemma}
\newtheorem{ex}[thm]{Example}
\newtheorem{coro}[thm]{Corollary}
\numberwithin{equation}{section}

\def \l { \left( }
\def \r {\right) }
\def \ll { \left\lbrace }
\def \rr { \right\rbrace }

\DeclareMathOperator{\E}{\mathds{E}}

\DeclareMathOperator{\N}{\mathbb{N}}
\DeclareMathOperator{\R}{\mathbb{R}}
\DeclareMathOperator{\cF}{\mathcal{F}}
\DeclareMathOperator{\cG}{\mathcal{G}}

\DeclareMathOperator{\cN}{\mathcal{N}}

\DeclareMathOperator{\cE}{\mathcal{E}}

\DeclareMathOperator{\cC}{\mathcal{C}}

\DeclareMathOperator{\cH}{\mathcal{H}}
\DeclareMathOperator{\cB}{\mathcal{B}}

\DeclareMathOperator{\zq}{\mathpzc{q}}

\newcommand{\Norm}[2]{\left\Vert #1 \right\Vert_{#2}}

\newcommand\restr[2]{{
  \left.\kern-\nulldelimiterspace 
  #1 
  \vphantom{\big|} 
  \right|_{#2} 
  }}

\allowdisplaybreaks

\newcommand{\bP}{\mathds{P}}

\newcommand{\Bk}{\mathfrak{B}}

\allowdisplaybreaks

\begin{document}
	
	\title[Time-changed processes and coupled non-local equations]{Time-changed Markov processes \\ and space-time coupled non-local equations}

	\author[]{Giacomo Ascione$^1$}
	\author[]{Enrico Scalas$^{2}$}
	\author[]{Bruno Toaldo$^3$}
		\author[]{Lorenzo Torricelli$^4$}
        \thanks{The authors E. Scalas and B. Toaldo acknowledge financial support under the National Recovery and Resilience Plan (NRRP), Mission 4, Component 2, Investment 1.1, Call for tender No. 104 published on 2.2.2022 by the Italian Ministry of University and Research (MUR), funded by the European Union – NextGenerationEU– Project Title “Non–Markovian Dynamics and Non-local Equations” – 202277N5H9 - CUP: D53D23005670006 - Grant Assignment Decree No. 973 adopted on June 30, 2023, by the Italian Ministry of Ministry of University and Research (MUR)}

\thanks{The author B. Toaldo would like to thank the Isaac Newton Institute for Mathematical Sciences, Cambridge, for support and hospitality during the programme Stochastic Systems for Anomalous Diffusion, where work on this paper was undertaken. This work was supported by EPSRC grant EP/Z000580/1}
\thanks{The authors would like to thank Prof. I. Bio\v{c}i\'{c} for fruitful discussions on several technical points that improved a previous version of this manuscript}
	\address[]{1: Scuola Superiore Meridionale, Largo S. Marcellino 10 - 80138, Napoli (Italy)}
	\email[]{g.ascione@ssmeridionale.it}
	\address[]{2: Dipartimento di Scienze Statistiche, Sapienza Universit\`a di Roma, Piazzale Aldo Moro 5, 00185, Roma - Italy}
	\email[]{enrico.scalas@uniroma1.it}
	\address[]{3: Dipartimento di Matematica ``Giuseppe Peano'', Università degli Studi di Torino, Via Carlo Alberto 10, 10123, Torino - Italy}
	\email[]{bruno.toaldo@unito.it}
	\address[]{4: Dipartimento di Scienze Statistiche ``Paolo Fortunati'', Alma Mater Studiorum Universit\`a di Bologna, Via dell Belle Arti 41, 40126, Bologna - Italy.}
	\email[]{lorenzo.torricelli2@unibo.it}
	\date{\today}


\begin{abstract}
		In this paper we study coupled fully non-local equations, where a linear non-local operator jointly acts on the time and space variables. We establish existence and uniqueness of the solution. A maximum principle is proved and used to derive uniqueness. Existence is established by providing a stochastic representation based on anomalous processes constructed as a time change via the undershooting of an independent subordinator. This leads to general non-stepped processes with intervals of constancy representing a sticky or trapping effect. Our theory allows these intervals to be dependent on the immediately subsequent jump. These processes include scaling limit of suitable coupled continuous time random walks previously studied in applications, in particular in the context of anomalous diffusion and option pricing. Here we exploit our general theory to obtain a non-local analog of the Black and Scholes equation, addressing the problem of determining the seasoned price of a derivative security, in case the price fluctuations are described by a process whose jumps are dependent on the previous interval.
	\end{abstract}

	\maketitle

	\tableofcontents

	\section{Introduction}
Continuous Time Random Walks (CTRWs) are a popular class of models for anomalous diffusion, see for instance \cite{METZLER20001} and references therein. The suitability of CTRWs to describe anomalous processes led to several contributions, for example in Hamiltonian chaos \cite{zaslavsky3, zaslavsky4, zaslavsky5}, but other applications have been studied  \cite{barkai0, barkai1, fedotov1, fedotov3, fedotov2, gianni,  barkai2, METZLER20001, giannisim}. 
CTRWs are stepped processes ({their} trajectories {are step functions}) and thus they can be seen as a tick-by-tick approximation of possibly continuous anomalous dynamics. From this point of view, there is an increasing interest in scaling limits of CTRWs, that form an interesting class of Markov and non-Markovian processes. Indeed, these limit processes can exhibit continuous trajectories with intervals of constancy, in which the process stays in a fixed position, so that they cannot satisfy the strong Markov property (in the time-homogeneous case). A systematic, but still non-exhaustive, study of CTRWs limit processes has been carried out, see for instance \cite{baemstra, meercoupled, Kolokoltsov2023, Meerschaert2004, meerschaert2008triangular, meerschaert2014semi} and also \cite{hairer2, hairer1} for the connection with the limit of an averaging-homogenization problem. In our context, reference \cite{meerschaert2014semi} is particularly relevant, where the authors proved Markov embeddings of the aforementioned processes in $\R^d$ and studied the corresponding transition semigroups. 

CTRWs limit processes can exhibit both jumps and intervals of constancy. The presence of jumps justifies non-local operators in the space variable appearing in their governing equations. On the other hand, intervals of constancy imply the presence of non-local operators in the time variable. A prototypical CTRWs limit process exhibits a structure either of the form $M(L(t))$ or $M(L(t)-)$, where $M$ is a c\`adl\`ag Feller process and $L$ is the inverse of a subordinator, i.e., a non-decreasing L\'evy process. In general, $M$ and $L$ are not necessarily independent: independence arises, for instance, when the pre-limit CTRWs are made of independent jumps and inter-jump times. The governing equation, in this case, inherits the non-locality in time from the behavior of the process $L$, while the possible non-locality in space only depends on the jumps of the parent Feller process $M$. The independent case, called {\em uncoupled} by physicists, has been widely studied in literature (see, e.g., \cite{Baeumer2001, dovidio, zqc, kochubei, kolokoltsov, patie}). Let us consider an example. If we assume that $M$ is an $\alpha$-stable isotropic L\'evy process on $\R^d$ and $L$ is the inverse of an independent $\beta$-stable subordinator, that typically happens when the generalized central limit theorem can be applied without using triangular array limits, the governing equation of the process $M(L(t))$ is the well-known space-time fractional heat equation with Caputo {fractional} time derivative and spatial fractional Laplacian:
\begin{equation}\label{eq:intro1}
    \left(\partial_t^\beta+(-\Delta)^{\frac{\alpha}{2}}\right)q(x,t)=0.
\end{equation}
The dependent case arises when the CTRW has dependent jumps and inter-arrival time intervals between jumps (see \cite{meercoupled}) while preserving the renewal temporal properties. It is of crucial importance in applications (see, e.g., \cite{coupledappl1, coupledappl2}), but the governing equation has not been fully understood. Some results (for processes in $\mathbb{R}^d$) can be found, for instance, in \cite{baemstra, kologeneral, savtoa}; however the first reference does not provide an evolution equation (in space-time variables) with an initial condition, while in the second and third references, the dependence only induces spatial heterogeneity, that is, the intervals of constancy depends on the position.
According to the particular dependence between $M$ and $L$, it is possible to observe several different behaviors \cite{meercoupled, kotulski1995asymptotic}. Here, we consider the following setting: the {parent Feller process, that in this case will be denoted as $M^\phi$, is of the form $M(S(t))$, where $M$ is a c\'adl\'ag Feller process and $S$ is an independent subordinator, while $L$ is the inverse of the very same subordinator $S$. As a consequence, the time-changed process $X(t)=M^\phi(L(t)-)$ can be also rewritten as $X(t)=M(H(t))$, where $H(t)=S(L(t)-)$ is the undershooting of the subordinator $S$.} 

In this research we study in detail the process {$X(t)$ introduced before}. This process arises from CTRWs with coupled jumps and inter-event times and has trajectories with intervals of constancy followed by dependent jumps. If for instance the process {$X(t)$} admits an interval of constancy for $t \in [t_1,t_2]$, the position it reaches after the jump is exactly {$M(t_2)$}. As we shall prove in the paper, it turns out that the governing equation of such processes presents coupled non-localities in time and space. For instance, if {$M$} is a Brownian motion and $S$ is $\frac{\alpha}{2}$-stable, with $\alpha \in (0,2)$, then the governing equation is a (coupled) fully fractional heat equation:
\begin{equation}\label{eq:intro2}
    \left(\partial_t-\Delta\right)^{\frac{\alpha}{2}} q(x,t)={\frac{t^{-\alpha}}{\Gamma(1-\alpha)}q(x,0)},
\end{equation}
that is a coupled counterpart of \eqref{eq:intro1} for $\beta=\frac{\alpha}{2}$. Such operator has been already adopted in \cite{ricciuti2023semi} in a different setting and it is strictly related to the fully fractional heat equation appearing in \cite{ATHANASOPOULOS20182614, fully1}, which is a particular master equation as discussed in \cite{CaffarelliSilvestre+2014+63+83}. Eq. \eqref{eq:intro2} also appeared (in the one-dimensional case) in the heuristic discussion of reference \cite[Example 5.2]{meercoupled}, inspired by the modeling aspects presented in \cite{coupledappl2}.

{A closely related line of research is provided by the theory of L\'evy walks, where the duration of a motion and the corresponding spatial displacement are intrinsically coupled. In scaling limits, this space-time coupling naturally leads to pseudo-differential equations in the joint time-space variables and, in several homogeneous cases, to fractional material derivatives. The functional limit theorems and governing equations obtained in \cite{kolokoltsov2023fractional, magdziarz2015limit} provide a basic probabilistic and analytic prototype for such coupled non-local dynamics. Moreover, the pointwise representation and numerical treatment of the fractional material derivative developed in \cite{plociniczak2024levy} provide complementary analytic and computational tools for equations arising from L\'evy-walk limits. These works therefore give useful context for the present approach, in which the dependence between trapping intervals and subsequent spatial motion is encoded by a linear non-local operator acting jointly on time and space.}

Anomalous diffusion is not the only possible context in which these processes appear.
A distinct line of research, focused on financial markets and their statistical analyses, considered {CTRWs} as a tick-by-tick model for price fluctuations. Then the corresponding fluid-dynamic limit is considered in order to gain insight on economic measures such as the return distributions, see e.g. \cite{aitsahalia2020, meerschaert2006, scalas2006, Scalas2000} and recently \cite{sojmark2024b}. Such limits typically produce stable-like processes (and thus discontinuous continuous-time trajectories) and variations thereof. 
A particularly consequential aspect of modelling financial markets prices with scaling limit of CTRWs is that the processes thus obtained are non-Markovian and exhibits periods of constant prices.
However, Markovianity can be recovered, as explained in \cite{meerschaert2014semi}, by augmenting the state space with the one variable keeping count of the time elapsed since the particle last move (the ``sojourn time'' or ``age''). 

In general, being able to establish some form of finite higher-dimensional Markovianity is especially useful in the mathematical modeling of random phenomena. Concretely, in a financial valuation context, it can be interpreted as the possibility to establish the theoretical fair value of a contract based only on a finite set of market-observable data.
Two recent papers have tried to bring together some of these aspects. In \cite{scalas2021} the authors chose a semi-Markov multiplicative CTRW model based on a sequence of independent and identically distributed log-normal random variables  and an independent renewal counting process. The limit process turns out to be semi-Markov with continuous trajectories. Although the governing equation of the limit process is studied, the independence between the spatial log-normal innovations and the renewal process, implies that the magnitude of the price moves remain independent from the trade duration, that is, the sojourn time of the price process.  A similar, but more general, model with independence is  presented in \cite{torricelli}, generalizing the earlier suggestions in \cite{magdziarz2009black} and \cite{cartea2010duration}. Crucially, in such work a time-changed model with dependence between temporal and spatial variables is also introduced. The authors use spectral methods to derive the option price for such a model, and comment on a weak form of semi-Markov property (Markov embedding) of the process  as a whole, but do not study the corresponding governing Black and Scholes type equation. Furthermore a financial model with ``stochastic volatility'', based on a system of SDEs subordinated to the inverse of an independent stable subordinator is proposed in \cite{dupret2023subdiffusive}.

Therefore, besides the applications in the context of anomalous diffusion, one significant applied problem addressed by the theory developed in the present paper is the extension and unification and of the contributions of \cite{scalas2021} and \cite{torricelli}, by answering the problem of pricing for, and determining the seasoned value of, an intraday call option written on an asset with dependent duration and returns as in \cite{torricelli}, using the techniques inspired from the semi-Markov analysis of the limiting process in \cite{scalas2021}. 

\subsection*{Structure of the article and summary of the results}

As a guide to readers, below we give a summary of the main results and topics  that are dealt with in each section.
\begin{itemize}
    \item In Section \ref{prelim} we introduce basic notions on uniqueness classes, semigroup actions and Bernstein functions. All that will be relevant in the subsequent sections.
    \item In Section \ref{secoperat} we define a family coupled non-local operators including \eqref{eq:intro2} as very special case, and we study the domain and regularity properties. Moreover, we provide a notion for the solution of the related Cauchy problem.
    \item In Section \ref{sec:undershooting} we review the theory of time change according to subordinators and their inverses, introduce the time-changed Feller process   related to the Cauchy problem of Section \ref{secoperat}, and establish the Markov-additivity of the pair process-subordinator;
    \item In Section \ref{sec:eq} the first main Theorem \ref{thm:main1}, is presented and proved. This states that the expectation of the process defined in Section \ref{sec:undershooting} is the stochastic solution,  in the sense of Definition \ref{def:sol}, of the Cauchy problem related with our class of coupled non-local operators, subject to a further local uniform boundary condition. To prove the theorem, firstly we focus on regularity properties for the expectation,  described in Subsection \ref{secregq}. In Subsection \ref{secproof} we give the details of the proof.
    \item In Section \ref{secfinance} we move to the applications in finance of the theory discussed in the previous sections. We define a process for pricing an intraday call option written on an asset with dependent duration and returns, whose log-price follows a time-changed Feller process as the one of Section \ref{sec:undershooting}. This leads to the second main result, Theorem \ref{thm:main2}, which is the applied counterpart of Theorem \ref{thm:main1}. A second result, proved separately is necessary because the call option initial (terminal) datum does not satisfy the regularity required  in Theorem \ref{thm:main1}.
    The proof of this  theorem is split into several auxiliary results, firstly proving regularity of the solution, then establishing existence, uniqueness and, finally, the renewal equation in separate subsections.
    \item In the appendices some auxiliary results are proved.
\end{itemize}

	\section{Preliminaries}
\label{prelim}
We let $E$ be a locally compact separable Hausdorff space and we set $\mathcal{E}=\cB(E)$, the Borel $\sigma$-algebra on $E$. We equip any interval $[0,T] \subseteq \R_0^+{:=}[0,+\infty)$ and the half-axis $\R_0^+$ with the Borel $\sigma$-algebra. We also denote $\R^+:=(0,+\infty)$. With an abuse of language, we say that a function $f:E \to \R$ belongs to $\mathcal{E}$ if it is Borel-measurable, so that $\cE$ will denote the linear space of Borel-measurable real valued functions on $E$. 
Furthermore, we denote by $C(E)$ the space of continuous functions $f:E \to \R$ and by $C_{\rm b}(E)$ the space of bounded continuous functions $f:E \to \R$. The space $C_{\rm b}(E)$ is a Banach space once it is equipped with the uniform norm
\begin{equation*}
    \Norm{f}{C(E)}:=\sup_{x \in E}|f(x)|.
\end{equation*}
Furthermore, we denote by $C_0(E)$ the space of continuous functions $f:E \to \R$ vanishing at infinity, i.e. such that for all $\varepsilon>0$ there exists a compact $K \subset E$ such that
\begin{equation*}
    \sup_{x \not \in K}|f(x)|<\varepsilon,
\end{equation*}
which is still a Banach space when equipped with the uniform norm.


\subsection{Uniqueness classes and semigroup actions}
Consider a linear operator $G:\cC \subseteq \cE \to \cE$, where $\cC$ is a suitable subspace of $\cE$. For such an operator, we consider the following differential equation
\begin{equation}\label{eq:Cauchyprob}
	\begin{cases}
		\partial_t q(t,x)=Gq(t,x), & t >0, x \in E,\\
		q(0,x)=u(x), & x \in E,
	\end{cases}
\end{equation}
where $u \in \cE$ is suitable initial data. To introduce a notion of solution, we take inspiration from the theory of ordinary differential equations (see \cite[Chapter $2$]{coddington1955theory}).
\begin{defin}
	We say that a function $q:[0,T] \times E \to \R$ is a local Carath\'{e}odory solution of \eqref{eq:Cauchyprob} if and only if
	\begin{itemize}
		\item[$(i)$] $q \in \cB([0,T]\times E)$, $q(\cdot,x) \in C[0,T]$ for any $x \in E$ and $q(t,\cdot) \in \cC$ for any $t \in (0,T]$;
		\item[$(ii)$] $q(0,x)=u(x)$ for any $x \in E$;
		\item[$(iii)$] $Gq(\cdot,x) \in L^1(0,T)$ for any $x \in E$;
		\item[$(iv)$] For any $x \in E$ and $t \ge 0$
		\begin{equation}\label{eq:globCara}
			q(t,x)=u(x)+\int_0^t Gq(s,x)ds.
		\end{equation}
	\end{itemize}
	We say that $q:\R_0^+ \times E \to \R$ is a global Carath\'{e}odory solution of \eqref{eq:Cauchyprob} if its restriction to $[0,T]$ is a local Carath\'{e}odory solution for any $T>0$. 
\end{defin}


\begin{rmk}
	This definition is related to the so-called mild solution of an abstract Cauchy problem (see \cite[Definition 3.1.1]{arendt2011vector}). For instance, suppose that $G$ is a closed linear operator on $\mathcal{C} \subset C_0 \l E \r$. Assume that $q$ is a global Carath\'eodory solution of \eqref{eq:Cauchyprob} such that $Gq(t,\cdot) \in C_0(E)$ and $t \in [0,T] \mapsto \Norm{Gq(t,\cdot)}{C_0(E)} \in \R$ belongs to $L^1[0,T]$ for all $T>0$. Notice also that, by definition of Carath\'eodory solution, $q(t,\cdot) \in \cC$. Hence, {for convenience, we can denote} $q(t)=q(t,\cdot)$ so that $q:\R_0^+ \to C_0(E)$. Then, in particular, by \cite[Proposition 1.1.7]{arendt2011vector}, we have
    \begin{align*}
			\int_0^t Gq(s) ds \, = \, G \int_0^t q(s) ds.
    \end{align*}
    If also $u \in C_0(E)$, we can rewrite \eqref{eq:globCara} in terms of Bochner integrals as
    \begin{equation*}
 		q(t)=u+G\int_0^t q(s)ds,
    \end{equation*}
    i.e., $q$ is a mild solution of the abstract Cauchy problem in $C_0(E)$
    \begin{equation*}
		\begin{cases}
			\partial_t q(t)=Gq(t), & t>0, \\[7pt]
			q(0)=u.
		\end{cases}
    \end{equation*}
\end{rmk}


Now we want to consider some classes of functions $\cC_0 \subseteq \cE$ and $\cC_{\rm sol} \subseteq \cB(\R_0^+ \times E)$ that will play the following role: $\cC_0$ will be the set of initial data so that the problem \eqref{eq:Cauchyprob} admits a unique solution belonging the class $\cC_{\rm sol}$. Let us formalize this (see, for instance, \cite[page $204$]{pascucci2011pde} in the case of the heat equation).

\begin{defin}
	We say that the couple $(\cC_0,\cC_{\rm sol})$, where $\cC_0 \subseteq \cE$ and $\cC_{\rm sol} \subseteq \cB(\R_0^+ \times E)$ is a \textit{uniqueness class} for $G$ if and only if
	\begin{itemize}
		\item[$(i)$] For any $q \in \cC_{\rm sol}$ and any $t>0$ it holds $q(t,\cdot) \in \cC_0 \cap \cC$;
		\item[$(ii)$] For any $u \in \cC_0$ there exists a unique global Carath\'{e}odory solution $q \in \cC_{\rm sol}$ to \eqref{eq:Cauchyprob}.
	\end{itemize}
\end{defin}

Let us give some examples. 
\begin{ex}
Consider $E=\R$, $\cC=C^2(\R)$ and $G=\partial^2_x$. If we set $\cC_0=C_0(\R)$ and 
\begin{equation*}
\cC_{\rm sol}=\{q \in C(\R_0^+ \times \R)\cap C^1(\R^+ \times \R): \ q(t,\cdot) \in C^2_0(\R), \ \forall t>0, \mbox{ and } q(0,\cdot) \in C_0(\R)\},
\end{equation*}
then $(\cC_0,\cC_{\rm sol})$ is a uniqueness class for $G$, as it is well-known that the heat equation
\begin{equation}\label{eq:heateq}
	\begin{cases}
		\partial_t q(t,x)=\partial^2_x q(t,x), & t >0, x \in \R,\\
		q(0,x)=u(x), & x \in \R,
	\end{cases}
\end{equation}
admits a unique solution in $C^2_0(\R)$ if $u\in C_0(\R)$. It is important to notice that such a solution is not unique among all the measurable functions on $\R$, as we could find other unbounded solutions of the same equation, that are usually called non-physical solutions (see \cite[Theorem 2.3.1]{cannon1984one}).
\end{ex}
\begin{ex}
	Again, consider $E=\R$, $\cC=C^2(\R)$ and $G=\partial^2_x$. Now set 
	\begin{equation*}
		\cC_0=\{u \in C(\R): \ \exists c_1,c_2>0, \ \exists \gamma \in [0,2), \ |u(x)| \le c_1e^{c_2|x|^\gamma}, \ \forall x \in \R\}
	\end{equation*}
	and
	\begin{multline*}
		\cC_{\rm sol}=\{q \in C(\R_0^+ \times \R) \cap C^1(\R^+ \times \R): \ q(t,\cdot) \in C^2(\R), \ \forall t>0,\\ \mbox{ and } \ \forall T>0,  \exists c_1,c_2>0, \exists \gamma \in [0,2), \ |q(t,x)| \le c_1e^{c_2|x|^\gamma},\\ \ \forall x \in \R, \ \forall t \in [0,T]\}.
	\end{multline*}
	 Then $(\cC_0,\cC_{\rm sol})$ is a uniqueness class for $G$ as \eqref{eq:heateq} admits a unique solution $q \in \cC_{\rm sol}$ whenever $u \in \cC_0$ (see \cite[Theorem 3.6.1]{cannon1984one}). Let us stress that we could also consider $G$ as the generator of the Brownian motion. In such a case, however, since $\cC_{0} \not \subset C_0(\R)$, the usual theory of Feller semigroups (see \cite[Definition III.6.5]{rogers2000diffusions}) does not apply. 
\end{ex}
\begin{ex}
	Now let $E=[a,b] \subset \R$ for some $a<b$. As before, set $G=\partial^2_x$ with $\cC=C^2(a,b)$. Set 
	\begin{align*}
	\cC_0&=\{u \in C[a,b]: \ u(a)=u(b)=0\}\\
	\cC_{\rm sol}&=\{q \in C(\R_0^+ \times [a,b]) \cap C^2(\R^+ \times (a,b)): \ q(t,a)=q(t,b)=0, \ \forall t > 0\}.
	\end{align*}
	Then $(\cC_0,\cC_{\rm sol})$ is a uniqueness class for $G$. Indeed, in such a case, $q \in \cC_{\rm sol}$ is a solution of \eqref{eq:Cauchyprob} if and only if it solves
	\begin{equation*}
		\begin{cases}
			\partial_t q(t,x)=\partial^2_x q(t,x), & t>0, \ x\in (a,b)\\
			q(0,x)=u(x) & x \in [a,b] \\
			q(t,a)=q(t,b)=0 & t>0,
		\end{cases}
	\end{equation*}
	that is a Dirichlet-Cauchy problem and then admits a unique solution.
\end{ex}
\begin{ex}\label{ex:Feller}
	Let now $(P_t)_{t \ge 0}$ be a Feller semigroup on $C_0(E)$ with generator $(G,{\rm Dom}(G))$. Then we can set $\cC_0=C_0(E)$ and $\cC_{\rm sol}= C(\R_0^+;C_0 \left( E \right) )$. By standard Feller semigroups theory (see \cite[Chapter 3]{arendt2011vector}), we know that $(\cC_0,\cC_{\rm sol})$ is a uniqueness class for $G$. This implies that the abstract setting we are considering can be applied in general to Feller semigroups and their generators. We can also consider strong solutions in this framework by setting $\cC_0 = {\rm Dom}(G)$ and $\cC_{\rm sol} = C^1 \left( \mathbb{R}_0^+ ; {\rm Dom}(G) \right)$.
\end{ex}
\begin{defin}\label{def:semigroupaction}
For a linear operator $G:\mathcal{C} \subseteq \cE \to \cE$ and a uniqueness class $(\cC_0,\cC_{\rm sol})$ for $G$, we define the family of linear operators $(P_t)_{t \ge 0}$ acting on $\cC_0$ as follows:
\begin{itemize}
	\item[$(i)$] $P_0=I$, where $I$ is the identity on $\cC_0$;
	\item[$(ii)$] For any $t>0$, $f \in \cC_0$ and $x \in \cE$, $P_tu(x)=q(t,x)$, where $q$ is the unique Carath\'{e}odory solution of \eqref{eq:Cauchyprob} in $\cC_{\rm sol}$.
\end{itemize}
We call the family $(P_t)_{t \ge 0}$ the semigroup action induced by $G$ on $(\cC_0,\cC_{\rm sol})$.
\end{defin}
{The family} $(P_t)_{t \ge 0}$ is indeed an action of the additive semigroup $(\R_0^+,+)$ on the linear space $\cC_0$. Indeed, fix any $u \in \cC_0$ and $t,s > 0$ and consider $q(t,\cdot)=P_tu$. Now let $\bar{q}(s,\cdot)=P_sq(t,\cdot)$ and observe that $\bar{q}$ is the unique function in $\cC_{\rm sol}$ such that for any $s \ge 0$ and any $x \in E$
\begin{equation}\label{eq:baru}
	\bar{q}(s,x)=q(t,x)+\int_0^s G\bar{q}(\tau,x)d\tau.
\end{equation}
However, by definition,
\begin{multline*}
q(t+s,x)=u(x)+\int_0^{t+s}Gq(\tau,x)d\tau\\
=q(t,x)+\int_t^{t+s}Gq(\tau,x)d\tau=q(t,x)+\int_0^{s}Gq(t+\tau,x)d\tau,
\end{multline*}
that is to say that the function $q(t+\cdot,\cdot) \in \cC_{\rm sol}$ satisfies \eqref{eq:baru} and then 
\begin{equation*}
P_{t+s}u(x)=q(t+s,x)=\bar{q}(s,x)=P_sq(t,x)=P_sP_tu(x)
\end{equation*}
for any $s>0$ and any $x \in E$. We use the name \textit{semigroup action} in place of \textit{semigroup} to avoid any possible confusion with Feller semigroups on Banach spaces, that are a particular case in this theory.

\subsection{Bernstein functions}
We here recall some basic definitions and properties concerning Bernstein functions that will be used throughout the paper. We mainly refer to \cite{schilling2009bernstein} where proofs of the theorems stated below can be found.
\begin{defin}
	A function $\phi:\R^+ \to \R^+_0$ is said to be a Bernstein function if, $\phi (\lambda) \geq 0$, $\phi \in C^\infty(\R^+)$ and for any $n \in \N$
	\begin{equation*}
		(-1)^{n-1}\phi^{(n)}(\lambda) \ge 0 \quad \forall \lambda \in \R^+.
	\end{equation*}
	We denote by $\mathfrak{B}$ the convex cone of Bernstein functions.
\end{defin}
Any function $\phi \in \mathfrak{B}$ admits a special representation.
\begin{thm}
	For any $\phi \in \mathfrak{B}$ there exists a unique triplet $(a_\phi,b_\phi,\nu_\phi)$, where $a_\phi,b_\phi \in \R^+$ and $\nu_\phi$ is a Borel measure on $(0,+\infty)$ with the property
	\begin{equation}\label{eq:Levydef}
		\int_{(0,\infty)}(1 \wedge t)\, \nu_\phi(dt)<\infty,
	\end{equation}
	such that
	\begin{equation}\label{eq:Bdef}
		\phi(\lambda)=a_\phi+b_\phi \lambda +\int_{(0,\infty)}(1-e^{-\lambda t})\nu_\phi(dt).
	\end{equation}
	Vice versa, any triplet $(a_\phi,b_\phi,\nu_\phi)$ with $a_\phi,b_\phi>0$ and $\nu_\phi$ a Borel measure on $(0,+\infty)$ satisfying \eqref{eq:Levydef} defines a Bernstein function $\phi$ by means of \eqref{eq:Bdef}.
\end{thm}
\begin{defin}
	The triplet $(a_\phi,b_\phi,\nu_\phi)$ is called the L\'evy triplet of $\phi$ and, in particular, $\nu_\phi$ is called the L\'evy measure of $\phi$. We denote by $\mathfrak{B}_0$ the subset of $\mathfrak{B}$ composed of Bernstein functions with triplet $(0,0,\nu_\phi)$ and such that $\nu_\phi(0,\infty)=\infty$.
\end{defin}
Any Bernstein function defines the law of a subordinator in a unique way. Namely, for any $\phi \in \mathfrak{B}$ there exists a (canonical) probability space and a (possibly killed) subordinator $\sigma:=\{\sigma(t), \ t \ge 0\}$, i.e. an increasing L\'evy process, such that
\begin{equation*}
	\E[e^{-\lambda \sigma(t)}]=e^{-t\phi(\lambda)}, \ \forall t \ge 0, \ \forall \lambda \ge 0,
\end{equation*}
where $\E$ is the expected value on the aforementioned probability space. Vice versa, if $\sigma$ is a subordinator on a suitable probability space, then the function
\begin{equation*}
	\phi(\lambda)=-\log(\E[e^{-\lambda \sigma(1)}])
\end{equation*}
belongs to $\mathfrak{B}$. For further details we refer to \cite[Chapter $5$]{schilling2009bernstein}, \cite[Chapter III]{bertoin1996levy} and \cite{bertoin1999subordinators}. For the ease of the reader, for any $0\le t_1 < t_2 <\infty$ and any Borel-measurable $f:\R^+ \to \R$, we shall denote
\begin{equation*}
    \int_{t_1}^{t_2}f(s)\nu_\phi(ds):=\int_{(t_1,t_2]}f(s)\nu_\phi(ds) \quad \mbox{and} \quad \int_{t_1}^{\infty}f(s)\nu_\phi(ds):=\int_{(t_1,\infty)}f(s)\nu_\phi(ds).
\end{equation*}
In the next sections we shall use the following notation
\begin{equation*}
	\overline{\nu}_\phi(t):=\nu_\phi(t,\infty), \quad \mathcal{I}_\phi(t):=\int_0^t \overline{\nu}_\phi(s)ds, \quad \mathcal{J}_\phi(t):=\int_0^t s\nu_\phi(ds);
\end{equation*}
{the previous quantities will play a crucial role later in computations.}
A simple integration by parts argument leads to the relation
\begin{equation}\label{Jphi}
	\mathcal{J}_\phi(t)=t\overline{\nu}_\phi(t)-\mathcal{I}_\phi(t).
\end{equation}
A Bernstein function $\phi \in \mathfrak{B}$ is said to be special if the function
\begin{equation*}
	\phi^\star(\lambda)=\frac{\lambda}{\phi(\lambda)} 
\end{equation*}
still belongs to $\mathfrak{B}$. We call $\mathfrak{SB}_0$ the set $\mathfrak{SB}_0:= \mathfrak{SB} \cap \mathfrak{B}_0$, where $\mathfrak{SB}$ is the set of special Bernstein functions.

	\section{The non-local operator}
\label{secoperat}
We here define the coupled non-local operator that will be used in the next sections.
\begin{defin}
Consider a linear operator $G:\cC \subseteq \cE \to \cE$, a uniqueness class $(\cC_0,\cC_{\rm sol})$ for $G$ and its induced semigroup action $(P_t)_{t \ge 0}$ on $\cC_0$. Let $\phi \in \Bk_0$ and $f \in \cB(\R_0^+ \times E)$ such that, for any $t \ge 0$, $f(t,\cdot) \in \cC_0$. For $t>0$ and $x \in E$ we define 
\begin{equation*}
	-\phi\left(\partial_t-G\right)f(t,x)=\int_0^{+\infty}(P_{s}f(t-s,x)\mathds{1}_{[0,t]}(s)-f(t,x))\nu_\phi(ds),
\end{equation*}
	where, for any $t<0$, we set $f(t,x)\equiv 0$, provided the integral is well-defined.
\end{defin}



Now we would like to find some conditions on $f$ so that $-\phi\left(\partial_t-G\right)f(t,x)$ is well-defined at least for some $(t,x) \in \R_0^+ \times E$. This is true as a consequence of the following theorem.
\begin{thm}\label{thm:integrability2}
	Let $G:\cC \subseteq \cE \to \cE$ be a linear operator, $(\cC_0, \cC_{\rm sol})$ be a uniqueness class for $G$ and $(P_t)_{t \ge 0}$ be the semigroup action induced by $G$ on $\cC_0$. Let also $\phi \in \Bk_0$. Consider a function $f \in \cB(\R_0^+ \times E)$ such that $f(t,\cdot) \in \cC_0$. Fix $(t,x) \in \R_0^+ \times E$ and assume further that:
	\begin{itemize}
		\item[$(i)$] It holds 
            \begin{equation*}
                \int_0^t|P_s {f}(t-s,x)-P_s {f}(t,x)|\nu_\phi(ds)<\infty; 
            \end{equation*}
            \item[$(ii)$]
                It holds
            \begin{equation*}
                \int_0^t|P_s {f}(t,x)-{f}(t,x)|\nu_\phi(ds)<\infty .
            \end{equation*}
	\end{itemize}
	Then $-\phi(\partial_t-G)f(t,x)$ is well-defined.
\end{thm}
\begin{proof}
It is sufficient to split the integral defining $-\phi(\partial_t-G)f(t,x)$ as follows:
    \begin{align*}
    -\phi(\partial_t-G)f(t,x)&=\int_0^t (P_sf(t-s,x)-P_sf(t,x))\nu_\phi(ds)\\
    &+\int_0^t (P_sf(t,x)-f(t,x))\nu_\phi(ds)-\overline{\nu}_\phi(t)f(t,x),
    \end{align*}
    where the two integrals are absolutely convergent by $(i)$ and $(ii)$.
\end{proof}
In the following, we shall make explicit use of the Laplace transform. The following theorem provides sufficient conditions to guarantee that $-\phi(\partial_t-G)f(\cdot,x)$ is Laplace transformable for fixed $x \in E$. The proof of this theorem is an immediate consequence of the existence conditions for Laplace transforms (see \cite[Section 1.4]{arendt2011vector}).
\begin{thm}\label{thm:Lapinside2}
	Let $G:\cC \subseteq \cE \to \cE$ be a linear operator, $(\cC_0, \cC_{\rm sol})$ be a uniqueness class for $G$ and $(P_t)_{t \ge 0}$ be the semigroup action induced by $G$ on $\cC_0$. Let also $\phi \in \Bk_0$. Consider a function $f \in \cB(\R_0^+ \times E)$ such that $f(t,\cdot) \in \cC_0$. Fix $x \in E$ and suppose that the assumptions of Theorem \ref{thm:integrability2} hold on $(t,x)$ for a.a. $t \in \R_0^+$. Assume further that
	\begin{itemize}
		\item[$(i)$] There exists $\lambda_0>0$ such that for all $\lambda>\lambda_0$
            \begin{equation*}
                \int_0^{\infty}e^{-\lambda t}\int_0^t|P_s {f}(t-s,x)-P_s {f}(t,x)|\,\nu_\phi(ds)\, dt<\infty ;
            \end{equation*}
            \item[$(ii)$]
                There exists $\lambda_1>0$ such that for all $\lambda>\lambda_1$
            \begin{equation*}
                \int_0^{\infty}e^{-\lambda t}\int_0^t|P_s{f}(t,x)-{f}(t,x)|\, \nu_\phi(ds)\, dt<\infty ;
            \end{equation*}
		\item[$(iii)$]
                There exists $\lambda_2>0$ such that for all $\lambda>\lambda_2$
            \begin{equation*}
                \int_0^{\infty}e^{-\lambda t}\overline{\nu}_\phi(t)|{f}(t,x)|\,  dt<\infty .
            \end{equation*}
	\end{itemize}
	Then, setting $\lambda_\star:=\max\{\lambda_0,\lambda_1,\lambda_2\}$, for all $\lambda>\lambda_\star$ it holds
    \begin{equation}\label{eq:Laptransex}
    \int_{0}^{\infty}e^{-\lambda t}|(-\phi(\partial_t-G)f(t,x))|dt <\infty
    \end{equation}
    and
\begin{align}\label{eq:postFubini}
		\int_0^\infty e^{-\lambda t}(-\phi(\partial_t-G)f(t,x))dt
  =\int_0^\infty \int_0^\infty e^{-\lambda t}(P_sf(t-s,x)\mathds{1}_{[0,t]}(s)-f(t,x)) dt\, \nu_\phi(ds).
	\end{align}
\end{thm}
The assumptions of both Theorems \ref{thm:integrability2} and \ref{thm:Lapinside2} are quite general and must be verified separately depending on the properties of the operator $G$ and the semigroup action $(P_t)_{t \ge 0}$. 

Furthermore, if we equip the state space $E$ with a measure $\mu_E$, then we could ask for some sufficient conditions on $f$ so that $-\phi(\partial_t-G)f(t,x)$ is well-defined for a.a. $t \ge 0$ and $\mu_E$-a.a. $x \in E$. This can be done by developing, for instance, a $L^1_{\rm loc}$-theory. Since, however, in our cases of interest it will be not necessary, we do not report these details.

On the other hand, we shall make use of a maximum principle, that holds under some additional assumptions on the semigroup action $(P_t)_{t \ge 0}$. We recall here some specific properties.
\begin{defin}
    We say that the semigroup action $(P_t)_{t \ge 0}$ is positivity preserving if for all $f \in \cC_0$ with $f(x) \ge 0$ for all $x \in E$ it holds $P_tf(x) \ge 0$ for all $x \in E$ and $t \ge 0$.

    We say that the semigroup action $(P_t)_{t \ge 0}$ is sub-Markov if there exists a one-parameter semigroup of operators $(P^{\rm b}_t)_{t \ge 0}$ acting on $C_{\rm b}(E)$ and coinciding with $(P_t)_{t \ge 0}$ on $\cC_0 \cap C_{\rm b}(E)$ such that $P^{\rm b}_t1 \le 1$.
\end{defin}
Throughout the paper we shall omit the superscript ${\rm b}$ on the additional semigroup $(P_t^{\rm b})_{t \ge 0}$. With these definitions, we are ready to prove the following Positive Maximum Principle (PMP).
\begin{prop}\label{prop:maxprin}
    Assume that the semigroup action $(P_t)_{t \ge 0}$ is positivity preserving and sub-Markov.


  \begin{itemize}
    \item[$(i)$] Assume there exists $(t_\star,x_\star) \in \R^+ \times E$ such that $f(t_\star,x_\star) \ge f(t,x)$ for all $(t,x) \in (0,t_\star] \times E$ and $f(t_\star,x_\star) > f(0,x)$, for all $x \in E$ and $-\phi(\partial_t-G)f(t_\star,x_\star)$ is well-defined. Then 
    \begin{equation*}
    -\phi(\partial_t-G)f(t_\star,x_\star)+\overline{\nu}_\phi(t_\star)f(0,x_\star)<0.
    \end{equation*}
    \item[$(ii)$] Assume there exists $(t_\star,x_\star) \in \R^+ \times E$ such that $f(t_\star,x_\star) \le f(t,x)$ for all $(t,x) \in (0,t_\star] \times E$, $f(t_\star,x_\star) < f(0,x)$ for all $x \in E$  and $-\phi(\partial_t-G)f(t_\star,x_\star)$ is well-defined. Then 
    \begin{equation*}
    -\phi(\partial_t-G)f(t_\star,x_\star)+\overline{\nu}_\phi(t_\star)f(0,x_\star)>0.
    \end{equation*}
    \end{itemize}
\end{prop}
\begin{proof}
    We just prove $(i)$, since the proof of $(ii)$ is analogous. For $t \in [0,t_\star]$ and $x \in E$ define $\overline{f}(t,x)=f(t_\star,x_\star)-f(t,x) \ge 0$. Since $P_t$ is positivity preserving, we have $P_s\overline{f}(t_\star-s,x) \ge 0$ for all $s \in [0,t_\star]$ and $x \in E$, that in turn implies $P_s f(t_\star-s,x) \le (P_s f(t_\star,x_\star))(x)$ for all $s \in [0,t_\star]$ and $x \in E$. Furthermore, since $P_s$ is sub-Markov and linear, we have
    \begin{equation*}
        P_s f(t_\star-s,x)\le f(t_\star,x_\star),  \ \forall s \in [0,t_\star] \ \forall x \in E.
    \end{equation*}
    Hence, it is sufficient to observe that
    \begin{align*}
        -\phi(\partial_t-G)&f(t_\star,x_\star)+\overline{\nu}_\phi(t_\star)f(0,x_\star)\\
        &=\int_0^{t_\star}(P_s f(t_\star-s,x_\star)-f(t_\star,x_\star))\nu_\phi(ds)\\
        &\quad +\overline{\nu}_\phi(t_\star)(f(0,x_\star)-f(t_\star,x_\star))<0.
    \end{align*}
\end{proof}
In the next section, we shall focus on the following equation:
\begin{equation}\label{eq:q}
\begin{cases}
    \phi(\partial_t-G)q(t,x)=\overline{\nu}_\phi(t) q(0,x) & t>0, \ x \in E \\
    q(0,x)=u(x) & x \in E,
\end{cases}
\end{equation}
where $u$ is suitable initial data. Let us give the definition of solution.
\begin{defin}\label{def:sol}
    We say that $q:\R_0^+ \times E \to \R$ is a solution of \eqref{eq:q} if  $q$ is continuous, $\phi(\partial_t-G)q(t,x)$ is well-defined for all $(t,x) \in \R^+ \times E$ and the equality in \eqref{eq:q} hold.
\end{defin}

We remark that operators in the form of $\phi \l \partial_t - G \r$ have already been considered in the literature. In \cite{fully1} the authors considered the equation $\l \partial_t - \Delta \r^\alpha q = q^p$, $\alpha \in (0,1)$, $p>0$, for $x \in \mathbb{R}^d$, $t \in (0,T)$,  and studied the existence and behaviour of blow-up. In \cite{ATHANASOPOULOS20182614} the authors considered the equation $\l \partial_t - \Delta \r^\alpha q =0$, $\alpha \in (0,1)$, in connection with the obstacle problem for the fractional Laplacian. Furthermore, in \cite{CaffarelliSilvestre+2014+63+83} the authors considered coupled operators of this kind: $\phi \l \partial_t - G \r$, under suitable assumptions, and they provided H\"older estimates for integro-differential equations involving those operators.

In this paper, we are interested in proving the existence and uniqueness of the solutions to \eqref{eq:q} and relating them to some appropriate non-Markov processes. We shall consider two specific cases of interest:
\begin{itemize}
\item In the first case, we shall assume that $G$ is the generator of a Feller semigroup on $C_0(E)$. In such a case \eqref{eq:q} is shown to be the governing equation of a large family of time-changed Feller processes. 
\item In the second case, we shall focus on a sub-diffusive Black-Scholes equation in the context of a sub-diffusive market. Notice that, in such a situation, we could consider $G$ to be the generator of a geometric Brownian motion, however the initial (or final) data we consider does not belong to $C_0(\R^+)$. For such a reason, a deep investigation of the uniqueness classes related to the problem is needed. Even in this case, however, the solution can be related to a suitably time-changed geometric Brownian motion.
\end{itemize}
In both cases, the existence of the solution can be proved using a stochastic representation, namely a suitable time-change of a Markov process generated by $G$. This idea will be described in the next section. 
\section{The undershooting of a subordinator and related time-changed processes}\label{sec:undershooting}
Fix a filtered probability space $(\Omega, \mathcal{F}_\infty,\{\mathcal{F}_t\}_{t \ge 0},\bP)$. Throughout the paper, we will denote filtrations with a calligraphic letter without subscript, as, for instance, $\cF:=\{\cF_t\}_{t \ge 0}$. Let $S=\{S(t), \ t \ge 0\}$ be a real-valued, strictly increasing c\`adl\`ag process with independent and stationary increments.  In particular, we have that $S_0=\{S_0(t), \ t \ge 0\}$ defined as $S_0(t)=S(t)-S(0)$ is a subordinator. We define the first-passage processes
\begin{align*}
    L(t)&:=\inf\{u \ge 0: \ S(u)>t\}\\
    L_0(t)&:=\inf\{u \ge 0: \ S_0(u)>t\},
\end{align*}
where for all $t \ge 0$ it holds $L(t)=L_0(t-S(0))$. Since $S$ is assumed to be strictly increasing, $L$ is continuous and increasing.  
Denote by $D(\R_0^+)$ the space of real-valued c\`adl\`ag functions and for any $f \in D(\R_0^+)$ and $t>0$ let $f(t-)=\lim_{\delta \uparrow 0}f(t+\delta)$, while we set $f(0-):=f(0)$. We define
\begin{align}
\label{eq:HH0Def}
H(t):=S_0(L(t)-) \qquad H_0(t):=S_0(L_0(t)-).
\end{align}
In particular, $H_0$ is called the undershooting of the subordinator $S_0$, see \cite[Section 1.4]{bertoin1999subordinators}. Concerning $H$, we have, for all $t \ge 0$
\begin{equation*}
H(t)=S(L(t)-)-S(0)=H_0(t-S(0)).
\end{equation*}
Throughout the paper, we assume that $\E[e^{-\lambda S_0(t)}]=e^{-t\phi(\lambda)}$, where $\phi \in \Bk_0$. The potential measure of the subordinator $S_0$ is defined as
 \begin{equation*}
 	U^\phi(t)=\E[L_0(t)].
 \end{equation*} 
 It is well-known (see \cite[Theorem~10.3 and Equation (10.9)]{schilling2009bernstein}) that if $\phi \in \mathfrak{SB}_0$ then there exists a non-increasing function $u^\phi:\R^+ \to \R^+$ (called the potential density) such that $\int_0^1 u^\phi(t)dt<\infty$ and
 \begin{equation*}
 	U^\phi(dt)=u^\phi(t)dt.
 \end{equation*}
 Since $u^\phi$ is non-increasing and non-negative, $U^\phi$ is concave. We recall (see \cite[Theorem 10.9]{schilling2009bernstein}) that $(u^\phi,\bar{\nu}_\phi)$ constitute a Sonine pair, meaning that
 \begin{equation*}
 	\int_0^t u^\phi(\tau)\bar{\nu}_\phi(t-\tau)d\tau=1.
 \end{equation*}
 Such a condition, called Sonine condition, has been first introduced in \cite{sonine1884generalisation} {and became very popular in fractional calculus (see, e.g., \cite{hanyga2020comment, kochubei, LuchkoYamamoto2016GeneralTimeFractionalDiffusion})}.
In general, for all Borel sets $A \subseteq [0,t]$ (see \cite[Lemma 1.10]{bertoin1999subordinators}),
\begin{equation*}
\bP(H_0(t) \in A)=\int_{A}\overline{\nu}_\phi(t-y)U^\phi(dy).
\end{equation*}
Furthermore, if $\phi \in \mathfrak{SB}_0$, it holds
\begin{equation}\label{eq:densityH}
\bP(H_0(t) \in A)=\int_{A}\overline{\nu}_\phi(t-y)u^\phi(y)\, dy.
\end{equation}

Since it is a suitable time translation of the undershooting $H_0$ of $S_0$, we refer to $H$ as the undershooting of $S$. Since it is the composition of the c\`agl\`ad function $S_0(\cdot -)$ with the continuous function $L(\cdot)$, it is true that also $H$ is c\`agl\`ad. 
%
%
%
Now, we want to use $H$ as a time-change. The fact that it is not right-continuous implies that this is not a time-change in the usual sense, e.g. \cite[Definition V.1.2]{revuz2013continuous}. However, a suitable theory of left-continuous time-changes has been developed in the context of subordinators and their inverses (see \cite{meerschaert2014semi}).

To define our time-changed process, let $M=\ll M(t), t \geq 0 \rr$ be a Feller process independent of $S$. We define $M^\phi(t)=M(S_0(t))$ and $\tau(t)=t+S(0)$ for all $t \ge 0$. Since $(M,\tau)$ is a Feller process, by Phillips' Theorem (see \cite[Theorem 32.1]{ken1999levy}) we know that $(M^\phi,S)$, that is obtained by subordination via $S_0$, is still a Feller process with respect to its natural filtration. {For the reader convenience we recall that Phillips' Theorem states that, by subordinating (via a subordinator with Laplace exponent $\phi$) a strongly continuous contraction semigroup of operators on a Banach space, one obtains another strongly continuous semigroup on the same Banach space. The theorem then also provides a form for the generator of this new semigroup. In our context, this is used by observing that since $(M, \tau)$ is a Feller process then it is associated with a Feller semigroup: the subordination of this semigroup then correponds the time-change of the process and this yields our previous assertion.} In the following, it will be useful to consider the canonical construction of $(M, S)$, as in \cite[Chapter III.1]{revuz2013continuous}, with the only exception that we are using the most general form of the Kolmogorov extension theorem, as in \cite[Theorem 2.4.3]{tao2011introduction}, so that we do not need $E$ to be a Polish space. We recall here the canonical construction. We have $\Omega=D(\R^+_0;E \times \R)$, i.e., the space of c\`adl\`ag functions with values in $E \times \R$, and we set for any $\omega \in \Omega$, denoting $\omega(t)=(\omega_1(t),\omega_2(t)) \in E \times \R$ for $t \ge 0$, $(M(t)(\omega),S(t)(\omega))=\omega(t)$. For $t \ge 0$, we let $\mathcal{F}_t^0=\sigma((M(s),S(s)), \ s \le t)$ and $\mathcal{F}_\infty^0=\sigma\left(\bigcup_{t \ge 0}\mathcal{F}^0_t\right)$. Then, for any $(x,v) \in E \times \R$, we construct, by using Kolmogorov extension theorem, the unique probability measure $\bP^{(x,v)}$ on $(\Omega, \mathcal{F}^0_\infty)$ such that $(M,S)$ admits the desired transition probability function and $\bP^{(x,v)}(M(0)=x, S(0)=v)=1$. Once this is done, we consider the completion $\cF_\infty$ and $\cF_t$, for $t \ge 0$, of $\cF^0_\infty$ and $\cF^0_t$, for $t \ge 0$, with respect to the measure $\bP^{(x,v)}$ and we use as (family of) filtered probability spaces $(\Omega, \cF_\infty, \{\cF_t\}_{t \ge 0}, \bP^{(x,v)})$. In such a case, $(M,S)$ is called a canonical Feller process. Now denote $\mathcal{G}_t^0=\sigma((M^\phi(s),{S}(s)), s \le t)$, $\mathcal{G}_\infty^0:=\sigma\left(\bigcup_{t \ge 0}\cG_t^0\right)$ and with $\mathcal{G}_t$ and $\mathcal{G}_\infty$ their completion with respect to $\bP^{(x,v)}$. Since the composition is a measurable operator (see \cite{billingsley2013convergence}), we know that $\mathcal{G}_\infty \subset \mathcal{F}_\infty$. Furthermore, by Phillip's theorem, $(M^\phi,S)$ is a Feller process with respect to the filtration $\cG:=\{\mathcal{G}_t\}_{t \ge 0}$. Before proceeding, let us recall the following definition, as in \cite[Definition 1.4]{cinlar1}.
\begin{defin}
    Consider a probability space $(\Omega, \cF_\infty, \bP)$ and a filtration $\cF:=\{\mathcal{F}_t\}_{t \ge 0}$ on it. An $E \times \R^m$-valued stochastic process $(X,Y)$ is said to be Markov additive with respect to $\cF$ if it is $\cF$-Markov and for all $A \in \cE$, $B \in \cB(\R^m)$ and $0 \le s \le t$
    \begin{align*}
        P_{s,t}\mathds{1}_{A \times B}(x,y)=P_{s,t}\mathds{1}_{A \times (B-y)}(x,0),
    \end{align*}
    where $B-y=\{z \in \R^m: \ z+y \in B\}$ and $P_{s,t}$ is the two-parameters transition semigroup of $(X,Y)$. Notice that the latter is equivalent to
    \begin{align}\label{eq:Markovadd}
    \begin{split}
        \bP&((X(t),Y(t)-Y(s)) \in A \times B \mid X(s),Y(s))\\
        &=\bP((X(t),Y(t)-Y(s)) \in A \times B \mid X(s)).       
    \end{split}
    \end{align}
\end{defin}
We can prove the following result.
\begin{lem}
\label{markovadd}
    The process $(M^\phi,S)$ is Markov additive with respect to its augmented natural filtration $\cG$.
\end{lem}
\begin{proof}
    We already know that $(M^\phi,S)$ is a Feller process with respect to $\cG$, hence we just need to show \eqref{eq:Markovadd}. Let $s < t$, $A \in \cE$ and $B \in \cB(\R)$, {let $M^\phi (t)$ play the role of $X(t)$ and $S(t)$ play the role of $Y(t)$} and observe that, setting $\Delta S=S(t)-S(s)$,
	\begin{align*}
		&\mathds{P}^{(x,v)} \l (M^\phi(t),\Delta S) \in A \times B \mid M^\phi(s), S (s)  \r \notag \\ 
		= \,  & \mathds{E}^{(x,v)} \left[ \mathds{E}^{(x,v)}  \left[\mathds{1}_A \l M^\phi(t) \r \mathds{1}_B \l \Delta S \r \mid M^\phi(s), S(s),\Delta S \right] \mid M^\phi(s), S(s) \right] \notag \\
		= \, &\mathds{E}^{(x,v)} \left[ \mathds{1}_B \l \Delta S \r \mathds{E}^{(x,v)}  \left[\mathds{1}_A \l M^\phi(t) \r  \mid M^\phi(s), S(s), \Delta S \right] \mid M^\phi(s), S(s) \right] \notag 
  \end{align*}
    Now recall that $M^\phi(t)=M(S(t)-v)=M(\Delta S+S(s)-v)$ and $M^\phi(s)=M(S(s)-v)$, where $S(s)<S(t)$. Now consider three Borel sets $B_1 \in \cE$, $B_2,B_3 \subset \R$ and the event
    \begin{equation*}
    B^\star=\{M^\phi(s) \in B_1\} \cap \{S(s) \in B_2\} \cap \{\Delta S \in B_3\} \in \sigma(M^\phi(s),S(s),\Delta S).
    \end{equation*}
    It holds
    \begin{align*}
        &\E^{(x,v)}[\mathds{1}_{B^\star}\mathds{1}_A (M^\phi(t)) ]\\
        &=\E^{(x,v)}[\mathds{1}_{B_1}(M^\phi(s))\mathds{1}_{B_2}(S(s))\mathds{1}_{B_3}(\Delta S)\mathds{1}_A (M(\Delta S+S(s)-v)) ]\\
        &=\E^{(x,v)}\left[\mathds{1}_{B_3}(\Delta S)\E^{(x,v)}\left[\mathds{1}_{B_1}(M^\phi(s))\mathds{1}_{B_2}(S(s))\mathds{1}_A (M(\Delta S+S(s)-v))\mid \Delta S\right]\right].
    \end{align*}
    Once we notice that $\{M(t), \ t \ge 0\}$ and $S(s)$ are independent of $\Delta S$, if we set
    \begin{equation*}
        F_1(y)=\E^{(x,v)}\left[\mathds{1}_{B_1}(M^\phi(s))\mathds{1}_{B_2}(S(s))\mathds{1}_A (M(y+S(s)-v))\right],
    \end{equation*}
    it holds
    \begin{equation*}
        \E^{(x,v)}[\mathds{1}_{B^\star}\mathds{1}_A (M^\phi(t)) ]=\E^{(x,v)}\left[\mathds{1}_{B_3}(\Delta S)F_1(\Delta S)\right].
    \end{equation*}
    However, we can also write
    \begin{equation}\label{eq:condition1}
    F_1(y)=\E^{(x,v)}\left[\mathds{1}_{B_1}(M^\phi(s))\mathds{1}_{B_2}(S(s))\E^{(x,v)}\left[\mathds{1}_A (M(y+S(s)-v))\mid M^\phi(s), S(s)\right]\right].
    \end{equation}
    Let us focus on the inner conditional expectation. Consider again two sets $B_4 \in \cE$ and $B_5 \in \cE$ and the event
    \begin{equation*}
        B_{\star \star}=\{M^\phi(s) \in B_4\} \cap \{S(s) \in B_5\}.
    \end{equation*}
    Then
    \begin{align*}
        \E^{(x,v)}&\left[\mathds{1}_{B_{\star \star}} \mathds{1}_A(M(y+{S}(s)-v)\right]=\E^{(x,v)}\left[\mathds{1}_{B_4}(M^\phi(s))\mathds{1}_{B_5}({S}(s)) \mathds{1}_A(M(y+{S}(s)-v)\right]\\
        &=\E^{(x,v)}\left[\mathds{1}_{B_5}(S(s))\E^{(x,v)}\left[\mathds{1}_{B_4}(M^\phi(s)) \mathds{1}_A(M(y+S(s)-v)\mid S(s)\right] \right].
    \end{align*}
    Since $M$ and ${S}$ are independent, if we set, for $z \ge v$,
    \begin{align*}
        F_2(z)=\E^{(x,v)}\left[\mathds{1}_{B_4}(M(z-v)) \mathds{1}_A(M(y+z-v)\right],
    \end{align*}
    it holds
    \begin{equation*}
    \E^{(x,v)}\left[\mathds{1}_{B_{\star \star}} \mathds{1}_A(M(y+S(s)-v)\right]=\E^{(x,v)}\left[\mathds{1}_{B_5}(S(s))F_2(S(s))\right].
    \end{equation*}
    However, since $\{M(z-v) \in B_4\} \in \sigma(M(z-v))$, we have
    \begin{align*}
    F_2(z)&=\E^{(x,v)}\left[\mathds{1}_{B_4}(M(z-v)) \E^{(x,v)}\left[\mathds{1}_A(M(y+z-v) \mid M(z-v)\right]\right]\\
    &=\E^{(x,v)}\left[\mathds{1}_{B_4}(M(z-v)) \E^{(M(z-v),v)}\left[\mathds{1}_A(M(y))\right]\right].
    \end{align*}
    Hence
    \begin{align*}
    \E^{(x,v)}&\left[\mathds{1}_{B_{\star \star}} \mathds{1}_A(M(y+S(s)-v)\right]\\&=\E^{(x,v)}\left[\mathds{1}_{B_5}(S(s))\mathds{1}_{B_4}(M^\phi(s))\E^{(M^\phi(s),v)}\left[\mathds{1}_A(M(y))\right]\right]\\
    &=\E^{(x,v)}\left[\mathds{1}_{B_{\star\star}}\E^{(M^\phi(s),v)}\left[\mathds{1}_A(M(y))\right]\right].
    \end{align*}
    Since the events of the form $B_{\star\star}$ constitute a $\pi$-system generating $\sigma(M^\phi(s),S(s))$, this shows that
    \begin{equation*}
        \E^{(x,v)}\left[\mathds{1}_A (M(y+S(s)-v))\mid M^\phi(s), S(s)\right]=\E^{(M^\phi(s),v)}\left[\mathds{1}_A(M(y))\right].
    \end{equation*}
    Define
    \begin{equation*}
        F_3(x,y)=\E^{(x,v)}\left[\mathds{1}_A(M(y))\right]
    \end{equation*}
    so that going back to \eqref{eq:condition1} we get
    \begin{equation*}
        F_1(y)=\E^{(x,v)}\left[\mathds{1}_{B_1}(M^\phi(s))\mathds{1}_{B_2}(S(s))F_3(M^\phi(s),\Delta S)\right]
    \end{equation*}
    and then
    \begin{equation*}
    \E^{(x,v)}[\mathds{1}_{B^\star}\mathds{1}_A (M^\phi(t)) ]=\E^{(x,v)}\left[\mathds{1}_{B^\star}F_3(M^\phi(s),\Delta S)\right],
    \end{equation*}
    i.e., since the events of the form $B^\star$ constitute a $\pi$-system generating $\sigma(M^\phi(s),S(s),\Delta S)$,
    \begin{equation*}
    \E^{(x,v)}[\mathds{1}_A (M^\phi(t))\mid M^\phi(s),S(s),\Delta S]=F_3(M^\phi(s),\Delta S).
    \end{equation*}
    Since $\Delta S$ is independent of both $M^\phi(s)$ and $S(s)$, we get (see Lemma \ref{lem:condexpprop})
    \begin{equation}
        \bP^{(x,v)}((M^\phi(t),\Delta S) \in A \times B \mid M^\phi(s), S(s))=\E^{(x,v)}\left[\mathds{1}_B(\Delta {S})F_3(M^\phi(s),\Delta S) \mid M^\phi(s)\right].
        \label{propmedcond}
    \end{equation}
    Use now \eqref{propmedcond} to see that
    \begin{align*}
    \bP^{(x,v)}&((M^\phi(t),\Delta S) \in A \times B \mid M^\phi(s))\\
    &=\E^{(x,v)}\left[\mathds{1}_B(\Delta S)\mathds{1}_{A}(M^\phi({t})) \mid M^\phi(s)\right]\\
    &=\E^{(x,v)}\left[\E^{(x,v)}\left[\mathds{1}_B(\Delta S)\mathds{1}_{A}(M^\phi({t}))\mid M^\phi(s),S(s)\right] \mid M^\phi(s)\right]\\
    &{=\E^{(x,v)}\left[ \bP^{(x,v)}((M^\phi(t),\Delta S) \in A \times B \mid M^\phi(s), S(s)) \mid M^\phi(s) \right]}\\
    &=\E^{(x,v)}\left[\E^{(x,v)}\left[\mathds{1}_B(\Delta S)F_3(M^\phi(s),\Delta S) \mid M^\phi(s)\right] \mid M^\phi(s)\right]\\
    &=\E^{(x,v)}\left[\mathds{1}_B(\Delta S)F_3(M^\phi(s),\Delta S) \mid M^\phi(s)\right]\\
    &=\bP^{(x,v)}((M^\phi(t),\Delta S) \in A \times B \mid M^\phi(s),S(s)). \qedhere
    \end{align*}
\end{proof}
Now we can define the process $X= \ll X(t), t \geq 0 \rr$ where $X(t)=M^{\phi}(L(t)-)$ that is, if we remember that $H(t) = S_0 \l L(t) - \r$, we have that $X(t) = M(H(t))$. Notice, in particular, that $X$ is a c\`agl\`ad process. We can highlight quite a special case. Indeed, if $E=\R^d$ and $M$ is a jump-diffusion (as in \cite{meerschaert2014semi}), then also $M^\phi$ is a jump-diffusion by \cite[Theorem 32.1]{ken1999levy}. Hence, by combining Lemma \ref{markovadd} and \cite[Theorem 4.1]{meerschaert2014semi}, we know that the process $\{(X(t),\gamma(t)), \ t \ge 0\}$ is a time-homogeneous simple Markov process, where $\gamma(t)=\max\{t-S(0)-H(t),0\}$. The process $\gamma$ is usually referred to as the \textit{age} process, while its right-limits, under $v=0$, coincide with the sojourn time of $X$ in the current position, i.e., the r.v. $J_X(t):=t-(0\vee \sup\{s \le t: \ X(s) \neq X(t)\})$, whenever $M$ (and thus $M^\phi$) does not admit holding points (i.e., points on which $M$ is stuck for an exponentially distributed random time). Such processes are related, as observed in \cite{meerschaert2014semi}, on the one hand, to the limits of continuous time random walks and, on the other hand, to the theory of semi-Markov processes (see \cite{harlamov2013continuous}). It is, however, worth noticing that the fact that the semi-Markov property holds for $X$ has not been shown even for $E=\R^d$, but the simple Markov property of $(X,\gamma)$ is already enough to provide interesting stochastic models. We conjecture that the process $X$ is actually semi-Markov even in the general case in which $E$ is a locally compact separable Hausdorff space: this problem is deferred to future research.

		\section{A coupled fully non-local equation for time-changed Feller processes}
\label{sec:eq}
In this section, we want to determine the governing evolution equation of the process $X(t)$ under the condition $S(0)=0$, relating it with \eqref{eq:q}. To do this, we shall consider the case in which $G$ is a generator of a Feller semigroup. Let us now dive into the details. 

Let $M$ be defined as in the previous section and consider its Feller semigroup $\{P_t\}_{t \ge 0}$ on $C_0(E)$. Denote by $G$ its generator, with domain ${\sf Dom}(G)$. As in Example \ref{ex:Feller}, $(C_0(E),C(\R_0^+;C_0(E)))$ is a uniqueness class for $G$. To get uniqueness for our non-local equation \eqref{eq:q}, the condition $q(t,\cdot) \in C_0(E)$ will not be sufficient: we need this to hold locally uniformly in time.
\begin{defin}
Let $u:\R_0^+ \times E \to \R$ and $T>0$. We say that 
\begin{equation*}
    \lim_{x \to \infty}u(t,x)=0 \quad \mbox{ uniformly with respect to }t \in [0,T]
\end{equation*}
if for all $\varepsilon>0$ there exists a compact set $K \subset E$ such that
\begin{equation*}
    \sup_{x \not \in K}|u(t,x)|<\varepsilon, \quad \forall t \in [0,T].
\end{equation*}
We say that
\begin{equation*}
    \lim_{x \to \infty}u(t,x)=0 \quad \mbox{ locally uniformly with respect to }t \ge 0
\end{equation*}
if the same property holds uniformly with respect to $t \in [0,T]$ for all $T>0$.
\end{defin}
\begin{thm}\label{thm:main1}
	Let $\phi \in \mathfrak{SB}_0$ such that $\Norm{\log(\cdot)u^\phi(\cdot)}{L^1[0,1]}<\infty$ and $M$ be a $E$-valued Feller process with semigroup $\{P_t\}_{t \ge 0}$ and generator $(G,{\sf Dom}(G))$. Let also $S$ be a real-valued, strictly increasing c\`adl\`ag process with independent and stationary increments. Assume further that $S$ is independent of $M$ and that the subordinator $S_0(t)=S(t)-S(0)$ is such that $\phi(\lambda)=\log\E^{(x,0)}[e^{-\lambda S_0(1)}]$. Let $H$ be the undershooting of $S$ and $X(t)=M(H(t))$. Finally, suppose that $S_0(t)$ admits a density {$g_{S_0}(\cdot;t)$} under $\bP^{(x,0)}$. Then, for all $u \in {\sf Dom}(G)$, the function
    \begin{equation*}
    q(t,x)=\E^{(x,0)}[u(X(t))]
    \end{equation*}
    is the unique solution of
    	\begin{equation}\label{eq:main1}
		\begin{cases}
		\phi(\partial_t-G)q(t,x)=\overline{\nu}_{\phi}(t)q(0,x), &  t>0, \ x \in E,\\
		q(0,x)=u(x), & x \in E\\
            \lim_{x \to \infty}q(t,x)=0 & \mbox{locally uniformly with respect to }t \ge 0,
		\end{cases}
        	\end{equation}
            i.e., it is the unique function $q: \mathbb{R}_0^+ \times E \mapsto \mathbb{R}$ that satisfies Definition \ref{def:sol} and the limit condition $\lim_{x \to +\infty}q(t,x)=0$ locally uniformly with respect to $t \geq 0$.
\end{thm}

\begin{rmk}
Notice that if $E \subset \R^d$ is a bounded open set with a suitably regular boundary, $f \in C_0(E)$ if and only if $f \in C(\overline{E})$ with $f(x)=0$ on $\partial E$. Furthermore,
\begin{equation*}
    \lim_{x \to \infty}q(t,x)=0 \quad \mbox{ locally uniformly with respect to }t \ge 0
\end{equation*}
if and only if $q \in C(\R_0^+ \times \overline{E})$ with $q(t,x)=0$ for all $x \in \partial E$. In particular \eqref{eq:main1} becomes
    	\begin{equation}\label{eq:CD}
		\begin{cases}
		\phi(\partial_t-G)q(t,x)=\overline{\nu}_{\phi}(t)q(0,x), &  t>0, \ x \in E,\\
		q(0,x)=u(x), & x \in E,\\
            q(t,x)=0 & x \in \partial E,
		\end{cases}
	\end{equation}
 i.e., it is a non-local Cauchy-Dirichlet problem. From a purely analytical point of view and in the special case $\phi(\lambda)=\lambda^\alpha$ and $G= \Delta$ these problems have been studied in \cite{ATHANASOPOULOS20182614} in connection with the obstacle problem for the fractional Laplacian and in \cite{fully1} with a non-linear term. Also in \cite{CaffarelliSilvestre+2014+63+83}, equations of these forms are considered and H\"older estimates are provided.
\end{rmk}

In order to prove {Theorem \ref{thm:main1}}, we need some preliminary regularity estimates on the function $q$ to verify the hypotheses of {Theorems \ref{thm:integrability2} and \ref{thm:Lapinside2}.
\subsection{Regularity of $q$}
\label{secregq}
 First of all, in order to guarantee that we can apply Theorems \ref{thm:integrability2} and \ref{thm:Lapinside2} to the function $q$, we show that it satisfies the assumptions of both theorems.
 This requires some smoothness properties. We start by proving that, for fixed $x \in E$, $q(\cdot, x)$ is absolutely continuous.

\begin{prop}\label{prop:AC}
Under the hypotheses of Theorem~\ref{thm:main1}, for all $x \in E$ the function $q(\cdot,x)$ belongs to ${\rm AC}(\R_0^+)$ with a.e. derivative $\partial_t q(\cdot,x)$ satisfying, for almost all $t \in [0,T]$,
\begin{equation}\label{eq:partialtcontrol}
	|\partial_t q(t,x)|\le C(\Norm{Gu}{C(E)}+\Norm{u}{C(E)})u^\phi(t)
\end{equation}
for a suitable constant $C>0$ (which is independent on $x \in E$). {Remember that $u^\phi (t)$ denotes the potential density}.
\end{prop}
\begin{proof}
	Consider the Feller semigroup $(\widetilde{P}_t)_{t \ge 0}$ acting on any function $w \in C_0(E \times \R)$ as
	\begin{equation*}
		\widetilde{P}_tw(x,s)=\E^{(x,0)}[w(M(t),t+s)].
	\end{equation*}
	Let $\widetilde{A}$ be the generator of $(\widetilde{P}_t)_{t \ge 0}$. By Phillips theorem (see \cite[Theorem 12.6]{schilling2009bernstein}), we obtain that the generator of $(M^\phi,S)$ (which is obtained by subordinating $\widetilde{P}$ by means of the subordinator $S_0$) is given, on functions $w \in {\sf Dom}(\widetilde{A})$, by
	\begin{equation*}
		\widetilde{A}^\phi w=\int_0^{\infty}(\widetilde{P}_sw-w)\nu_\phi(ds)
	\end{equation*}
    {{where the integral is meant} in the sense of Bochner}.
	Now let $u \in {\sf Dom}(G)$ and consider, for any $N \in \N$, a cut-off function of $[-N,N]$, i.e. $\eta_N \in C^\infty_c(\R)$ such that $\eta_N(t)=1$ for $t \in [-N,N]$, $\eta_N(t)=0$ for $t \in [-N-1,N+1]$ and $\eta_N(t) \in [0,1]$ for all $t \in \R$. Without loss of generality, we can assume that $\Norm{\frac{d}{dt}\eta_N}{L^\infty(\R)} \le C_\eta$, where $C_\eta$ is independent of $N$. We define $u_N(x,t)=u(x)\eta_N(t)$. Let us show that $u_N \in {\sf Dom}(\widetilde{A})$. Indeed,
	\begin{align*}
		&\lim_{h \to 0}\frac{\widetilde{P}_hu_N(x,t)-u_N(x,t)}{h}\notag \\ &=\lim_{h \to 0}\frac{\E^{(x,0)}[u(M(h))\mid M(0)=x]\eta_N(t+h)-u(x)\eta_N(t)}{h}\\
		&=\lim_{h \to 0}\frac{\E^{(x,0)}[u(M(h))\mid M(0)=x]-u(x)}{h}\eta_N(t+h)+u(x)\lim_{h \to 0}\frac{\eta_N(t+h)-\eta_N(t)}{h}\\
		&=Gu(x)\eta_N(t)+u(x)\frac{d}{dt}\eta_N(t).
	\end{align*}
	Since the pointwise generator coincides with the classical one (see \cite[Theorem 1.33]{bottcher2013levy}), then this proves that $u_N \in {\sf Dom}(\widetilde{A})$ and
	\begin{equation}\label{eq:ss}
		\widetilde{A}u_N(x,t)=Gu(x)\eta_N(t)+u(x)\frac{d}{dt}\eta_N(t) \ \forall x \in E, \ \forall t \in \R.
	\end{equation}
	By Phillips theorem, we know that $u_N \in {\sf Dom}(\widetilde{A}^\phi)$ and in particular
	\begin{align*}
		\widetilde{A}^\phi u_N&=\int_0^{\infty}(\widetilde{P}_su_N-u_N)\nu_\phi(ds)\\
		&=\int_0^{1}\int_0^s\widetilde{P}_\tau\widetilde{A} u_N d\tau \, \nu_\phi(ds)+\int_1^{\infty}(\widetilde{P}_su_N-u_N)\nu_\phi(ds).
	\end{align*}
	By Bochner theorem, see \cite[Theorem 1.1.4]{arendt2011vector}, we have
	\begin{align*}
		\Norm{\widetilde{A}^\phi u_N}{C(E \times \R)}& \le \int_0^{1}\int_0^s\Norm{\widetilde{P}_\tau\widetilde{A} u_N}{C(E \times \R)} d\tau \, \nu_\phi(ds)\notag \\
  &+\int_1^{\infty}\Norm{\widetilde{P}_su_N-u_N}{C(E \times \R)}\nu_\phi(ds)\\
		&\le \Norm{\widetilde{A}u_N}{C(E \times \R)}\mathcal{J}_\phi(1)+2\Norm{u_N}{C(E \times \R)}\overline{\nu}_\phi(1),
	\end{align*}
 {where $\mathcal{J}_\phi(1) = \int_0^1 s \nu_\phi (ds).$}
	Using \eqref{eq:ss}
        and recalling that $\Norm{\eta_N}{L^\infty(\R)}=1$, we have
	\begin{align*}
		\Norm{\widetilde{A}^\phi u_N}{C(E \times \R)}& \le (\Norm{Gu}{C(E)}+\Norm{u}{C(E)}C_\eta)\mathcal{J}_\phi(1)+2\Norm{u}{C(E)}\overline{\nu}_\phi(1) \\
		&\le C(\Norm{Gu}{C(E)}+\Norm{u}{C(E)}),
	\end{align*}
	for some suitable constant $C$ that is independent of $N$. Now let
	\begin{equation}\label{eq:qNDynkin}
		q_N(t,x,s)=\E^{(x,s)}[u_N(X(t),S(L(t)-))]=\E^{(x,s)}[u_N(M^\phi(L(t)-),S(L(t)-))],
	\end{equation}
	observe that $L(t)$ is a stopping time with respect to the natural filtration of $(M^\phi(\cdot-),S(\cdot-))$ and use Dynkin's formula \cite[Equation (1.55)]{bottcher2013levy} to write
	\begin{equation*}
		q_N(t,x,s)=u_N(x,s)+\E^{(x,s)}\left[\int_0^{L(t)}\widetilde{A}^\phi u_N (M^\phi(\tau),S(\tau))d\tau\right],
	\end{equation*}
	where we also used the fact that {$S$} is stochastically continuous. Now fix $T>0$ and consider $t \in \left[0,\frac{3}{2}T\right]$. We extend $q_N(t,x,s)$ to $t \in \left[-\frac{T}{2},0\right]$ by setting $q_N(t,x,s)={u_N}(x,s)$. Let $h \in \R$ with $0<|h|<\frac{T}{2}$ and define for any $\varphi:\left[-\frac{T}{2},\frac{3}{2}T\right] \to \R$ the quantity
	\begin{equation*}
		D^h\varphi(t):=\frac{\varphi(t+h)-\varphi(t)}{h}.
	\end{equation*}
	We argue first for $h>0$. Observe that, by \eqref{eq:qNDynkin}, for almost all $t \in \left[0,T\right]$. 
	\begin{align*}
		\left|D^h q_N(t,x,0)\right| &\le h^{-1}\E^{(x,0)}\left[\int_{L(t)}^{L(t+h)}|\widetilde{A}^\phi u_N(M^\phi(\tau),S(\tau))|d\tau\right]\\
		&\le \Norm{\widetilde{A}^\phi u_N}{C(E \times \R)}h^{-1}\E^{(x,0)}[L(t+h)-L(t)]\\
		&\le C(\Norm{Gu}{C(E)}+\Norm{u}{C(E)})h^{-1}(U^\phi(t+h)-U^\phi(t))\\
		&\le C(\Norm{Gu}{C(E)}+\Norm{u}{C(E)})u^\phi(t).
\end{align*}    
    If $h<0$ instead, let us distinguish among two cases. If $|h|\le\frac{t}{2}$ then we have
    \begin{equation*}
    \left|D^h q_N(t,x,0)\right|\le C(\Norm{Gu}{C(E)}+\Norm{u}{C(E)})\frac{U^\phi(t)-U^\phi(t+h)}{|h|} \le C(\Norm{Gu}{C(E)}+\Norm{u}{C(E)})u^\phi\left(\frac{t}{2}\right).
    \end{equation*}
    If, instead, $|h| \in \left[\frac{t}{2},t\right)$, then
    \begin{equation*}
    \left|D^h q_N(t,x,0)\right|\le C(\Norm{Gu}{C(E)}+\Norm{u}{C(E)})\frac{U^\phi(t)-U^\phi(t+h)}{|h|} \le 2C(\Norm{Gu}{C(E)}+\Norm{u}{C(E)})\frac{U^\phi\left(t\right)}{t}.
    \end{equation*}
    Finally, if $|h|>t$, we get
    \begin{equation*}
    \left|D^h q_N(t,x,0)\right|\le C(\Norm{Gu}{C(E)}+\Norm{u}{C(E)})\frac{U^\phi(t)}{|h|} \le C(\Norm{Gu}{C(E)}+\Norm{u}{C(E)})\frac{U^\phi\left(t\right)}{t}.
    \end{equation*}
    Hence, we have in general (for any $h \in \R$)
    \begin{equation*}
    \left|D^hq_N(t,x,0)\right| \le C(\Norm{Gu}{C(E)}+\Norm{u}{C(E)})\left(u^\phi(t)+u^\phi\left(\frac{t}{2}\right)+\frac{U^\phi(t)}{t}\right).
    \end{equation*}
	Now, since $u \in C_0(E)$ and $\Norm{\eta_N}{L^\infty(\R)}=1$, we can use the dominated convergence theorem to take the limit as $N \to \infty$ and get
	\begin{equation}\label{eq:bound}
		\left|D^h q(t,x)\right| \le C(\Norm{Gu}{C(E)}+\Norm{u}{C(E)})\left(u^\phi(t)+u^\phi\left(\frac{t}{2}\right)+\frac{U^\phi(t)}{t}\right).
	\end{equation}
	We now use a suitable modification of \cite[Theorem 5.8.3]{evans2010partial} in the case of $L^1$ functions. For completeness we give here all the details of the proof. Consider any sequence $h_n \to 0$, fix $x \in {E}$ and set, for simplicity, $g_n(t)=D^{h_n}q(t,x)$ for $t \in [0,T]$. Before proceeding, notice that for any $T>0$ it holds
    \begin{equation*}
        \int_0^T \frac{U^\phi(t)}{t}\, dt=\int_0^T \int_0^t \frac{u^\phi(s)}{t}\, ds \, dt=\int_0^T \int_s^T \frac{u^\phi(s)}{t}\, dt \, ds=\log(T)U^\phi(T)-\int_0^T\log(s)u^\phi(s) \, ds<\infty,
    \end{equation*}
    where the last inequality follows by assumption. Set then $g^\phi(t)=u^\phi(t)+u^\phi\left(\frac{t}{2}\right)+\frac{U^\phi(t)}{t}$ and notice that $g^\phi \in L^1_{\rm loc}(\R_0^+)$ and $g_n(t) \le C(\Norm{Gu}{C(E)}+\Norm{u}{C(E)})g^\phi(t)$ for all $t \ge 0$. We have then
    \begin{equation*}
	\Norm{g_n}{L^1[0,T]} \le C(\Norm{Gu}{C(E)}+\Norm{u}{C(E)})\Norm{g^\phi}{L^1[0,T]}, \ \forall n \in \N.
    \end{equation*}
    Furthermore, since $g^\phi \in L^1[0,T]$, we know that for any $\varepsilon>0$ there exists $\delta>0$ such that for all measurable $A \subset [0,T]$ such that $|A|<\delta$ it holds
    \begin{equation*}
		\int_{A} g_n(t) dt  \le  C(\Norm{Gu}{C(E)}+\Norm{u}{C(E)})\int_{A} g^\phi(t)\, dt<\varepsilon, \ \forall n \in \N,
\end{equation*}   
	that implies that the sequence $\{g_n\}_{n \in \N}$ is uniformly integrable in $[0,T]$. Hence, by Dunford-Pettis theorem \cite[Theorem 4.30]{brezis2011functional} we know that there exists a function $q \in L^1(0,T)$ such that $g_n \upharpoonright g$. Next, let $\varphi \in C^\infty_c(0,T)$ and observe that
	\begin{equation*}
		\int_0^T D^hq(t,x)\varphi(t)dt=\int_0^{\frac{3}{2}T} D^hq(t,x)\varphi(t)dt=-\int_0^{\frac{3}{2}T} q(t,x)D^{-h}\varphi(t)dt
	\end{equation*}
	and, in particular,
	\begin{equation*}
		\int_0^T g_n(t)\varphi(t)dt=-\int_0^{\frac{3}{2}T} q(t,x)D^{-{h_n}}\varphi(t)dt.
	\end{equation*}
	Taking the limit as $n \to +\infty$ we have
	\begin{equation*}
		\int_0^T g(t)\varphi(t)dt=-\int_0^{\frac{3}{2}T} q(t,x)\varphi'(t)dt=-\int_0^{T} q(t,x)\varphi'(t)dt.
	\end{equation*}
	It follows that $q(\cdot,x) \in W^{1,1}[0,T]$, where $W^{1,1}[0,T]$ is the usual Sobolev space, with weak derivative $\partial_tq(t,x)=g(t)$. In particular, we recall that any function $W^{1,1}[0,T]$ admits an absolutely continuous version. Since $q(\cdot,x)$ is continuous, then it is absolutely continuous with a.e. derivative $\partial_t q(\cdot,x)$. Finally, notice that for $h>0$ we actually have
    \begin{equation*}
        \left|D^hq(t,x)\right| \le C(\Norm{Gu}{C(E)}+\Norm{u}{C(E)})u^\phi(t)
    \end{equation*}
    and taking the limit as $h \to 0^+$ we get the statement.
\end{proof}
Thanks to the previous result, {now} we can prove that $q$ satisfies Items (i) of Theorem \ref{thm:integrability2} and (i) of Theorem \ref{thm:Lapinside2}.
\begin{coro}\label{cor:ACest}
 Under the assumptions of Theorem \ref{thm:main1} we have that for all $(t,x) \in \R_0^+ \times E$,
 \begin{equation*}
 \int_0^t \Norm{P_sq(t-s,\cdot)-P_sq(t,\cdot)}{C(E)}\, \nu_\phi(ds)\le C(\Norm{Gu}{C(E)}+\Norm{u}{C(E)})
 \end{equation*}
 and
 \begin{equation*}
 \int_0^\infty e^{-\lambda t} \int_0^t \Norm{P_sq(t-s,\cdot)-P_sq(t,\cdot)}{C(E)}\, \nu_\phi(ds) \, dt \le \frac{C(\Norm{Gu}{C(E)}+\Norm{u}{C(E)})}{\lambda}
 \end{equation*}
 for all $x \in E$ and $\lambda>0$.
\end{coro}
\begin{proof}
    Let us fix $x \in E$ and observe that by \eqref{eq:partialtcontrol} we have
    \begin{align}\label{eq:partialcontrol2}
        \left|q(t-s,x)-q(t,x) \right| \, \leq \,& \int_0^s \left| \partial_tq(t-\tau,x) \right| d\tau \notag \\
        \leq \,& C(\Norm{Gu}{C(E)}+\Norm{u}{C(E)}) \int_0^s u^\phi(t-\tau)  d\tau.
    \end{align}
    %
    Since $P_s$ is non-expansive, positive preserving and sub-Markov we have for all $x \in E$
    \begin{align}\label{eq:ACest1}
    \begin{split}
    |P_sq(t-s,x)-P_sq(t,x)|&\le P_s\left|q(t-s,\cdot)-q(t,\cdot) \right|(x) \\
    &\le C(\Norm{Gu}{C(E)}+\Norm{u}{C(E)}) \int_0^s u^\phi(t-\tau)  d\tau
    \end{split}
    \end{align}
    hence, taking the supremum, then integrating {and using Fubini},
    \begin{align*}
    \int_0^t &\Norm{P_sq(t-s,\cdot)-P_sq(t,\cdot)}{C(E)}\, \nu_\phi(ds) \\
    &\le C(\Norm{Gu}{C(E)}+\Norm{u}{C(E)}) \int_0^t\int_0^s u^\phi(t-\tau) \, d\tau \, \nu_\phi(ds)\\
    &=C(\Norm{Gu}{C(E)}+\Norm{u}{C(E)})\int_0^t u^\phi(t-\tau)  \int_\tau^t \nu_\phi(ds) \, {d\tau}\\
    &\le C(\Norm{Gu}{C(E)}+\Norm{u}{C(E)})\int_0^tu^\phi(t-\tau)\overline{\nu}_\phi(\tau)\, d\tau\\
    &=C(\Norm{Gu}{C(E)}+\Norm{u}{C(E)}),
    \end{align*}
 {where the first equality is due to the fact that $(\tau,s) \in [0,s] \times [0,t]$ is equivalent to $(\tau,s) \in [0,t] \times [\tau,t]$.}
    Furthermore, we have
    \begin{equation*}
    \int_0^{+\infty}e^{-\lambda t}\int_0^t \Norm{P_sq(t-s,\cdot)-P_sq(t,\cdot)}{C(E)}\, \nu_\phi(ds)\, dt \le \frac{C(\Norm{Gu}{C(E)}+\Norm{u}{C(E)})}{\lambda}
    \end{equation*}
    for all $\lambda>0$.
    
\end{proof}
Next, we show that $q(t,\cdot)$ belongs to ${\sf Dom}(G)$ and we control its uniform norm.
\begin{prop}\label{prop:regular2}
Under the hypotheses of Theorem~\ref{thm:main1}, for all $t \ge 0$ the function $q(t,\cdot)$ belongs to ${\sf Dom}(G)$ and
\begin{equation}\label{eq:controlG}
	\Norm{Gq(t,\cdot)}{C(E)} \le \Norm{Gu}{C(E)}.
\end{equation}	
{Furthermore $\lim_{x \to \infty}q(t,x)=0$ locally uniformly with respect to $t \ge 0$.}
\end{prop}
\begin{proof}
Observe that, since $H(t)$ {defined in equation \eqref{eq:HH0Def}} is independent of $M$, we have
\begin{align}\label{eq:charq}
	q(t,x)=\E^{(x,0)}[u(X(t))]=&\E^{(x,0)}[\E^{(x,0)}[u(X(t)) \mid H(t)]]=\E^{(x,0)}[P_{H(t)}u(x)] \notag \\ = \, &\int_0^t P_\tau u(x)\bar{\nu}_\phi(t-\tau)u^\phi(\tau)d\tau,
\end{align}
where we used the fact that $H(t)$ admits density $\bar{\nu}_\phi(t-\tau)u^\phi(\tau)$ for $\tau \in (0,t)$, see \cite[Lemma 1.10]{bertoin1999subordinators}. Let us stress that
\begin{equation*}
	\int_0^t \Norm{P_\tau u}{C(E)}\bar{\nu}_\phi(t-\tau)u^\phi(\tau)d\tau \le \Norm{u}{C(E)},
\end{equation*}
hence, {a simple application of the dominated convergence theorem shows that $q(t,\cdot) \in C(E)$. Next, we prove that
\begin{equation}
\lim_{x \to \infty}q(t,x)=0 \quad \mbox{ locally uniformly with respect to }t \ge 0,
\end{equation}
which in turn implies that $q(t,\cdot) \in C_0(E)$. To do this, fix $T>0$ and observe that $P_tu \in C_0(E)$ for all $t \ge 0$ and $(P_t)_{t \ge 0}$ is strongly continuous. Let us extend the semigroup by setting $P_t=I$ for all $t \le 0$, in such a way that $(P_t)_{t \in \R}$ is still a strongly continuous family of operators such that $P_tu \in C_0(E)$ for all $t \in \R$. Fix $\varepsilon>0$. For all $t \in [0,T]$ let $K(t)$ be a compact set such that $\sup_{x \not \in K(t)}|P_tu(x)|<\frac{\varepsilon}{2}$. Since $P_{t+h}u$ uniformly converges towards $P_tu$ as $h \to 0$, we know that there exists $\delta(t)>0$ such that for any $s \in (t-\delta(t),t+\delta(t))=:I(t)$ it holds $\sup_{x \not \in K(t)}|P_{s}u(x)-P_tu(x)|<\frac{\varepsilon}{2}$. As a consequence, in particular, $\sup_{x \not \in K(t)}|P_su(x)|<\varepsilon$. Now notice that $\bigcup_{t \in [0,T]}I(t) \supset [0,T]$, where the latter is a compact subset of $\R$, thus $\{I(t)\}_{t \in [0,T]}$ is an open covering of $[0,T]$. Hence, we can extract a finite number of such open sets $I(t_1),\dots,I(t_N)$ such that $[0,T] \subset \bigcup_{i=1}^{N}I(t_i)$. Let $K=\bigcup_{i=1}^{N}K(t_i)$, which is a compact set since it is a finite union of compact sets. Let $s \in [0,T]$. Then there exists $i=1,\dots,N$ such that $s \in I(t_i)$ and then
\begin{equation*}
\sup_{x \not \in K}|P_su(x)| \le \sup_{x \not \in K(t_i)}|P_su(x)|<\varepsilon.
\end{equation*}
We have shown that, for all $s \in [0,T]$,
\begin{equation*}
\sup_{x \not \in K}|P_su(x)|<\varepsilon.
\end{equation*}
Hence, by \eqref{eq:altrapq} we get
\begin{equation*}
\sup_{x \not \in K}|q(t,x)| \le \int_0^t \overline{\nu}_\phi(t-s)u^\phi(s)\sup_{x \not \in K}|P_su(x)|\, ds<\varepsilon\int_0^t \overline{\nu}_\phi(t-s)u^\phi(s)\, ds=\varepsilon. 
\end{equation*}
Since $\varepsilon>0$ and $T>0$ are arbitrary, we proved that $\lim_{x \to \infty}q(t,x)=0$ locally uniformly with respect to $t \ge 0$. Now we need to prove that $q(t,\cdot) \in {\sf Dom}(G)$ and that \eqref{eq:controlG} holds.} Define
\begin{equation*}
	G_s=\frac{P_s-\mathds{I}}{s}, \ s>0,
\end{equation*}
where $\mathds{I}$ is the identity operator on $C_0(E)$.  Setting $Q:\R_0^+ \to C_0(E)$ as $Q(t)=q(t,\cdot)$ for $t \in \R_0^+$, we have
\begin{equation*}
	Q(t)=\int_0^t P_\tau u\bar{\nu}_\phi(t-\tau)u^\phi(\tau)d\tau.
\end{equation*}
Thus,
\begin{align}\label{eq:Gstildeq}
	G_sQ(t)=G_s \int_0^t P_\tau u\, \bar{\nu}_\phi(t-\tau)u^\phi(\tau)d\tau= \, &\int_0^t G_sP_\tau u \, \bar{\nu}_\phi(t-\tau)u^\phi(\tau)d\tau\notag \\ = \, &\int_0^t P_\tau G_s u \, \bar{\nu}_\phi(t-\tau)u^\phi(\tau)d\tau,
\end{align}
where we used \cite[Proposition 1.1.6]{arendt2011vector} and the fact that $G_s$ and $P_\tau$ commute on ${\sf Dom}(G)$. Now, since $u \in {\sf Dom}(G)$, we have {(see definition \ref{def:semigroupaction})}
\begin{equation*}
	G_su=\frac{1}{s}\int_0^sP_wGudw
\end{equation*}
and in particular $\Norm{G_su}{C(E)} \le \Norm{Gu}{C(E)}$. Furthermore, since $P_\tau$ is contractive, we have
\begin{equation*}
	\Norm{P_\tau G_s u \, \bar{\nu}_\phi(t-\tau)u^\phi(\tau)}{C(E)} \le \bar{\nu}_\phi(t-\tau)u^\phi(\tau)\Norm{Gu}{C(E)},
\end{equation*}
where the right-hand side is integrable in $[0,t]$ and independent of $s$. Hence, taking the limit as $s \to 0$ in \eqref{eq:Gstildeq}, using the dominated convergence theorem for Bochner integrals \cite[Theorem 1.1.8]{arendt2011vector} and the fact that $P_\tau$ is continuous, we get
\begin{equation*}
	\lim_{s \to 0}G_sQ(t)=\int_0^t P_\tau G u \, \bar{\nu}_\phi(t-\tau)u^\phi(\tau)d\tau,
\end{equation*}
i.e. $q(t,\cdot) \in {\sf Dom}(G)$ for all $t \ge 0$ and
\begin{equation*}
	Gq(t,\cdot)=\int_0^t P_\tau G u(\cdot) \, \bar{\nu}_\phi(t-\tau)u^\phi(\tau)d\tau.
\end{equation*}
Finally, again by Bochner theorem, the contractivity of $P_\tau$ and the fact that $(\bar{\nu}_\phi,u^\phi)$ is a Sonine pair, we have
\begin{equation*}
	\Norm{Gq(t,\cdot)}{C(E)} \le \Norm{Gu}{C(E)}.
\end{equation*}
\end{proof}
Since $q(t, \cdot) \in{\sf Dom} (G)$, one has $Gq(t, \cdot) \in C_0 \left( E \right)$. Hence, we get the following corollary, which guarantees that Items (ii) of Theorem \ref{thm:integrability2} and (ii) of Theorem \ref{thm:Lapinside2} hold.
\begin{coro}
 Under the assumptions of Theorem \ref{thm:main1} we have, for all $(t,x) \in \R^+ \times E$
 \begin{equation*}
     \int_0^t \Norm{P_sq(t,\cdot)-q(t,\cdot)}{C(E)}\, \nu_\phi(ds) \le \Norm{Gu}{C(E)}\mathcal{J}_\phi(t)
 \end{equation*}
 {where $\mathcal{J}_\phi(t) = \int_0^t s \nu_\phi (d s)$} and, for all $\lambda>0$ and $x \in E$,
 \begin{equation*}
     \int_0^{+\infty}e^{-\lambda t}\int_0^t \Norm{P_sq(t,\cdot)-q(t,\cdot)}{C(E)}\, \nu_\phi(ds)\, dt<\infty.
 \end{equation*}
\end{coro}
\begin{proof}
Since $q(t,\cdot) \in {\sf Dom}(G) \subset C_0(E)$, we have
\begin{equation*}
    P_sq(t,x)-q(t,x)= \int_0^s GP_\tau q(t,x)\,d\tau= \int_0^s P_\tau G q(t,x)\, d\tau.
\end{equation*}
Taking the absolute value, using again the fact that $P_\tau$ is non-expansive, positivity preserving and sub-Markov,
\begin{multline}\label{eq:Gest1}
    |P_sq(t,x)-q(t,x)| \le \int_0^s |P_\tau G q(t,x)|\, d\tau \le \int_0^s P_\tau |G q(t,x)|\, d\tau 
    \le \Norm{Gq(t,\cdot)}{C(E)}s \le \Norm{Gu}{C(E)}s.  
\end{multline}
Hence, taking the supremum and then integrating,
 \begin{equation*}
     \int_0^t \Norm{P_sq(t,\cdot)-q(t,\cdot)}{C(E)}\, \nu_\phi(ds) \le \Norm{Gu}{C(E)}\mathcal{J}_\phi(t)
 \end{equation*}
 and it follows by \eqref{Jphi}, that $\int_0^{+\infty}e^{-\lambda t}\mathcal{J}_\phi(t)\, dt<\infty$ for all $\lambda>0$.
\end{proof}


Next, we need to show that Item (iii) of Theorem \ref{thm:Lapinside2} holds. Next proposition follows once we observe that $|q(t,x)| \le \Norm{u}{C(E)}$ for all $t \ge 0$ and $x \in E$.
\begin{prop}
\label{prop55}
Under the hypotheses of Theorem~\ref{thm:main1} {the following inequality} holds, for all $\lambda>0$ and $x \in E$
\begin{equation*}
    \int_0^{+\infty}e^{-\lambda t}\overline{\nu}_\phi(t)q(t,x)dt \le \frac{\phi(\lambda)}{\lambda}\Norm{u}{C(E)}.
\end{equation*}
\end{prop}
\subsection{Proof of Theorem~\ref{thm:main1}}
\label{secproof}
\begin{proof}
	First, notice that, by the discussion in Subsection \ref{secregq}, we know that we are under the hypotheses of Theorems \ref{thm:integrability2} and \ref{thm:Lapinside2}. Hence, observe that by Theorem~\ref{thm:integrability2}, $\phi(\partial_t-G)q(t,x)$ is well-defined for all $x \in E$ and $t>0$. Furthermore, by Theorem~\ref{thm:Lapinside2}, we know that for all $x \in E$ the function $\phi(\partial_t-G)q(\cdot,x)$ is Laplace transformable for any $\lambda>0$. 
		{Consider that} $q(0,x)=u(x)$ by definition. Next, observe that, for $\lambda>0$, {the Laplace transform of the right-hand side of the first equation in \eqref{eq:main1} is given by}
	\begin{equation*}
		\int_0^{\infty}e^{-\lambda t}\overline{\nu}_\phi(t)q(0,x)dt=\frac{\phi(\lambda)}{\lambda}q(0,x),
	\end{equation*}
    {where $q(0,x)$ does not play any role in the determination of the aforementioned Laplace transform. Hence, we need to prove that for all $x \in E$}
	\begin{equation*}
		\int_0^{\infty}e^{-\lambda t}\phi(\partial_t-G)q(t,x)dt=\frac{\phi(\lambda)}{\lambda}q(0,x){=\int_0^{\infty}e^{-\lambda t}\overline{\nu}_\phi(t)q(0,x)dt},
	\end{equation*}
    {so that the injectivity of the Laplace transform then yields the first equation for fixed $x \in E$ and almost all $t>0$.} 
For this, first observe that by Theorem~\ref{thm:Lapinside2} we have
	\begin{equation*}
		\int_0^{\infty}e^{-\lambda t}\phi(\partial_t-G)q(t,x)dt=-\int_0^{\infty}\int_0^\infty e^{-\lambda t}(P_sq(t-s,x)\mathds{1}_{[0,t]}(s)-q(t,x))dt\, \nu_\phi(ds).
	\end{equation*}
	Set $Q$ as in the proof of Proposition~\ref{prop:regular2}. Then we have
	\begin{align*}
		\int_0^\infty e^{-\lambda t}P_sQ(t-s)\mathds{1}_{[0,t]}(s)dt&=P_s\int_0^\infty e^{-\lambda t}Q(t-s)\mathds{1}_{[s,+\infty)}(t)dt\\
		&=P_s\int_s^\infty e^{-\lambda t}Q(t-s)dt\\
		&=e^{-\lambda s}P_s\int_0^\infty e^{-\lambda t}Q(t)dt:=e^{-\lambda s}P_s\widetilde{Q}(\lambda),
	\end{align*}
where {the first equality is due to the fact that $P_s$ does not depend on $t$ and that $s \in [0,t]$ is equivalent to $t \in [s,\infty)$ and} where all the integrals are Bochner integrals on $C_0(E)$. Hence, in particular, the equality holds for all $x \in E$ and then we get
\begin{align*}
	\int_0^\infty e^{-\lambda t}P_sq(t-s,x)\mathds{1}_{[0,t]}(s)dt=e^{-\lambda s}P_s\widetilde{q}(\lambda,x),
\end{align*}
where $\widetilde{q}(\lambda,x)=\int_0^\infty e^{-\lambda t}q(t,x)dt$. This implies
\begin{equation}\label{eq:preLaplace}
	\int_0^{\infty}e^{-\lambda t}\phi(\partial_t-G)q(t,x)dt=-\int_0^{\infty} \left(e^{-\lambda s}P_s\widetilde{q}(\lambda,x)-\widetilde{q}(\lambda,x)\right) \nu_\phi(ds).
\end{equation}
Now we need to evaluate $\widetilde{q}(\lambda,x)$. First, we use Fubini's theorem to write
\begin{equation*}
	\widetilde{q}(\lambda,x)=\E^{(x,0)}\left[\int_0^{+\infty}e^{-\lambda t}u(X(t))\, dt\right].
\end{equation*}
Then, we recall that \textcolor{blue}{$S$} performs at most countably many jumps that are a.s. dense in $[0, +\infty)$ and it increases continuously on a set with Lebesgue measure zero (see \cite[Chapter 4, Section 21]{ken1999levy}) and also that $H(t)=S(y-)$ for any $t \in (S(y-),S(y)]$, so that 
\begin{align*}
	\widetilde{q}(\lambda,x)&=\E^{(x,0)}\left[\sum_{y \ge 0}\int_{S(y-)}^{S(y)}e^{-\lambda t}u(M(S(y-)))dt\right]\\
	&=\E^{(x,0)}\left[\sum_{y \ge 0}\frac{e^{-\lambda S(y-)}-e^{-\lambda S(y)}}{\lambda}u(M(S(y-)))\right]\\
	&=\frac{1}{\lambda}\E^{(x,0)}\left[\sum_{y \ge 0}e^{-\lambda S(y-)}(1-e^{-\lambda (S(y)-S(y-))})u(M(S(y-)))\right]\\
\end{align*}
Now observe that all the previous argument can be applied to $u_{\pm}(x)= \max\{0,\pm u(x)\}$, so that
\begin{align}\label{eq:presubtract}
&\frac{1}{\lambda}\E^{(x,0)}\left[\sum_{y \ge 0}e^{-\lambda S(y-)}(1-e^{-\lambda (S(y)-S(y-))})u_{\pm}(M(S(y-)))\right] \notag \\
=	\, &\int_0^{+\infty}e^{-\lambda t}\E^x[u_{\pm}(X(t))]dt<\infty.
\end{align}
Consider the $C_{\rm b}(\R_0^+)$-valued stochastic processes: $Z_{\pm}:=\{Z_{\pm}(y), y \ge 0\}$ defined, for all $y,z \ge 0$, as
\begin{equation*}
	Z_{\pm}(y)(z)=u_{\pm}(M(S(y-)))e^{-\lambda S(y-)}(1-e^{-\lambda z}).
\end{equation*}
Notice that $Z_{\pm}$ is left-continuous, i.e. for all $y>0$ and $y_n \uparrow y$ it holds
\begin{equation*}
\Norm{Z_{\pm}(y_n)-Z_{\pm}(y)}{C(E)}
\le |u_{\pm}(M(S(y-)))e^{-\lambda S(y-)}-u_{\pm}(M(S(y_n-)))e^{-\lambda S(y_n-)}| \to 0
\end{equation*}
as $n \to \infty$, hence it is a predictable process
Furthermore, observe that $\mathfrak{p}=\{\mathfrak{p}(y), \ y \ge 0\}$ with $\mathfrak{p}(y)=S(y)-S(y-)$ is a Poisson point process of intensity $\nu_\phi$. Hence, by the compensation formula (see \cite[Page $7$]{bertoin1996levy}), we get
\begin{align*}
	\frac{1}{\lambda}&\E^{(x,0)}\left[\sum_{y \ge 0}e^{-\lambda S(y-)}(1-e^{-\lambda (S(y)-S(y-))})u_{\pm}(M(S(y-)))\right]\\
	&=\frac{1}{\lambda}\E^{(x,0)}\left[\int_0^{+\infty}\left(\int_0^{+\infty}(1-e^{-\lambda w})\nu_\phi(dw)\right)u_{\pm}(M(S(y-)))e^{-\lambda S(y-)}dy\right]\\
	&=\frac{\phi(\lambda)}{\lambda}\E^{(x,0)}\left[\int_0^{+\infty}u_{\pm}(M(S(y-)))e^{-\lambda S(y-)}dy\right]\\
	&=\frac{\phi(\lambda)}{\lambda}\int_0^{+\infty}\E^{(x,0)}\left[u_{\pm}(M(S(y-)))e^{-\lambda S(y-)}\right]dy \notag \\ =& \frac{\phi(\lambda)}{\lambda}\int_0^{+\infty}\E^{(x,0)}\left[u_{\pm}(M(S(y)))e^{-\lambda S(y)}\right]dy,
\end{align*}
where in the last equality we used the fact that $S$ is stochastically continuous and independent of $M$. Since the left-hand side of the previous equality is finite, so it is the right-hand side. Hence, subtracting term to term and {using the fact that $u(x)=u_+(x)-u_-(x)$ according to the definition of $u_{\pm}$}, we get 
\begin{align*}
	\widetilde{q}(\lambda,x)=\frac{1}{\lambda}&\E^{(x,0)}\left[\sum_{y \ge 0}e^{-\lambda S(y-)}(1-e^{-\lambda (S(y)-S(y-))})u(M(S(y-))) \right]\\
	&=\frac{\phi(\lambda)}{\lambda}\int_0^{+\infty}\E^{(x,0)}\left[u(M(S(y)))e^{-\lambda S(y)}\right]dy\\
	&=\frac{\phi(\lambda)}{\lambda}\int_0^{+\infty}\int_0^{+\infty}e^{-\lambda \tau}\E^{(x,0)}\left[u(M(\tau))\right]g_S(\tau;y)d\tau\, dy,
\end{align*}
where $g_S(\tau;y)d\tau$ is the law of $S(y)$ under $\bP^{(x,0)}$ and we used the independence of $M$ and $S$. Substituting this equality into \eqref{eq:preLaplace} we get
\begin{align}
	\int_0^{\infty}e^{-\lambda t}&\phi(\partial_t-G)q(t,x)dt \nonumber \\
	&=-\frac{\phi(\lambda)}{\lambda}\int_0^{\infty} \left[e^{-\lambda s}P_s\left(\int_0^{+\infty}\int_0^{+\infty}e^{-\lambda \tau}\E^{(\cdot,0)}\left[u(M(\tau))\right]g_S(\tau;y)d\tau\, dy\right)(x)\right.\nonumber\\
	&-\left.\int_0^{+\infty}\int_0^{+\infty}e^{-\lambda \tau}\E^{(x,0)}\left[u(M(\tau))\right]g_S(\tau;y)d\tau\, dy\right] \nu_\phi(ds). \nonumber
\end{align}
Now observe that
\begin{align*}
	\int_0^{+\infty}& \int_0^{+\infty}e^{-\lambda \tau}\Norm{\E^{(\cdot,0)}[u(M(\tau))]}{C(E)}g_S(\tau;y)d\tau\, dy\\
	&\le \Norm{u}{C(E)}\int_0^{+\infty}\int_0^{+\infty}e^{-\lambda \tau}g_S(\tau; y)d\tau\, dy \\
	&=\Norm{u}{C(E)}\int_0^{+\infty}e^{-y\phi(\lambda)}dy=\frac{\Norm{u}{C(E)}}{\phi(\lambda)}<\infty,
\end{align*}
hence, by \cite[Proposition 1.1.6]{arendt2011vector} and recalling that Bochner integrability in $C_0(E)$ implies Lebesgue integrability for fixed $x \in E$, we have
\begin{align}
	\int_0^{\infty}e^{-\lambda t}&\phi(\partial_t-G)q(t,x)dt \nonumber \\
	&=-\frac{\phi(\lambda)}{\lambda}\int_0^{\infty} \int_0^{+\infty}\int_0^{+\infty}\left(e^{-\lambda (s+\tau)}P_s\E^{(\cdot,0)}\left[u(M(\tau))\right](x)\right.\nonumber\\
	&\qquad -\left.e^{-\lambda \tau}\E^{(x,0)}\left[u(M(\tau))\right]\right) g_S(\tau;y)d\tau \, dy \nu_\phi(ds).\label{eq:preLaplace2}
\end{align}
Next, we want to use Fubini theorem. To do this, we first observe that $\E^{(\cdot,0)}[u(M(\tau))] \in {\sf Dom}(G)$. Furthermore, the semigroup $\{\widetilde{P}_s\}_{s \ge 0}$ where $\widetilde{P}_s=e^{-\lambda s}P_s$ admits $G-\lambda I$ as generator, with domain ${\sf Dom}(G)$. Hence,
\begin{align*}
	\int_0^{\infty} &\int_0^{+\infty}\int_0^{+\infty}\left|e^{-\lambda (s+\tau)}P_s\E^{(\cdot,0)}\left[u(M(\tau))\right](x)-
	e^{-\lambda \tau}\E^{(x,0)}\left[u(M(\tau))\right]\right| g_S(\tau;y)d\tau\, dy \nu_\phi(ds)\\
	&=\int_0^{\infty} \int_0^{+\infty}\int_0^{+\infty}e^{-\lambda \tau}\left|\int_0^{s}\widetilde{P}_w (G-\lambda I)\E^{(\cdot,0)}\left[u(M(\tau))\right](x)dw\right| g_S(\tau;y)d\tau\, dy \nu_\phi(ds)\\
	&\le \int_0^{\infty} \int_0^{+\infty}\int_0^{+\infty}\int_0^s e^{-\lambda \tau} \left| \widetilde{P}_w (G-\lambda I)\E^{(\cdot,0)}\left[u(M(\tau))\right](x)\right| dw\, g_S(\tau;y)d\tau \, dy \, \nu_\phi(ds)\\
	&\le \int_0^{\infty} \int_0^{+\infty}\int_0^{+\infty}\int_0^s e^{-\lambda \tau}\Norm{\widetilde{P}_w (G-\lambda I)\E^{(\cdot,0)}\left[u(M(\tau))\right]}{C(E)}dw\, g_S(\tau;y)d\tau\, dy\, \nu_\phi(ds)\\
	&\le \int_0^{\infty} \int_0^{+\infty}\int_0^{+\infty}e^{-\lambda \tau}\Norm{(G-\lambda I)\E^{(\cdot,0)}\left[u(M(\tau))\right]}{C(E)}\left(\int_0^{s}e^{-\lambda w}dw\right) g_S(\tau;y)d\tau\, dy\, \nu_\phi(ds)\\
	&\le \frac{\Norm{Gu}{C(E)}+\lambda \Norm{u}{C(E)}}{\lambda}\int_0^{+\infty}\int_0^{+\infty} {\int_0^{+\infty}}e^{-\lambda \tau}(1-e^{-\lambda s}) g_S(\tau;y)d\tau\, dy\, \nu_\phi(ds)\\
	&=\frac{\Norm{Gu}{C(E)}+\lambda \Norm{u}{C(E)}}{\lambda}<\infty.
\end{align*}
Hence, by Fubini theorem, changing the order of the integrals in \eqref{eq:preLaplace2}, we have
\begin{align}
	\int_0^{\infty}e^{-\lambda t}&\phi(\partial_t-G)q(t,x)dt \nonumber \\
	&=-\frac{\phi(\lambda)}{\lambda}\int_0^{\infty} \int_0^{+\infty}\left(e^{-\lambda (s+\tau)}P_s\E^{(\cdot,0)}\left[u(M(\tau))\right](x)\right.\nonumber\\
	&\qquad -\left.e^{-\lambda \tau}\E^{(x,0)}\left[u(M(\tau))\right]\right) \left(\int_0^{+\infty} g_S(\tau;y)dy\right) \, d\tau \nu_\phi(ds)\nonumber\\
	&=-\frac{\phi(\lambda)}{\lambda}\int_0^{\infty} \int_0^{+\infty}\left(e^{-\lambda (s+\tau)}P_s\E^{(\cdot,0)}\left[u(M(\tau))\right](x)\right.\nonumber\\
	&\qquad
	 -\left.e^{-\lambda \tau}\E^{(x,0)}\left[u(M(\tau))\right]\right)u^\phi(\tau) \, d\tau \nu_\phi(ds),\nonumber
	 \end{align}
where we used the relation $u^\phi(\tau)=\int_0^{+\infty}g_S(\tau;y)dy$. Again, since $\E^{(\cdot,0)}[u(M(\tau))] \in {\sf Dom}(G)$, we can write
\begin{align}
	\int_0^{\infty}e^{-\lambda t}&\phi(\partial_t-G)q(t,x)dt \nonumber \\
	&=-\frac{\phi(\lambda)}{\lambda}\int_0^{\infty} \int_0^{+\infty}\int_0^s e^{-\lambda (\tau+w)}(G-\lambda)P_w\E^{(\cdot,0)}\left[u(M(\tau))\right](x)u^\phi(\tau) dw \, d\tau \nu_\phi(ds).\nonumber
\end{align}
Next, we use again Fubini theorem and we get{, by observing that $s > w$,}
\begin{align}
	\int_0^{\infty}e^{-\lambda t}&\phi(\partial_t-G)q(t,x)dt \nonumber \\
	&=-\frac{\phi(\lambda)}{\lambda}\int_0^{\infty} \int_0^{+\infty}\int_w^{\infty} e^{-\lambda (\tau+w)}(G-\lambda)P_w\E^{(\cdot,0)}\left[u(M(\tau))\right](x)u^\phi(\tau) \nu_\phi(ds) \, dw \, d\tau\nonumber\\
	&=-\frac{\phi(\lambda)}{\lambda}\int_0^{\infty} \int_0^{+\infty} e^{-\lambda (\tau+w)}(G-\lambda)P_w\E^{(\cdot,0)}\left[u(M(\tau))\right](x)u^\phi(\tau) \overline{\nu}_\phi(w) \, dw \, d\tau\nonumber\\
	&=-\frac{\phi(\lambda)}{\lambda}\int_0^{\infty} \int_0^{+\infty} e^{-\lambda (\tau+w)}(G-\lambda)P_{w+\tau}u(x)\, u^\phi(\tau) \overline{\nu}_\phi(w) \, dw \, d\tau.\label{eq:preLaplace3}
\end{align}
Now we set $w+\tau=v$ in the inner integral so that \eqref{eq:preLaplace3} gives us {(again, remember that $\tau \in [0,v]$)}
\begin{align}
	\int_0^{\infty}e^{-\lambda t}&\phi(\partial_t-G)q(t,x)dt \nonumber \\
	&=-\frac{\phi(\lambda)}{\lambda}\int_0^{\infty} \int_\tau^{+\infty} e^{-\lambda v}(G-\lambda)P_{v}u(x)\, u^\phi(\tau) \overline{\nu}_\phi(v-\tau) \, dv \, d\tau\nonumber \\
	&=-\frac{\phi(\lambda)}{\lambda}\int_0^{\infty}  e^{-\lambda v}(G-\lambda)P_{v}u(x)\, \left(\int_0^v u^\phi(\tau) \overline{\nu}_\phi(v-\tau) \, d\tau\right) \, dv\nonumber \\
	&=-\frac{\phi(\lambda)}{\lambda}\int_0^{\infty}  e^{-\lambda v}(G-\lambda)P_{v}u(x) dv\nonumber \\
	&=\frac{\phi(\lambda)}{\lambda}\int_0^{\infty}  e^{-\lambda v}P_{v}(\lambda-G)u(x) dv\nonumber\\ 
	&=\frac{\phi(\lambda)}{\lambda}u(x),
\end{align}
where we used the fact that $u^\phi$ and $\overline{\nu}_\phi$ constitute a Sonine pair {as a consequence of \eqref{eq:densityH}}, $P_v$ and $G$ commute on ${\sf Dom}(G)$ and 
\begin{equation*}
\int_0^{\infty}  e^{-\lambda v}P_{v} dv=(\lambda I-G)^{-1}.
\end{equation*}
Taking the inverse Laplace transform we get that, for all $x \in E$, there exists $I_x \subset \R^+$, with Lebesgue measure zero, such that for $t \not \in I_x$ it holds
\begin{equation*}
\phi(\partial_t-G)q(t,x)=\overline{\nu}_\phi(t)u(x).
\end{equation*}
Now we need to show that $I_x={\varnothing}$ for all $x \in E$. To this purpose, first denote by $\mathcal{R}=\overline{\{{S}(t), \ t \ge 0\}}$ the topological closure of the range of \textcolor{blue}{$S$}. Since $H$ can only jump when $t \in \mathcal{R}$ and it is isolated on the left,
\begin{equation}\label{eq:contH}
    \bP^{(x,0)}(H(t)\not = H(t+)) \le \bP^{(x,0)}(t \in \mathcal{R})=0
\end{equation}
by \cite[Proposition 1.9]{bertoin1999subordinators}, since $\phi \in \mathfrak{SB}_0$. Once this has been established, notice that one can rewrite, given that $H$ and $M$ are independent, 
\begin{multline*}
q(t,x)=\E^{(x,0)}[u(M(H(t))]=\E^{(x,0)}\left[\E^{(x,0)}\left[u(M(H(t)) \mid H(t)\right]\right]\\
=\E^{(x,0)}\left[P_{H(t)}u(x)\right]=\int_0^t P_su(x)\overline{\nu}_\phi(t-s)u^\phi(s)\,ds.
\end{multline*}
Notice, however, that for any two fixed $x_0,x \in E$ it holds
\begin{equation}\label{eq:altrapq}
    q(t,x)=\int_0^t P_su(x)\overline{\nu}_\phi(t-s)u^\phi(s)\,ds=\E^{(x_0,0)}\left[P_{H(t)}u(x)\right].
\end{equation}
Fix $t \in \R_0^+$ and let $t_n \to t$. Then it holds
\begin{equation*}
\Norm{q(t_n,\cdot)-q(t,\cdot)}{C(E)} \le \E^{(x_0,0)}\left[\Norm{P_{H(t_n)}u-P_{H(t)}u}{C(E)}\right].
\end{equation*}
Notice that $\lim_{n \to \infty}\Norm{P_{H(t_n)}u-P_{H(t)}u}{C(E)}=0$ almost surely by \eqref{eq:contH} and the strong continuity of $(P_t)_{t \ge 0}$. Furthermore, $\Norm{P_{H(t_n)}u-P_{H(t)}u}{C(E)} \le 2\Norm{u}{C(E)}$. Hence, by the dominated convergence theorem,
\begin{equation}\label{eq:uniformcont}
\lim_{n \to \infty}\Norm{q(t_n,\cdot)-q(t,\cdot)}{C(E)}=0.
\end{equation}
In particular, this shows that $q$ is continuous in both variables. Next, fix $t \in \R^+$ and consider a sequence $t_n \downarrow t$. Then we have for all $x \in E$
    \begin{align*}
        |\phi(\partial_t-G)q(t_n,x)&-\phi(\partial_t-G)q(t,x)|\\
        & \le \int_0^{t}|P_sq(t-s,x)-q(t,x)-P_sq(t_n-s,x))+q(t_n,x)|\nu_\phi(ds)\\
        &\qquad +\int_{t}^{t_n}|P_sq(t_n-s,x)-q(t_n,x)|\nu_\phi(ds)\\
        &\qquad +|q(t,x)|(\overline{\nu}_\phi(t)-\overline{\nu}_\phi(t_n))+|q(t,x)-q(t_n,x)|\overline{\nu}_\phi(t_n).
    \end{align*}
Taking the supremum over $x \in E$, this leads to
    \begin{align*}
        &\Norm{\phi(\partial_t-G)q(t_n,\cdot)-\phi(\partial_t-G)q(t,\cdot)}{C(E)}\\
        &\le \int_0^{t}\Norm{P_sq(t-s,\cdot)-q(t,\cdot)-P_sq(t_n-s,\cdot))+q(t_n,\cdot)}{C(E)}\nu_\phi(ds)\\
        &\qquad +\int_{t}^{t_n}\Norm{P_sq(t_n-s,\cdot)-q(t_n,\cdot)}{C(E)}\nu_\phi(ds)\\
        &\qquad +\Norm{q(t,\cdot)}{C(E)}(\overline{\nu}_\phi(t)-\overline{\nu}_\phi(t_n))+\Norm{q(t,\cdot)-q(t_n,\cdot)}{C(E)}\overline{\nu}_\phi(t_n)\\
        &:=I_n+J_n+\Norm{q(t,\cdot)}{C(E)}(\overline{\nu}_\phi(t)-\overline{\nu}_\phi(t_n))+\Norm{q(t,\cdot)-q(t_n,\cdot)}{C(E)}\overline{\nu}_\phi(t_n).
    \end{align*}
Concerning $I_n$, first notice that
\begin{align*}
&\Norm{P_sq(t-s,\cdot)-q(t,\cdot)-P_sq(t_n-s,\cdot))+q(t_n,\cdot)}{C(E)} \\
&\quad \le \Norm{P_s(q(t-s,\cdot)-q(t_n-s,\cdot)\textcolor{blue}{)}}{C(E)}+\Norm{q(t,\cdot)-q(t_n,\cdot)}{C(E)}\\
&\quad \le \Norm{q(t-s,\cdot)-q(t_n-s,\cdot)}{C(E)}+\Norm{q(t,\cdot)-q(t_n,\cdot)}{C(E)},
\end{align*}
since $P_s$ is non-expansive. Hence, taking the limit as $n \to \infty$ and using \eqref{eq:uniformcont} we have
\begin{align*}
\lim_{n \to \infty}\Norm{P_sq(t-s,\cdot)-q(t,\cdot)-P_sq(t_n-s,\cdot))+q(t_n,\cdot)}{C(E)}=0.
\end{align*}
Furthermore,
\begin{align*}
&\Norm{P_sq(t-s,\cdot)-q(t,\cdot)-P_sq(t_n-s,\cdot)+q(t_n,\cdot)}{C(E)} \\
&\quad \le \Norm{P_sq(t-s,\cdot)-P_sq(t,\cdot)}{C(E)}+\Norm{P_sq(t,\cdot)-q(t,\cdot)}{C(E)}\\
&+\Norm{P_sq(t_n-s,\cdot)-P_sq(t_n,\cdot)}{C(E)}+\Norm{P_sq(t_n,\cdot)-q(t_n,\cdot)}{C(E)}\\
&\le 2C\left(\Norm{Gu}{C(E)}+\Norm{u}{C(E)}\right)\int_0^su^\phi(t-\tau)d\tau+2\Norm{Gu}{C(E)}s,
\end{align*}
where we used \eqref{eq:ACest1}, \eqref{eq:Gest1} and the fact that $u^\phi(t_n-\tau) \le u^\phi(t-\tau)$ since $u^\phi$ is non-increasing. The right-hand side of the previous inequality is integrable against $\nu_\phi(ds)$, hence we can use the dominated convergence theorem to achieve
\begin{equation*}
    \lim_{n \to \infty}I_n=0.
\end{equation*}
Concerning $J_n$ instead, we have
\begin{align*}
    &\Norm{P_sq(t_n-s,\cdot)-q(t_n,\cdot)}{C(E)} \\
    &\quad \le \Norm{P_sq(t_n-s,\cdot)-P_s q(t_n,\cdot)}{C(E)}+\Norm{P_sq(t_n,\cdot)-q(t_n,\cdot)}{C(E)}\\
    &\quad \le C\left(\Norm{Gu}{C(E)}+\Norm{u}{C(E)}\right)\int_0^su^\phi(t^n-\tau)d\tau+\Norm{Gu}{C(E)}s,
\end{align*}
where we used again \eqref{eq:ACest1} and \eqref{eq:Gest1}. Hence
\begin{align*}
    J_n &\le C\left(\Norm{Gu}{C(E)}+\Norm{u}{C(E)}\right)\int_{t}^{t_n}\int_0^su^\phi(t^n-\tau)d\tau\nu_\phi(ds)+\Norm{Gu}{C(E)}\int_t^{t_n}s\nu_\phi(ds)\\
    &=C\left(\Norm{Gu}{C(E)}+\Norm{u}{C(E)}\right)J_n^1+\Norm{Gu}{C(E)}(\mathcal{J}_\phi(t_n)-\mathcal{J}_\phi(t)).
\end{align*}
By \eqref{Jphi}, it turns out that $\mathcal{J}_\phi$ is right-continuous. Furthermore
\begin{align*}
    J_n^1&=\int_{t}^{t_n}\int_0^tu^\phi(t^n-\tau)d\tau\nu_\phi(ds)+\int_{t}^{t_n}\int_t^su^\phi(t^n-\tau)d\tau\nu_\phi(ds)\\
    &=\left(\overline{\nu}_\phi(t)-\overline{\nu}_\phi(t_n)\right)(U^\phi(t_n)-U^\phi(t_n-t))\\
    &\qquad+\int_{t}^{t_n}u^\phi(t_n-\tau)(\overline{\nu}_\phi(\tau)-\overline{\nu}_\phi(t_n))\, d\tau\\
    &=\left(\overline{\nu}_\phi(t)-\overline{\nu}_\phi(t_n)\right)(U^\phi(t_n)-U^\phi(t_n-t))\\
    &\qquad +\int_{t}^{t_n}u^\phi(t_n-\tau)\overline{\nu}_\phi(\tau)\, d\tau-\overline{\nu}_\phi(t_n)U^\phi(t_n-t)\\
    &\le \left(\overline{\nu}_\phi(t)-\overline{\nu}_\phi(t_n)\right)(U^\phi(t_n)-U^\phi(t_n-t))\\
    &\qquad +\overline{\nu}_\phi(t)U^\phi(t_n-t)-\overline{\nu}_\phi(t_n)U^\phi(t_n-t)\\
    &=\left(\overline{\nu}_\phi(t)-\overline{\nu}_\phi(t_n)\right)U^\phi(t_n),
\end{align*}
hence, since $\overline{\nu}_\phi$ is also right-continuous, $\lim_{n \to \infty}J_n^1=0$, thus $\lim_{n \to \infty}J_n=0$. This proves that
\begin{equation}\label{eq:rightcontG}
    \lim_{n \to \infty}\Norm{\phi(\partial_t-G)q(t_n,\cdot)-\phi(\partial_t-G)q(t,\cdot)}{C(E)}=0.
\end{equation}
Now fix $x \in E$ and assume, by contradiction, that $I_x \not = {\varnothing}$. Let $t \in I_x$ and consider any interval of the form $[t,t+\delta]$ for some $\delta>0$. Let $I_x^\delta=I_x \cap [t,t+\delta]$ and observe that $|I_x^\delta|=0$, hence $|[t,t+\delta]\setminus I_x^\delta|=\delta=|[t,t+\delta]|$. Since the Lebesgue measure has full topological support (i.e., any non-empty open set has positive measure), then $[t,t+\delta]\setminus I_x^\delta$ is dense in $[t,t+\delta]$ and there exists a sequence $t_n \downarrow t$ such that $t_n \not \in I_x$ for all $n \in \N$. For such a sequence
\begin{equation*}
\phi(\partial_t-G)q(t_n,x)=\overline{\nu}_\phi(t_n)q(t_n,x).
\end{equation*}
However, taking the limit as $n \to \infty$, using \eqref{eq:uniformcont}, \eqref{eq:rightcontG} and the fact that $\overline{\nu}_\phi$ is right-continuous, we get
\begin{equation*}
\phi(\partial_t-G)q(t,x)=\overline{\nu}_\phi(t)q(t,x),
\end{equation*}
that is absurd since $t \in I_x$. Hence $I_x = {\varnothing}$ and, since $x \in E$ is arbitrary, this shows that
\begin{equation}
\phi(\partial_t-G)q(t,x)=\overline{\nu}_\phi(t)q(t,x), \quad \forall (t,x) \in \R^+ \times E.
\end{equation}
{Combining the latter with the fact that $\lim\limits_{x \to \infty}q(t,x)=0$ locally uniformly with respect to $t \ge 0$ as shown in Proposition \ref{prop:regular2},} we finally have that $q$ is solution of \eqref{eq:main1}.

It remains to show that this is the unique solution. Indeed, let $q_1$ be another solution of \eqref{eq:main1}. Then $w=q-q_1$ solves \eqref{eq:main1} with $u \equiv 0$. Now let $T>0$ and assume, by contradiction, that $\varepsilon=\sup_{(t,x) \in [0,T] \times E}|w(t,x)|>0$. Let $K \subset E$ be such that
\begin{equation*}
    \sup_{x \not \in K}|w(t,x)|<\varepsilon, \ \forall t \in [0,T].
\end{equation*}
Then, since $w$ is continuous,
\begin{equation*}
    \varepsilon=\max_{(t,x) \in [0,T]\times K}|w(t,x)|.
\end{equation*}
Let $(t_\star,x_\star) \in [0,T]\times K$ be such that $|w(t_\star,x_\star)|=\varepsilon$. If $w(t_\star,x_\star)=\varepsilon$, then $w(t_\star,x_\star) \ge w(t,x)$ for all $(t,x) \in (0,t_\star] \times E$ and $w(t_\star,x_\star)=\varepsilon>0=w(0,x)$ for all $x \in E$. Hence, by Proposition \ref{prop:maxprin} we have
\begin{equation*}
    0=-\phi(\partial_t-G)w(t_\star,x_\star)<0,
\end{equation*}
which is a contradiction. The same holds if $w(t_\star,x_\star)=-\varepsilon$. Hence $w \equiv 0$ and this ends the proof.
\end{proof}
\subsection{Another uniqueness criterion}
Before moving to option pricing, we give a further result concerning the uniqueness of the solution to our non-local problem.
\begin{prop}\label{prop:unique2}
    Let $q:\R_0^+ \times E \to \R$ be a function such that $q(t,\cdot) \in {\sf Dom}(G)$ for all $t \ge 0$ and $t \in \R_0^+ \mapsto q(t,\cdot) \in C_0(E)$ is measurable. Assume further that there exists $\lambda_\star>0$ such that:
    \begin{itemize}
    \item The function
    \begin{equation*}
    \widetilde{q}(\lambda,x)=\int_0^{\infty}e^{-\lambda t}q(t,x)dt
    \end{equation*}
    is well-defined for all $x \in E$ and $\lambda>\lambda_\star$
    \item It holds
    \begin{equation*}
        \Norm{\widetilde{q}}{C(E)} \le \int_0^{\infty}e^{-\lambda_\star t}\Norm{q(t,\cdot)}{C(E)}dt<\infty.
    \end{equation*}
    \item For all $x \in E$ there exists $\mathcal{I}_{x} \subset \R_0^+$ such that $|\mathcal{I}_x|=0$ and $\phi(\partial_t-G)q(t,x)$ is well-defined for all $t \in \R_0^+ \setminus \mathcal{I}_x$. 
    \item For all $x \in E$ and $\lambda>\lambda_\star$ it holds
    \begin{equation}\label{eq:exchange}
        \int_0^\infty e^{-\lambda t}\phi(\partial_t-G)q(t,x)dt=-\int_0^\infty\int_0^\infty e^{-\lambda t}(P_sq(t-s,x)\mathds{1}_{(0,t]}(s)-q(t,x))dt\, \nu_\phi(ds)
    \end{equation}
    \item For all $x \in E$ and $t \in \R_0^+\setminus \mathcal{I}_x$ it holds
    \begin{equation}\label{eq:eqspecial}
    \phi(\partial_t-G)q(t,x)=\overline{\nu}_\phi(t)q(0,x), \quad \mbox{ for almost all }t>0.
    \end{equation}
    \end{itemize}
    Then, for almost all $t \in \R_0^+$ it holds
    \begin{equation*}
        \Norm{q(t,\cdot)-\E^{(\cdot,0)}[q(0,X(t))]}{C(E)}=0.
    \end{equation*}
\end{prop}
\begin{proof}
    Since $q(0,\cdot) \in {\sf Dom(G)}$, arguing as in the proof of Theorem \ref{thm:main1}, we know that the function $q_\star(t,x)=\E^{(x,0)}[q(0,X(t))]$ satisfies all the assumptions in the statement. As a consequence, also the function $q-q_\star$ satisfies all the assumptions in the statement. Hence, without loss of generality, we can assume that $q(0,\cdot) \equiv 0$. Integrating \eqref{eq:eqspecial} against $e^{-\lambda t}$ for $\lambda>\lambda_\star$ and using \eqref{eq:exchange} we get
    \begin{equation}\label{eq:equality0}
        \int_0^\infty\int_0^\infty e^{-\lambda t}(P_sq(t-s,x_\star)\mathds{1}_{(0,t]}(s)-q(t,x))dt\, \nu_\phi(ds)=0.
    \end{equation}
    Next, notice that
    \begin{equation}\label{eq:Pbeforeex}
        \int_0^\infty e^{-\lambda t}P_sq(t-s,x) \mathds{1}_{(0,t]}(s)dt=e^{-\lambda s}\int_0^\infty P_s\left(e^{-\lambda (t-s)}q(t-s,\cdot) \mathds{1}_{(0,t]}(s)\right)(x)dt.
    \end{equation}
    Now observe that
    \begin{align}\label{eq:Bochnerint}
    \begin{split}
        \int_0^\infty \Norm{e^{-\lambda (t-s)}q(t-s,\cdot) \mathds{1}_{(0,t]}(s)}{C(E)}dt&=\int_0^\infty \Norm{e^{-\lambda (t-s)}q(t-s,\cdot) \mathds{1}_{\R_0^+}(t-s)}{C(E)}dt\\
        &=\int_{-s}^\infty \Norm{e^{-\lambda \tau}q(\tau,\cdot) \mathds{1}_{\R_0^+}(\tau)}{C(E)}d\tau\\
        &=\int_{0}^\infty e^{-\lambda \tau}\Norm{q(\tau,\cdot)}{C(E)}d\tau<\infty.
        \end{split}
    \end{align}
    Observing that, for fixed $s>0$, also $t \in \R_0^+ \mapsto e^{-\lambda (t-s)}q(t-s,\cdot) \mathds{1}_{(0,t]}(s) \in C_0(E)$ is measurable, we know by \eqref{eq:Bochnerint} and Bochner's theorem (see \cite[Theorem 1.1.4]{arendt2011vector}) that the latter is Bochner-integrable and then, by \cite[Proposition 1.1.6]{arendt2011vector} and \eqref{eq:Pbeforeex}, we have
    \begin{multline}
        \int_0^\infty e^{-\lambda t}P_sq(t-s,x) \mathds{1}_{(0,t]}(s)dt=e^{-\lambda s}P_s\left(\int_0^\infty e^{-\lambda (t-s)}q(t-s,\cdot) \mathds{1}_{(0,t]}(s)(x)dt\right)=e^{-\lambda s}P_s\widetilde{q}(x)
    \end{multline}
    and then \eqref{eq:equality0} becomes
    \begin{equation}\label{eq:equality1}
        \int_0^\infty \left(e^{-\lambda s}P_s\widetilde{q}(x)-\widetilde{q}(x)\right) \nu_\phi(ds)=0.
    \end{equation}
    Now let $\widetilde{P}_s=e^{-\lambda s}P_s$ and notice that this is still a Feller semigroup with generator $G-\lambda I$, where $I$ is the identity operator. In particular, by \cite[Proposition 1.1.7]{arendt2011vector}, $\widetilde{q} \in {\sf Dom}(G) \subset {\sf Dom}(G-\lambda I)$. In particular, this implies that
    \begin{equation*}
        e^{-\lambda s}P_s\widetilde{q}(x)-\widetilde{q}(x)=\int_0^s \widetilde{P}_\tau(G-\lambda I)\widetilde{q}(x)\, d\tau.
    \end{equation*}
    Taking the norm we get
    \begin{align*}
        \Norm{e^{-\lambda s}P_s\widetilde{q}-\widetilde{q}}{C(E)} &\le \int_0^s \Norm{\widetilde{P}_\tau(G-\lambda I)\widetilde{q}}{C(E)}\, d\tau\le \frac{\Norm{(G-\lambda I)\widetilde{q}}{C(E)}}{\lambda}(1-e^{-\lambda s}).
    \end{align*}
    Integrating the right-hand side against $\nu_\phi(ds)$ we get
    \begin{align*}
        \int_0^\infty \Norm{e^{-\lambda s}P_s\widetilde{q}-\widetilde{q}}{C(E)}\, \nu_\phi(ds) \le \frac{\Norm{(G-\lambda I)\widetilde{q}}{C(E)}}{\lambda}\phi(\lambda).
    \end{align*}
    Thus, we can rewrite \eqref{eq:equality1} in terms of Bochner integrals as
    \begin{equation}\label{eq:equality2}
        \int_0^\infty \left(e^{-\lambda s}P_s\widetilde{q}-\widetilde{q}\right) \nu_\phi(ds)=0.
    \end{equation}
    Now consider the Feller semigroup $(P_s^\phi)_{s \ge 0}$ acting on $f \in C_0(E)$ as
    \begin{equation*}
    \widetilde{P}_s^\phi f= \int_0^{+\infty} \widetilde{P}_\tau f g_S(\tau;s)\, d\tau
    \end{equation*}
    and denote by $\left(-\phi(-(G-\lambda I)), {\sf Dom}(\phi(-(G-\lambda I)))\right)$ its generator. By Phillips formula (see \cite[Theorem \textcolor{blue}{12}.6]{schilling2009bernstein}) we know that ${\sf Dom}(G-\lambda I) \subset {\sf Dom}(\phi(-(G-\lambda I)))$ and on $f \in {\sf Dom}(G-\lambda I)$ it holds
    \begin{equation*}
        -\phi(-(G-\lambda I))f=-\int_0^\infty \left(e^{-\lambda_\star s}P_sf-f\right) \nu_\phi(ds).
    \end{equation*}
    Since $\widetilde{q} \in {\sf Dom}(G-\lambda I)$, we get by \eqref{eq:equality2}
    \begin{equation}\label{eq:equality3}
        -\phi(-(G-\lambda I))\widetilde{q}=0.
    \end{equation}
    Finally, notice that for all $f \in {\sf Dom}(G-\lambda I)$ it holds
    \begin{equation*}
    \Norm{\widetilde{P}_s^\phi f}{C(E)}\le  \int_0^{+\infty} \Norm{\widetilde{P}_\tau f}{C(E)} g_S(\tau;s)\, d\tau \le e^{-s\phi(\lambda)}\Norm{f}{C(E)},
    \end{equation*}
    where $\phi(\lambda)>0$. Hence, by the Hille-Yosida Theorem (see \cite[Theorems 3.3.2 and 3.3.4]{arendt2011vector}) we know that $0$ belongs to the resolvent set of $-\phi(\lambda I-G)$ and thus the operator is invertible. As a consequence, \eqref{eq:equality3} implies
    \begin{equation*}
        \widetilde{q}(\lambda,\cdot) \equiv 0.
    \end{equation*}
    By injectivity of the Laplace transform (see \cite[Theorem 1.7.3]{arendt2011vector}), this ends the proof.
\end{proof}
\begin{rmk}
    Notice that the function $q$ in Proposition \ref{prop:unique2} is not a solution of \eqref{eq:q} in the sense of Definition \ref{def:sol}. Indeed, on the one hand we are not requiring that $q$ is a continuous function of $(t,x)$, on the other hand $q$ solves \eqref{eq:q} not for all $(t,x) \in \R_0^+ \times E$ but only on a subset $\widetilde{E} \subset \R_0^+ \times E$ such that for all $x \in E$ the set $\widetilde{I}_x=\{t \in \R_0^+: \ (t,x) \not \in \widetilde{E}\}$ satisfies $|\widetilde{I}_x|=0$. With an abuse of notation, we could refer to such a function $q$ as an \textit{almost-everywhere (a.e.) solution} of \eqref{eq:q}. From this point of view, Proposition \ref{prop:unique2} tells us that the solution provided in {Theorem} \ref{thm:main1} is also the unique a.e. solution satisfying some additional conditions on the Laplace transform. Finally, notice that we are not requiring that $\lim_{x \to \infty}q(t,x)=0$ locally uniformly with respect to $t \ge 0$. 
\end{rmk}
\section{{Option pricing with the undershooting of the $\alpha$-stable subordinator}}
\label{secfinance}
{In this section we deal with the model with dependent returns and trade durations (DRD) introduced in \cite{torricelli}. In \cite[Theorem 6.2]{torricelli}, an option pricing formula is obtained for such a model using Fourier methods. Here, we shall instead use the results of the previous sections, in particular Theorem \ref{thm:main1}, to derive a suitable final problem for a non-local pseudo-differential equation, in the same spirit as for the local Black and Scholes differential equation.} 
\subsection{The price process}
To describe the price process underlying our {option-pricing method}, we shall make use of the time-changed processes introduced in Section \ref{sec:undershooting}. Let $B=\{B(t), \ t \ge 0\}$ be a $1$-dimensional standard Brownian motion and consider the process $M(t)=e^{B(t)}$. Consider $S$ as in Section \ref{sec:undershooting} and assume $(B,S)$ is a canonical Feller process, on the canonical {filtered} probability space $(\Omega, \mathcal{F}_\infty, \{\cF_t\}_{t \ge 0}, \bP^{(x,v)})$. By Lemma \ref{markovadd} we know that $(B^\phi,S)$ is a Markov additive Feller process with respect to its augmented natural filtration $\cG:=\{\mathcal{G}_t\}_{t \ge 0}$. Now let $X_{\rm e}(t)=B^\phi(L(t)-)=B(H(t))$ and consider the age process 
\begin{equation}
\label{ageprocess}
    \gamma(t):=t-S(0)-H(t).
\end{equation}
Then, $(X_{\rm e},\gamma)$ is a c\`agl\`ad process. Furthermore, notice that by Phillip's theorem, since $C^\infty_c(\R)$ is contained in the domain of the generator of $(B,S)$, we know that $C^\infty_c(\R)$ is a subset of the domain of the generator of $(B^\phi,S)$. Hence, we can use the Courr\'ege-Waldenfels theorem (see \cite[Theorem 2.21]{bottcher2013levy}) to establish the jump-diffusion form of the generator of $(B^\phi,S)$. Thus, we have the following result.
\begin{thm}
For $t \ge 0$ set $\cH_t=\cG_{L(t)-}$. Then the process $(X_{\rm e},\gamma)$ is a time-homogeneous Markov process with respect to the filtration $\cH:=\{\cH_t\}_{t \ge 0}$.
\end{thm}
\begin{proof}
    The proof is exactly the same of the ones of \cite[Theorems 3.2 and 4.1]{meerschaert2014semi}, with the only exception that since $(B^\phi,{S})$ is not a canonical Feller process in our setting, one has to modify the translation operator accordingly. Precisely, in our case we just need to set $\theta_t \omega(s)=(\omega_1(s),\omega_2(s+t))$ for all $t,s \ge 0$. Indeed, for $s,t \ge 0$ and $\omega \in \Omega$
    \begin{multline*}
        (B^\phi(\theta_t\omega,s),S(\theta_t\omega,s))=((B(\theta_t\omega),S(\theta_t\omega,s)),S(\theta_t\omega,s))\\
        =
        (\omega_1(\omega_2(t+s)),\omega_2(t+s))=(B^\phi(\omega,t+s),S(\omega,t+s)).
    \end{multline*}
\end{proof}
Now we set
\begin{equation}
\label{Xdefinition}
    X(t)=e^{X_{\rm e}(t)}=M(H(t))=M^\phi(L(t)-).
\end{equation}
$(X,\gamma)$ is still a time-homogeneous Markov process with respect to $\cH$. Now define $\cN_t^0=\sigma((X_{\rm e}(s),\gamma(s)), s \le t)$ and $\cN_t$ its completion with respect to $\bP^{(x,v)}$. $\cN:=\{\cN_t^0\}_{t \ge 0}$ is the natural filtration of $(X,\gamma)$, as the latter is related to $(X_{\rm e},\gamma)$ by means of a homeomorphism. Furthermore, since $(X,\gamma)$ is $\cH$-Markov, then $\cN_t \subseteq \cH_t$ for all $t \ge 0$ and both $(X_{\rm e},\gamma)$ and $(X,\gamma)$ are $\cN$-Markov and time-homogeneous. Since we are going to work with c\`agl\`ad processes, let us introduce now, for any $T>0$, the space $D_-[0,T]$ of c\`agl\`ad functions $f:[0,T] \to \R$ and $\Lambda_-[0,T]$ its subspace of non-decreasing and non-negative functions. As usual in this context, we need to introduce a new (family of) probability measures under which the process $X$ is actually a $\cN$-martingale. From now on, whenever it is not ambiguous, we shall denote by $\restr{X}{[0,T]}$ ,$\restr{X_{\rm e}}{[0,T]}$ ,$\restr{\gamma}{[0,T]}$ and $\restr{\cN}{[0,T]}$ the restriction of the considered quantities on the interval $[0,T]$. We need the following lemma.
\begin{lem}
\label{lemmamart}
For any {$m \in \mathbb{R}$} and $x,v \in \R$, the process
\begin{equation*}
    Z^m(t):=e^{mX_{\rm e}(t)-\frac{m^2 H(t)}{2}}
\end{equation*}
is a $\cN$-martingale under $\bP^{(x,v)}$.
\end{lem}
\begin{proof}
$Z^m$ is $\cN$-adapted by definition. To show that  $Z^m(t) \in L^1(\Omega;\bP^{(x,v)})$ for all $t \ge 0$, just notice that
\begin{equation*}
\E^{(x,v)}[Z^m(t)]=\E^{(x,v)}\left[\E^{(x,v)}\left[\cE(mB)(H(t)) \mid H(t)\right]\right],
\end{equation*}
where
\begin{equation*}
\cE(mB)(\tau)=e^{mB(\tau)-\frac{m^2\tau}{2}}
\end{equation*}
is the Doleans-Dade exponential of $mB$. Since $B$ and $H$ are independent, we have
\begin{equation*}
\E^{(x,v)}\left[\cE(mB)(H(t)) \mid H(t)\right]=F_1(H(t)),
\end{equation*}
where
\begin{equation*}
F_1(\tau):=\E^{(x,v)}\left[\cE(mB)(\tau)\right]= {\E^{(x,0)}\left[\cE(mB)(\tau)\right] = e^x}.
\end{equation*}
Hence
\begin{equation*}
\E^{(x,v)}[Z^m(t)] = {e^x \int_{0}^{t-v} \, \overline{\nu}_\phi(t-v-s)u^\phi(s) \, ds = e^x<\infty}.
\end{equation*}
To show the martingale property, fix $0 \le s \le t$ and consider two Borel sets $A_1,A_2 \subset D_-[0,s]$. Consider the event
\begin{equation*}
    A_\star:=\{\restr{X_{\rm e}}{[0,s]} \in A_1\} \cap \{\restr{\gamma}{[0,s]} \in A_2\}. 
\end{equation*}
Since $B$ and $H$ are independent and $\gamma(t)=t-v-H(t)$ for all $t \ge 0$, it holds
\begin{equation}\label{eq:DDexpo1}
    \E^{(x,v)}[Z^m(t)\mathds{1}_{A_\star}]=\E^{(x,v)}[\mathds{1}_{A_2}(\gamma)F_2(\restr{H}{[0,t]})],
\end{equation}
where, for all $g \in \Lambda_-[0,t]$, we set
\begin{align*}
F_2(g)=\E^{(x,v)}\left[\mathds{1}_{A_1}(\restr{B(g(\cdot))}{[0,s]})\cE(mB)(g(t))\right] \ \mbox{ and } \ ,
\end{align*}
and we notice that
\begin{equation*}
F_2(\restr{H}{[0,t]})=\E^{(x,v)}[Z^m(t)\mathds{1}_{A_1}(X_{\rm e}) \mid H(\tau), \ \tau \ge 0].
\end{equation*}
Notice in particular that $\cE(mB)$ is the Doleans-Dade exponential of $mB$. Since $A_1 \subset D_-[0,s]$, the event $\{\restr{B(g(\cdot))}{[0,s]} \in A_1\}$ belongs to $\sigma(B(w), \ w \le g(s))$. Furthermore, once we recall that $\cE(mB)$ is a martingale and $g(s) \le g(t)$, we have
\begin{equation*}
F_2(g)=\E^{(x,v)}\left[\mathds{1}_{A_1}(B(g(\cdot)))\cE(mB)(g(s))\right]=:\widetilde{F}(\restr{g}{[0,s]}).
\end{equation*}
However, we have that, still by independence between $B$ and $H$
\begin{equation*}
\widetilde{F}_2(\restr{H}{[0,s]})=\E^{(x,v)}[Z^m(s)\mathds{1}_{A_1}(X_{\rm e}) \mid H(\tau), \ \tau \ge 0].
\end{equation*}
Substituting it back into \eqref{eq:DDexpo1}, we get
\begin{equation}\label{eq:DDexpo2}
    \E^{(x,v)}[Z^m(t)\mathds{1}_{A_\star}]=\E^{(x,v)}[\mathds{1}_{A_2}(\gamma)\widetilde{F}_2(\restr{H}{[0,s]})]=\E^{(x,v)}[Z^m(s)\mathds{1}_{A_\star}].
\end{equation}
Since the events of the form $A_\star$ constitute a $\pi$-system generating $\cN_s$, this ends the proof.
\end{proof}
From now on we set $m=-1/2$ and we write $Z:=Z^{-\frac{1}{2}}$. {Therefore we have $$
Z(t) = e^{-\frac{X_{\rm e}(t)}{2}-\frac{H(t)}{8}}
$$} For any $T>0$ define the probability measure
\begin{equation*}
	\cN_T \ni A \mapsto \widetilde{\bP}_T^{(x,v)}(A)=e^{\frac{x}{2}}\E^{(x,v)}\left[Z(T)\mathds{1}_{A}\right],
\end{equation*}
i.e., in terms of the Radon-Nykodim derivative:
\begin{equation*}
    \frac{d \widetilde{\bP}_T^{(x,v)}}{d \restr{\bP^{(x,v)}}{\cN_T}}=Z(T).
\end{equation*}
Since $Z(T)$ is an $\cN$-martingale, we have that $\widetilde{\bP}_t^{(x,v)}=\restr{\widetilde{\bP}_T^{(x,v)}}{\cN_t}$ for all $T \ge 0$ and $t \in [0,T]$. This guarantees that we can use the Kolmogorov extension theorem to construct a probability measure $\widetilde{\bP}^{(x,v)}$ on $(\Omega, \cN_\infty)$, where $\cN_\infty:=\sigma\left(\bigcup_{t \ge 0}\cN_t\right)$ such that $\restr{\widetilde{\bP}^{(x,v)}}{\cN_t}=\widetilde{\bP}^{(x,v)}_t$ for all $t \ge 0$. The method of construction of such a measure is classical and we refer to \cite[Page 192]{karatzas2014brownian}. Furthermore, notice that since the Radon-Nykodim derivative is strictly positive, the measure $\widetilde{\bP}^{(x,v)}$ is equivalent to $\bP^{(x,v)}$ when restricted on $\cN_t$ for any $t \ge 0$.

Now let us show that $X$ defined in \eqref{Xdefinition} is a{n} $\cN$-martingale under such a measure.
\begin{prop}
 	For any $(x,v) \in \R^2$, the process $X$ is a{n} $\cN$-martingale under $\widetilde{\bP}^{(x,v)}$.
\end{prop}
\begin{proof}
By definition $X$ is $\cN$-adapted. Now we show that for all $t \in [0,T]$ it is true that $X(t) \in L^1(\Omega;\widetilde{\bP}^{(x,v)})$. To do this, first notice that
\begin{align*}
		\widetilde{\E}^{(x,v)}[X(t)]&=e^{\frac{x}{2}}\E^{(x,v)}\left[e^{X_{\rm e}(t)-\frac{X_{\rm e}(T)}{2}-\frac{H(T)}{8}}\right]\\
		&=\E^{(x,v)}\left[\E^{(x,v)}\left[e^{\frac{x}{2}+X_{\rm e}(t)-\frac{X_{\rm e}(T)}{2}-\frac{H(T)}{8}}\mid H(t),H(T)\right]\right].
\end{align*}
Since $H(t) \le H(T)$ are independent of the Brownian motion $B$, if we set, for $s_1 \le s_2$,
\begin{equation*}
F_1(s_1,s_2)=\E^{(x,v)}\left[e^{B(s_1)-\frac{B(s_2)}{2}+\frac{x}{2}-\frac{s_2}{8}}\right]
\end{equation*}
we get
\begin{align*}
	\widetilde{\E}^{(x,v)}[X(t)]&=\E^{(x,v)}\left[F_1(H(t),H(T))\right].
\end{align*}
However, by Girsanov's theorem (see \cite[Theorem 5.1]{karatzas2014brownian}), we know that $\{B(s)+\frac{s}{2}, \ s \in [0,s_2]\}$ is a Brownian motion under the measure 
\begin{equation*}
A \in \cF_{s_2} \mapsto \E^{(x,v)}\left[\mathds{1}_Ae^{-\frac{B(s_2)}{2}+\frac{x}{2}-\frac{s_2}{8}}\right],
\end{equation*}
hence in particular $e^{B(s)}=e^{B(s)+\frac{s}{2}-\frac{s}{2}}$ is the Doleans-Dade exponential of $B(s)+\frac{s}{2}$ and thus a martingale under the same measure. As a consequence, we have that
\begin{equation*}
F_1(s_1,s_2)=\E^{(x,v)}\left[e^{B(s_1)-\frac{B(s_2)}{2}+\frac{x}{2}-\frac{s_2}{8}}\right]=e^{x}
\end{equation*}
and then
\begin{align*}
	\widetilde{\E}^{(x,v)}[X(t)]&=\E^{(x,v)}\left[F_1(H(t),H(T))\right]=e^x<\infty.
\end{align*}

	Now we need to prove the martingale property. To do this, fix {$s$ and $t$ such that} $0 \le s \le t$ and let $A_1,A_2 \in D_-[0,s]$ be Borel sets. Define the event
 \begin{equation*}
     A_{\star}:=\{\restr{X_{\rm e}}{[0,s]} \in A_1\} \cap \{\restr{\gamma}{[0,s]} \in A_2\}.
 \end{equation*}
Then we have, for any $T>t$,
\begin{align}
&\widetilde{\E}^{(x,v)}\left[X(t)\mathds{1}_{A_\star}\right]=\E^{(x,v)}\left[\mathds{1}_{A_1}(\restr{X_{\rm e}}{[0,s]})\mathds{1}_{A_2}(\restr{\gamma}{[0,s]})e^{\frac{x}{2}+X_{\rm e}(t)-\frac{X_{\rm e}(T)}{2}-\frac{H(T)}{8}}\right]\nonumber\\
&=\E^{(x,v)}\left[\mathds{1}_{A_2}(\restr{\gamma}{[0,s]})F_2(\restr{H}{[0,T]} {)}\right]\label{eq:premartingale},
	\end{align}
where for any $g \in \Lambda_-[0,T]$ we set
\begin{equation*}
    F_2(g):=\E^{(x,v)}\left[\mathds{1}_{A_1}(\restr{B(g(\cdot))}{[0,s]})e^{\frac{x}{2}+B(g(t))-\frac{B(g(T))}{2}-\frac{g(T)}{8}}\right]
\end{equation*}
and we notice that
\begin{equation*}
    F_2(\restr{H}{[0,T]})=\E^{(x,v)}\left[\mathds{1}_{A_1}(\restr{X_{\rm e}}{[0,s]})e^{\frac{x}{2}+X_{\rm e}(t)-\frac{X_{\rm e}(T)}{2}-\frac{H(T)}{8}}\mid \ \restr{H}{[0,T]}\right].
\end{equation*}
Since $A_1 \subset D_-[0,s]$, the event $\{\restr{B(g(\cdot))}{[0,s]} \in A_1\}$ belongs to $\sigma(B(w), \ w \le g(s))$. Therefore we use again Girsanov's theorem and use the martingale property of Brownian (conditionally on the independent $H|_{[0,T]}$) to get now
\begin{equation*}
    F_2(g)=\E^{(x,v)}\left[\mathds{1}_{A_1}(\restr{B(g(\cdot))}{[0,s]})e^{\frac{x}{2}+B(g(s))-\frac{B(g(T))}{2}-\frac{g(T)}{8}}\right].
\end{equation*}
In particular, we have that
\begin{equation*}
    F_2(\restr{H}{[0,T]})=\E^{(x,v)}\left[\mathds{1}_{A_1}(\restr{X_{\rm e}}{[0,s]})e^{\frac{x}{2}+X_{\rm e}(s)-\frac{X_{\rm e}(T)}{2}-\frac{H(T)}{8}}\mid \ \restr{H}{[0,T]}\right].
\end{equation*}
When we replace the latter equality back into \eqref{eq:premartingale}, we obtain
\begin{align}
\widetilde{\E}^{(x,v)}&\left[X(t)\mathds{1}_{A_\star}\right]=\E^{(x,v)}\left[\mathds{1}_{A_1}(\restr{X_{\rm e}}{[0,s]})\mathds{1}_{A_2}(\restr{\gamma}{[0,s]})e^{\frac{x}{2}+X_{\rm e}(t)-\frac{X_{\rm e}(T)}{2}-\frac{H(T)}{8}}\right]\nonumber\\
&=\E^{(x,v)}\left[\mathds{1}_{A_1}(\restr{X_{\rm e}}{[0,s]})\mathds{1}_{A_2}(\restr{\gamma}{[0,s]})e^{\frac{x}{2}+X_{\rm e}(s)-\frac{X_{\rm e}(T)}{2}-\frac{H(T)}{8}}\right]=\widetilde{\E}^{(x,v)}\left[X(s)\mathds{1}_{A_\star}\right].
\end{align}
Given that events of the form of $A_\star$ are a $\pi$-system generating $\cN_s$, this completes the proof.
\end{proof}
We need, however, to guarantee that the change of measure does not affect the Markov property and the time-homogeneity of the process $(X,\gamma)$. As for the Markov property, this is considered in the following result, while the preservation of the time-homogeneity will be proved later.

\begin{prop}\label{prop:Markov}
	The process $(X,\gamma)$ is $\cN$-Markov under $\widetilde{\bP}^{(x,v)}$.
\end{prop}
\begin{proof}
	Consider any function $f \in L^\infty(\R^2)$ and let $0 \le s \le t$. We have, for any $T>t$,
	\begin{align*}
		\widetilde{\mathds{E}}^{(x,v)}\left[f(X(t),\gamma(t))\mid \cN_s\right]&= { \frac{\mathds{E}^{(x,v)}\left[f(X(t),\gamma(t))e^{-\frac{X_e (T)}{2}-\frac{H(T)}{8}}\mid \cN_s\right]}{\mathds{E}^{(x,v)}\left[e^{-\frac{X_e(T)}{2}-\frac{H(T)}{8}}\mid \cN_s\right]}}\\
        &=\frac{\mathds{E}^{(x,v)}\left[f(X(t),\gamma(t))e^{-\frac{\log(X(T))}{2}-\frac{T-v-\gamma(T)}{8}}\mid \cN_s\right]}{\mathds{E}^{(x,v)}\left[e^{-\frac{\log(X(T))}{2}-\frac{T-v-\gamma(T)}{8}}\mid \cN_s\right]}\\
		&=\frac{\mathds{E}^{(x,v)}\left[f(X(t),\gamma(t))e^{-\frac{\log(X(T))}{2}-\frac{T-v-\gamma(T)}{8}}\mid X(s),\gamma(s)\right]}{\mathds{E}^{(x,v)}\left[e^{-\frac{\log(X(T))}{2}-\frac{T-v-\gamma(T)}{8}}\mid X(s),\gamma(s)\right]}\\
		&=\widetilde{\mathds{E}}^{(x,v)}\left[f(X(t),\gamma(t))\mid X(s),\gamma(s)\right],
	\end{align*}
	where we used the fact that $H(T)=T-v-\gamma(T)$ and $X_{\rm e}(T)=\log(X(T))$ under $\E^{(x,v)}$ and the $\cN$-Markov property of $(X,\gamma)$ with respect to $\mathds{P}^{(x,v)}$.
%
%
%
\end{proof}
Before proving that $(X,\gamma)$ is also \textit{time-homogeneous}, we need to investigate another property. Indeed, notice that both $X$ and $\gamma$ depend on $B$ and $S$ and in particular $X(0)$ and $\gamma(0)$ are {\textit{constrained}} by the values $(x,v)$. 
Observe that under $\mathds{P}^{(x,v)}$, one has that if $\sigma(0)=v$ and thus $L(t)$ is defined for $t \geq v$, as well as $\gamma(t)$. In practice, under $\mathds{P}^{(x,v)}$, $v\leq 0$, the position of process $(X, \gamma)= (X(w), \gamma (w))$, $w \geq 0$, at time $t$, conditionally to $X(s), \gamma(s)$, does note see the choice of $x$ and $v$. Hence, as a first step, we need to show that, under $\widetilde{\E}^{(x,v)}$, the distribution of $(X(t),\gamma(t))$ given $(X(s),\gamma(s))$ is independent of the choice of $(B(0),S(0))$, given that $S(0) \le s$ at least. This is done in the next proposition, in which we also prove that the process $(X,\gamma)$ is time-homogeneous with respect to the measure $\widetilde{\bP}^{(x,v)}$.
\begin{prop}\label{lem:timehomo1}
Let $0 \le s \le t$. Then, for all $x,x^\prime \in \R$ and $v,v^\prime \le s$ it holds
\begin{equation}\label{eq:equalityexp}
	\widetilde{\E}^{(x,v)}\left[f(X(t),\gamma(t))\mid X(s), \ \gamma(s)\right]=\widetilde{\E}^{(x^\prime,v^\prime)}\left[f(X(t),\gamma(t))\mid X(s), \ \gamma(s)\right].
\end{equation}
In particular, one can define, for $0 \le s \le t$, the transition semigroup $\widetilde{Q}_{s,t}$ of $(X,\gamma)$ under $\widetilde{\bP}^{(x,v)}$ acting on bounded measurable functions $f:\R^+ \times \R_0^+ \to \R$ in such a way that
\begin{equation}\label{eq:transsemi1}
\widetilde{Q}_{s,t}f(X(s),\gamma(s))=\widetilde{\E}^{(x,v)}\left[f(X(t),\gamma(t)) \mid X(s), \gamma(s)\right],
\end{equation}
for any $x \in \R$ and $v \le s$. Furthermore, $\widetilde{Q}_{s,t}$ only depends on $t-s$ and then $(X,\gamma)$ is time-homogeneous.
\end{prop}
\begin{proof}
Let $(Q_{t})_{t \ge 0}$ be the transition semigroup under $\bP^{(x,v)}$ of the time-homogeneous process $(X,\gamma)$, i.e. for all bounded measurable functions $f:\R^+ \times \R_0^+ \to \R$ and for all $t \ge 0$, $Q_tf:\R^+ \times \R_0^+ \to \R$ is the bounded measurable function such that for any $s \ge 0$, $x \in \R^+$ and $v \in [0,s]$ 
\begin{equation}\label{eq:semigroupQ}
    Q_tf(X(s),\gamma(s))=\E^{(x,v)}[f(X(t+s),\gamma(t+s)) \mid X(s), \gamma(s)].
\end{equation}
The independence on $(x,v)$ is justified by observing that the transition semigroup of $(X_{\rm e},\gamma)$ is already independent of $(x,v)$ by \cite[Theorems 3.2 and 4.1]{meerschaert2014semi}. 

Notice that \eqref{eq:semigroupQ} holds whenever $f:\R^+ \times \R_0^+ \to \R$ is a measurable function such that
\begin{equation*}
    \E^{(x,v)}\left[|f(X(t),\gamma(t))|\right]<\infty
\end{equation*}
for all $t \ge 0$ and $x,v \in \R$. Now fix $0 \le s \le t$, $x,x^\prime \in \R^+$ and $v,v^\prime \le s$. Then, {as a consequence of Proposition \ref{prop:Markov}, for any $T>t$,} we have
\begin{align}\label{eq:prehomo0}
    \widetilde{\E}^{(x,v)}\left[f(X(t),\gamma(t)) \mid X(s), \gamma(s)\right]=\frac{\E^{(x,v)}\left[f(X(t),\gamma(t))e^{-\frac{\log(X(T))}{2}-\frac{T-\textcolor{blue}{v -}\gamma(T)}{8}} \mid X(s), \gamma(s)\right]}{\E^{(x,v)}\left[e^{-\frac{\log(X(T))}{2}-\frac{T-\textcolor{blue}{v -}\gamma(T)}{8}} \mid X(s), \gamma(s)\right]}.
\end{align}
To deal with the numerator note that, by Lemma \ref{lemmamart}, we can use the martingale property to get, for any $y$, $w$ and $T>t$, that for the Markov semigroup $Q_t$, $t \geq 0$, it is true that
\begin{align}
    Q_{T-t}f(y,w) = f(y,w).
\end{align}
Hence, for the numerator in \eqref{eq:prehomo0}, we have
\begin{align}
\E^{(x,v)}\left[f(X(t),\gamma(t))e^{-\frac{\log(X(T))}{2}-\frac{T-v-\gamma(T)}{8}} \mid X(s), \gamma(s)\right]=\E^{(x,v)}\left[f(X(t),\gamma(t))Q_{T-t}g(X(t),\gamma(t)) \mid X(s), \gamma(s)\right], \label{eq:pretimehomo1}
\end{align}
where 
\begin{equation*}
    g(y,w)=e^{-\frac{\log(y)}{2}-\frac{T{-S(0)}-w}{8}}.
\end{equation*}
Next, let
\begin{equation*}
    h(y,w)=f(y,w)Q_{T-t}g(y,w).
\end{equation*}
Given that $f$ is bounded, we can write 
\begin{align*}
    \E^{(x,v)}\left[|h(X(t),\gamma(t))|\right] &\le \left(\sup_{(y,w) \in \R^+\times \R_0^+}|f(y,w)|\right)\E^{(x,v)}\left[g(X(T),\gamma(T))\right]\\
    &=e^x\left(\sup_{(y,w) \in \R^+\times \R_0^+}|f(y,w)|\right)<\infty.
\end{align*}
Hence \eqref{eq:pretimehomo1} can be rewritten as
\begin{equation}\label{eq:pretimehomo2}
    \E^{(x,v)}\left[f(X(t),\gamma(t))e^{-\frac{\log(X(T))}{2}-\frac{T-v-\gamma(T)}{8}} \mid X(s), \gamma(s)\right]=Q_{t-s}h(X(s),\gamma(s)).
\end{equation}
For the denominator of \eqref{eq:prehomo0}, we have
\begin{align*}
&\E^{(x,v)}\left[e^{-\frac{\log(X(T))}{2}-\frac{T-v-\gamma(T)}{8}} \mid X(s), \gamma(s)\right]=Q_{T-s}g(X(s),\gamma(s)).
\end{align*}
Hence, \eqref{eq:prehomo0} can be rewritten as
\begin{equation}\label{eq:transsemi1.5}
\widetilde{\E}^{(x,v)}\left[f(X(t),\gamma(t)) \mid X(s), \gamma(s)\right]=\frac{Q_{t-s}h(X(s),\gamma(s))}{Q_{T-s}g(X(s),\gamma(s))}.
\end{equation}
The right-hand side is independent of the choice of $(x,v)$, therefore we get \eqref{eq:equalityexp}. Now, let us select $T=t+1$, so that $h(y,w)=f(y,w)Q_{1}g(t,w)$, and let us set
\begin{equation*}
\widetilde{Q}_{s,t}f(y,w)=\frac{Q_{t-s}\left(f(\cdot,\cdot)Q_{1}g(\cdot,\cdot)\right)(y,w)}{Q_{t-s+1}g(y,w)}.
\end{equation*}
Then we see that $\widetilde{Q}_{s,t}$ only depends on $t-s$ and \eqref{eq:transsemi1} holds by virtue of \eqref{eq:transsemi1.5}.
%
%
\end{proof}
With the previous proposition in mind, we can directly define $\widetilde{Q}_{t-s}:=\widetilde{Q}_{s,t}$. Finally, we show the equivalent of the Cameron-Martin formula in this context.
\begin{prop}[{Cameron-Martin formula}]\label{lem:CameronMartin}
For any $(x,v) \in \R^2$ and any $T>0$, the distribution of $\restr{\left(X_{\rm e}+\frac{H}{2},\gamma\right)}{[0,T]}$ under $\widetilde{\bP}^{(x,v)}$ coincides with the one of $\restr{\left(X_{\rm e},\gamma\right)}{[0,T]}$ under $\bP^{(x,v)}$. As a consequence, the distribution of $(X,\gamma)$ under $\widetilde{\bP}^{(x,v)}$ coincides with the one of $\left(Xe^{-\frac{H}{2}},\gamma\right)$ under $\bP^{(x,v)}$.
\end{prop}
\begin{proof}
Fix $T>0$ and consider a bounded and measurable function $f:D_-[0,T] \times \Lambda_-[0,T] \to \R$. Then we have
\begin{align}
	\widetilde{\E}^{(x,v)}&\left[f\left(\restr{\left(X_{\rm e}+\frac{H}{2}\right)}{[0,T]},\restr{\gamma}{[0,T]}\right)\right]\\
 &=\E^{(x,v)}\left[f\left(\restr{\left(X_{\rm e}+\frac{H}{2}\right)}{[0,T]},\restr{\gamma}{[0,T]}\right)e^{-\frac{X_{\rm e}(T)}{2}-\frac{H(T)}{8}}\right]\nonumber \\
	&=\E^{(x,v)}\left[F(\restr{H}{[0,T]})\right]\label{eq:predist},
\end{align} 
 where we set, for any $g \in \Lambda_-[0,T]$,
 \begin{equation*}
     \widetilde{\gamma}(g)(t):=t-v-g(t) \ \mbox{ and } \  F(g)=\E^{(x,v)}\left[f\left(\restr{\left(B(g(\cdot))+\frac{g(\cdot)}{2}\right)}{[0,T]},\restr{\widetilde{\gamma}(g)}{[0,T]}\right)e^{-\frac{B(g(T))}{2}-\frac{g(T)}{8}}\right]
 \end{equation*}
and we notice that
\begin{equation*}
    F(\restr{H}{[0,T]})=\E^{(x,v)}\left[f\left(\restr{\left(X_{\rm e}+\frac{H}{2}\right)}{[0,T]},\restr{\gamma}{[0,T]}\right)e^{-\frac{X_{\rm e}(T)}{2}-\frac{H(T)}{8}} \mid \restr{H}{[0,T]}\right].
\end{equation*}
By Girsanov theorem, however, we know that under the measure
\begin{equation*}
A \in \cF_{g(T)} \mapsto \E^{(x,v)}\left[\mathds{1}_{A}e^{-\frac{B(g(T))}{2}-\frac{g(T)}{8}}\right]
\end{equation*}
the process $t \in [0,g(T)] \mapsto B(t)+\frac{t}{2}$ is a Brownian motion. Hence, the distribution of $\restr{\left(B(g(\cdot))+\frac{g(\cdot)}{2}\right)}{[0,T]}$ under the aforementioned measure coincides with the distribution of $\restr{B(g(\cdot))}{[0,T]}$ under $\bP^{(x,v)}$ and then
\begin{equation*}
    F(g)=\E^{(x,v)}\left[f(\restr{B(g(\cdot))}{[0,T]},\restr{\widetilde{\gamma}(g)}{[0,T]})\right].
\end{equation*}
In particular, we notice that, by independence of $B$ and $H$,
\begin{equation*}
    F(\restr{H}{[0,T]})=\E^{(x,v)}\left[f(\restr{X_{\rm e}}{[0,T]},\restr{\gamma}{[0,T]}) \mid \restr{H}{[0,T]}\right].
\end{equation*}
Substituting this into \eqref{eq:predist} we have
\begin{align}
	\widetilde{\E}^{(x,v)}\left[f\left(\restr{\left(X_{\rm e}+\frac{H}{2}\right)}{[0,T]},\restr{\gamma}{[0,T]}\right)\right]=\E^{(x,v)}\left[f(\restr{X_{\rm e}}{[0,T]},\restr{\gamma}{[0,T]})\right].
\end{align} 
Since $f$ is arbitrary, we have the equality in distribution.

Next, notice that with the same argument as in Proposition \ref{prop:Markov}, one can prove that $\left(X_{\rm e}+\frac{H}{2},\gamma\right)$ is a $\cN$-Markov process under $\widetilde{\bP}^{(x,v)}$. Recalling also that $(X_{\rm e},\gamma)$ is $\cN$-Markov with respect to the measure $\E^{(x,v)}$, we have that since the distribution of $\restr{\left(X_{\rm e}+\frac{H}{2},\gamma\right)}{[0,T]}$ under $\widetilde{\bP}^{(x,v)}$ and $\restr{(X_{\rm e},\gamma)}{[0,T]}$ under $\bP^{(x,v)}$ coincide for all $T>0$, then the distribution of the whole trajectories must coincide. Finally, the distribution of $(X,\gamma)$ under $\widetilde{\bP}^{(x,v)}$ coincides with the one of $\left(Xe^{-\frac{H}{2}},\gamma\right)$ under $\bP^{(x,v)}$ since both of them are expressed as deterministic functions of $\left(X_{\rm e}+\frac{H}{2},\gamma\right)$ and $\left(X_{\rm e},\gamma\right)$.
\end{proof}

\subsection{The coupled non-local Black-Scholes equation}
We can now introduce the option price process. Fix a maturity $T>0$ and a strike price $K>0$. We want to focus on the case of a call option, hence we define the option price process by
\begin{equation*}
	\mathcal{X}(t):=\widetilde{\E}^{(x,v)}\left[({X}(T)-K)_+ \mid \cN_t\right], \qquad t \in [0,T]
\end{equation*}
{where $(x)_+ = \max(x,0)$.}
By Proposition~\ref{prop:Markov}, ${X}(t)$ is given by
\begin{equation*}
	\mathcal{X}(t)=\widetilde{\E}^{(x,v)}\left[(\mathcal{X}(T)-K)_+ \mid X(t), \ \gamma(t)\right]=\widetilde{Q}_{T-t}u(X(t),\gamma(t)),
\end{equation*}
where
\begin{equation*}
    u(y,w)=(y-K)_+.
\end{equation*}
Hence, let us set
\begin{equation*}
q_\star(t,y,w)\coloneqq \widetilde{Q}_{T-t}u(y,w).
\end{equation*}
For a call option with maturity $T>0$, strike price $K>0$, when the interest rate is $0$, the classical Black and Scholes equation is
\begin{align}
    \partial_t q_{\rm BS}(T-t,x) \, = \, - \frac{1}{2} x^2 \partial_x^2 q_{\rm {BS}}(T-t,x), \qquad t \in [0,T), \, x \in \mathbb{R}^+, 
\label{classbsproblem}
\end{align}
under the (final) condition $ q_{\text{BS}}(T-t,x) \, \mid_{t=T} = \, (x-K)_+$.
In practice, denoting
\begin{equation}
    q(t,x)=q_\star(T-t,x,0)=\widetilde{Q}_{t}u(x,0)
    \label{prezzofinbsnonloca}
\end{equation}
and, for $w>0$,
\begin{align}
    q(t,x,w)=q_\star(T-t,x,w)=\widetilde{Q}_{t}u(x,w),
    \label{pricewpos}
\end{align}
in this section we show that the function \eqref{prezzofinbsnonloca} is the solution of a Black-and-Scholes type (non-local) equation containing the coupled operator introduced in Section \ref{secoperat}; moreover, we provide a renewal-type equation that gives the price (for any $w>0$) in \eqref{pricewpos}.

In our analysis we limit ourselves to the case when the process $S$ is an $\alpha$-stable subordinator, that is $\phi(\lambda)=\lambda^{\alpha}$, and we furthermore impose, for tehcnical reasons, the condition $\alpha>\frac{1}{2}$.

\begin{rmk}\label{rem:alphadist}
 The undershoot of an inverse stable subordinator is particularly tractable, in that its marginal distributions are known. Indeed, the law of $H(t)$ coincides with that of $t B$ where $B$ is a Beta$(\beta,1-\ \beta)$ random variable (\cite{meercoupled}, Example 5.5). Based on this, after the usual conditioning argument it can be deduced that the characteristic function $\varphi_{X(t)}(z)$ of $X(t)$ is given by 
\begin{equation*}
    \varphi_{X(t)}(z)=_1\hspace{-.1cm}F_1(\alpha, t/2) , \qquad z \in \R,
\end{equation*}
where $_1F_1$ is the confluent hypergeometric function. The above follows as a particular case of \cite{torricelli}, Theorem 6.2, (6.4), second line.
\end{rmk}

Let us begin with some preliminary considerations. By Remark \ref{rem:alphadist}, denote the density of the undershooting $H_0$ by
\begin{equation}\label{eq:densityHalpha}
{g_{H_0}(s;t)}=\frac{s^{\alpha-1}(t-s)^{-\alpha}\mathds{1}_{(0,t)}(s)}{\Gamma(\alpha)\Gamma(1-\alpha)}.
\end{equation}
A general form of \eqref{eq:densityHalpha} for general subordinators is available in \cite[Lemma 1.10]{bertoin1999subordinators}.
Now we recall that, since when $v=0$ it must hold $X_{\rm e}(0)=B(0)$,
\begin{equation*}
    q(t,x)=\widetilde{\E}^{(\log(x),0)}\left[(X(t)-K)_+\right].
\end{equation*}
By the time-changed Cameron-Martin formula given in Proposition \ref{lem:CameronMartin}, recalling that, under $v=0$, $H$ and $H_0$ coincide, we do have
\begin{equation}\label{eq:qspecial}
    q(t,x)=\E^{(\log(x),0)}\left[(X(t)e^{-\frac{H(t)}{2}}-K)_+\right]=\int_0^{+\infty}q_{\rm BS}(s,x)g_H(s;t)ds,
\end{equation}
where
\begin{align*}
q_{\rm BS}(s,x)&=\E^{(\log(x),0)}\left[\left(B(s)e^{-\frac{s}{2}}-K\right)_+\right]\\
&=x\Phi\left(\frac{2\log(x)-2\log(K)+s}{2\sqrt{s}}\right)-K\Phi\left(\frac{2\log(x)-2\log(K)-s}{2\sqrt{s}}\right)
\end{align*}
and
\begin{equation*}
    \Phi(z):=\frac{1}{\sqrt{2\pi}}\int_{-\infty}^{z}e^{-\frac{x^2}{2}}dx.
\end{equation*}
The function $q_{\rm BS}$ is the solution of the classical Black and Scholes equation \eqref{classbsproblem} with respect to the variable $s=T-t$, i.e., the quantity $q_{\text{BS}}(s,x)$, for fixed $(s,x)$, represents the price, at time zero, of a {plain-vanilla} call option with maturity $s>0$ and the corresponding function satisfies the Black and Scholes equation with the operator $2^{-1}x^2 \partial_x^2$ (with positive sign). The positive sign in front of the space operator is due to the fact that we are taking the derivative with respect to variable $s>0$ that, in this context, is the maturity. Precisely, let, for any $\beta\ge 0$ 
\begin{equation*}
\cC_0(\beta):=\{u \in C(\R^+): \ \exists C>0, \ |u(x)| \le Ce^{\beta|\log(x)|}\}
\end{equation*}
and define
\begin{equation*}
\cC_0=\bigcup_{\beta \ge 0}\cC_0(\beta).
\end{equation*}
Furthermore, we say that a function $f \in \cC_{\rm sol}$ if and only if
\begin{itemize}
\item[$(i)$] $f  \in C(\R_0^+ \times \R^+) \cap C^1(\R^+\times \R^+)$
\item[$(ii)$] $f(t,\cdot) \in C^2(\R^+)$ for all $t>0$
\item[$(iii)$] For all $T>0$ there exist $\beta_T,C_T>0$ such that $|f(t,x)| \le C_Te^{\beta_T|\log(x)|}$ for all $t \in [0,T]$ and $x \in \R^+$.
\end{itemize}
For $u \in \cC_0(\beta)$, we define the quantity
\begin{equation*}
[u]_{\beta}:=\sup_{x \in \R^+}|u(x)|e^{-\beta|\log(x)|}.
\end{equation*}
Finally, let us introduce $\cC_{0, \ {\rm loc}}(\R_0^+)$ as the set of functions $f \in C(\R_0^+ \times \R^+)$ such that for all $T>0$ there exist $\beta_T,C_T>0$ with the property that $|f(t,x)| \le C_Te^{\beta_T|\log(x)|}$ for all $t \in [0,T]$ and $x \in \R^+$. Notice that $\cC_{\rm sol} \subset \cC_{0, \ {\rm loc}}(\R_0^+)$.

Before proceeding to introduce the Black-Scholes equation, let us give a simple lemma, which will be useful in the following.
\begin{lem}\label{lem:momgen}
    It holds
    \begin{align*}
        \E^{(x,v)}[e^{\lambda_1 |B(t)|+\lambda_2B(t)}]&=e^{(\lambda_1+\lambda_2) x+\frac{(\lambda_1+\lambda_2)^2}{2}t}\Phi\left(-\frac{(\lambda_1+\lambda_2) t+x}{\sqrt{t}}\right)\\
        &\qquad +e^{-(\lambda_1-\lambda_2) x+\frac{(\lambda_1-\lambda_2)^2}{2}t}\Phi\left(-\frac{(\lambda_1-\lambda_2) t-x}{\sqrt{t}}\right)
    \end{align*}
\end{lem}
The proof of such a lemma is given in Appendix \ref{app:B}. Now we can introduce the operator $G:=\frac{x^2}{2}\partial_x^2$ and provide some well-known properties of such an operator. 
\begin{lem}\label{lem:BSuniqueness}
	For all $f \in \cC_0$ there exists a unique classical solution $u_f \in \cC_{\rm sol}$ of 
	\begin{equation}\label{eq:BSclass1}
		\begin{cases}
			\partial_t u(t,x)=Gu(t,x), & t>0, \ x \in \R^+ \\
			u(0,x)=f(x), & x>0,
		\end{cases}
	\end{equation}
	i.e., $(\cC_0,\cC_{\rm sol})$ is a uniqueness class for $G$. If we denote by $(P_t)_{t \geq 0}$ the associated semigroup action, then the following properties hold true.
	\begin{enumerate}
		\item For all $f \in \cC_0$ we have that 
		\begin{align}
			P_t f (x) \, = \,  \int_0^{+\infty} f(y) p_{\rm GBM}(t,y;x)dy
		\end{align}
	where {GBM stands for geometric Brownian motion and}
	\begin{equation}\label{heatker}
		p_{\rm GBM}(t,y;x):=x^{\frac{1}{2}}e^{-\frac{t}{8}}\frac{p(t,\log(y)-\log(x))}{y^{\frac{3}{2}}} \ \mbox{ and } \ p(t,w)=\frac{1}{\sqrt{2\pi t}}e^{-\frac{w^2}{2t}}.
	\end{equation}
        As a consequence, the semigroup $(P_t)_{t \ge 0}$ is Markovian and positivity preserving.
	\item Let $\beta \ge 0$ and assume $f \in \cC_0(\beta)$. Then $\partial_x P_tf \in \cC_0(\beta+1)$ and 
	\begin{equation*}
		[\partial_x P_tf(x)]_{\beta+1} \le 4[f]_\beta\exp\left(\beta\frac{t}{2}+\frac{\beta^2 t^2}{2}\right).
	\end{equation*}
	\item For all $f \in \cC_0 \cap \cC$ such that $\partial_x f \in \cC_0$ and $Gf \in \cC_0$ it holds
	\begin{equation*}
		P_tGf=GP_tf.
	\end{equation*}
        \item Fix $\beta_1,\beta_2 \in \R$ and let $f_{\beta_1,\beta_2}(x)=e^{\beta_1|\log(x)|+\beta_2\log(x)}$. Then
            \begin{align*}
                P_tf_{\beta_1,\beta_2}(x)&=e^{(\beta_1+\beta_2) \log(x)+\frac{(\beta_1+\beta_2)(\beta_1+\beta_2-1)}{2}t}\Phi\left(-\left(\beta_1+\beta_2-\frac{1}{2}\right)\sqrt{t}-\frac{\log(x)}{\sqrt{t}}\right)\\
                &\qquad + e^{-(\beta_1-\beta_2)\log(x)+\frac{(\beta_1-\beta_2)(\beta_1-\beta_2+1)}{2}t}\Phi\left(-\left(\beta_1-\beta_2+\frac{1}{2}\right)\sqrt{t}+\frac{\log(x)}{\sqrt{t}}\right).
            \end{align*}
        \item If $f \in \cC_0(\beta)$, then $P_tf \in \cC_0(\beta)$ with
        \begin{equation*}
            [P_tf]_{\beta} \le 2[f]_{\beta}e^{\frac{\beta(\beta+1)}{2}t}.
        \end{equation*}
	\end{enumerate}
\end{lem}
We give, for completeness, the proof of the previous lemma in Appendix \ref{app:C}.

Concerning the function $q_{\rm BS}$, we recall the following properties, that can be verified by direct evaluation.
\begin{lem}\label{lem:qBS}
	For all $t,x>0$ it holds
	\begin{align}
		\partial_x q_{\rm BS}(t,x)&=\Phi\left(\frac{2\log(x/K)+t}{2\sqrt{t}}\right), \label{eq:der1}\\
		\partial^2_x q_{\rm BS}(t,x)&=\sqrt{\frac{K}{2\pi tx^3}}\exp\left(-\frac{4\log^2\left(x/K\right)+t^2}{8t}\right), \label{eq:der2}\\
		\partial_t q_{\rm BS}(t,x)&=G{q}_{\rm BS}(t,x)=\sqrt{\frac{Kx}{8\pi t}}\exp\left(-\frac{4\log^2\left(x/K\right)+t^2}{8t}\right). \label{eq:der3}
	\end{align}
	In particular,
	\begin{align}
            \left|q_{\rm BS}(t,x)\right| &\le e^{|\log(x)|} \label{eq:der3.5}\\
		\left|\partial_x q_{\rm BS}(t,x)\right|&\le 1, \label{eq:der4}\\
		\left|\partial_t q_{\rm BS}(t,x)\right|&=\left|G q_{\rm BS}(t,x)\right| \le \sqrt{\frac{Kx}{8\pi t}}.\label{eq:der5}
	\end{align}
	Furthermore, for all $a>0$ and $x \ge a$ we have
	\begin{equation}
		\left|\partial_x^2 q_{\rm BS}(t,x)\right| \le \sqrt{\frac{K}{8\pi a^3 t}}. \label{eq:der6}
	\end{equation}
        In particular $q_{BS}$ is the unique solution of
        	\begin{equation}\label{eq:BSclass1}
		\begin{cases}
			\partial_t q_{\rm BS}(t,x)=Gq_{\rm BS}(t,x), & t>0, \ x \in \R^+ \\
			q_{\rm BS}(0,x)=(x-K)_+, & x>0,\\
                q_{\rm BS} \in \cC_{0,  {\rm loc}}(\R_0^+).
		\end{cases}
	\end{equation}
\end{lem}
The remainder of this section is devoted to the proof of the following theorem, i.e., the main result of this part. Notice that we focus on the case $\phi(\lambda)=\lambda^{\alpha}$ with $\alpha>\frac{1}{2}$. While it will be clear that some arguments apply to a generic $\Phi$, most of the proof relies on the specific form of $\Phi$, in particular on the self-similarity of $H_0$ under this choice. Furthermore, as it will be clear in the following, the condition $\alpha>\frac{1}{2}$ is dictated by regularity issues.
\begin{thm}\label{thm:main2}
Let $\phi(\lambda)=\lambda^{\alpha}$ for $\alpha \in \left(\frac{1}{2},1\right)$. Let also, for any $h \in \R$, $A \in \cB(\R)$ and $w \ge 0$
\begin{align}\label{eq:K}
\widetilde{\mathcal{K}}(A,h)=\int_0^{+\infty}\mathds{1}_A(s)p(s,h)\frac{\alpha s^{-\alpha-1}}{\Gamma(1-\alpha)}ds \qquad \mathcal{K}(d\tau,dh)=\widetilde{\mathcal{K}}(d\tau,h)dh
\end{align}
and
\begin{equation*}
    \mathcal{K}_w(A)=\begin{cases}
    \displaystyle \frac{\mathcal{K}(A \cap (\R \times [w,+\infty)))}{\mathcal{K}(\R \times [w,+\infty))} & w>0\\[7pt]
    \delta_{(0,0)}(A) & w=0.
    \end{cases}
\end{equation*}
Then
\begin{equation*}
q_\star(t,x,0)=q(T-t,x)
\end{equation*}
and, for $w>0$ and any $v<t-w$,
\begin{multline}\label{eq:possoj}
q_\star(t,x,w)=\left(xe^{-\frac{t-v-w}{2}}-K\right)_+\mathcal{K}_w(x;\R^d \times [w+T-t,\infty))\\
    +\int_{\R \times [w,w+T-t)}\!\!\!\!\!\!\! q(T-t+w-\tau,\log(x)+y+t-w+\tau-v)\mathcal{K}_w(d\tau,dy),
\end{multline}
where $q$ is the unique solution of
\begin{equation}\label{eq:BSUnder}
    \begin{cases}
        (\partial_t-G)^\alpha q(t,x)=\frac{t^{-\alpha}}{\Gamma(1-\alpha)}q(0,x) \\
        q(0,x)=(x-K)_+\\
        q \in \cC_{0, {\rm loc}}(\R_0^+).
    \end{cases}
\end{equation}
\end{thm}
This will be proved using several preliminary results:
\begin{enumerate}
\item First we study the regularity of the function $q$ in \eqref{eq:qspecial};
\item next, we prove that $q$ is solution of \eqref{eq:BSUnder};
\item we then prove that the solution to \eqref{eq:BSUnder} is unique for initial data in a suitable class (to which $(x-K)_+$ belongs);
\item finally, we prove Theorem \ref{thm:main2}.
\end{enumerate}

\begin{rmk}
Notice that \eqref{eq:possoj} is a formula that determines the value of $q_\star$ uniquely for positive sojourn time. Indeed, the only quantity that is involved is the function $q$ which is related to the value of $q_\star$ on renewal states, i.e. for zero sojourn time, which is the unique solution with at most power growth at $0$ and $\infty$ of the fully nonlocal Black-Scholes equation \eqref{eq:BSUnder}. This is also underlined by the fact that \eqref{eq:possoj} can be rewritten as
\begin{multline}\label{eq:possoj2}
q_\star(t,x,w)=\left(xe^{-\frac{t-v-w}{2}}-K\right)_+\mathcal{K}_w(x;\R^d \times [w+T-t,\infty))\\
    +\int_{\R \times [w,w+T-t)}\!\!\!\!\!\!\! q_\star(t-w+\tau,\log(x)+y+t-w+\tau-v,0)\mathcal{K}_w(d\tau,dy),
\end{multline}
that only involves the function $q_\star(t,x,0)$, which is given by the solution of \eqref{eq:BSUnder}.{We remark that $(x-K)_+$ appears as an initial condition in \eqref{eq:BSUnder} but, given that the fair price of our call option at time $t>0$ is $q_\star(t,x,w)$, then 
    \begin{align}
        (x-K)_+=q(0,x)  \, = \, q_\star (T,x,0)
    \end{align}
and thus $(x-K)_+$ is the usual final condition for
$q_\star(\cdot,\cdot,0)$. It follows that, according to our model, the fair price at time $t>0$ of a(n) (intraday) plain vanilla call option (with zero interest rate, strike price $K$ and maturity $T>t$) can be determined by solving the initial value problem \eqref{eq:BSUnder} then changing the variable $t \rightsquigarrow  T-t$ and finally using \eqref{eq:possoj2}.
}
\end{rmk}

\subsubsection{Regularity of $q$}
\label{regstar}
In this section, we provide some regularity results on the function $q$ that will be needed to apply Theorems \ref{thm:integrability2} and \ref{thm:Lapinside2}. Before proceeding, let us recall that for any $f \in L^\infty(\R)$, by  \eqref{eq:densityHalpha}, it holds
\begin{align}\label{eq:selfsim}
    \E^{(x,0)}[f(H(t))]=\int_0^t f(y)\frac{(t-y)^{-\alpha}y^{\alpha-1}}{\Gamma(1-\alpha)\Gamma(\alpha)}\, dy
    =\int_0^1 f(ts)\frac{(1-s)^{-\alpha}s^{\alpha-1}}{\Gamma(1-\alpha)\Gamma(\alpha)}\, ds=\E^{(x,0)}[f(tH(1))],
\end{align}
i.e., $H(t)\overset{d}{=}tH(1)$. We shall use this self-similarity property throughout the proof.

Now, we prove that for all $x>0$ the function $q(\cdot,x)$ belongs to ${\rm AC}(\R_0^+)$ and we provide a suitable upper bound on its a.e. derivative. 

\begin{prop}\label{prop:ACqstar}
	Under the hypotheses of Theorem~\ref{thm:main2}, for all $x>0$ the function $q(\cdot,x) \in {\rm AC}(\R_0^+)$ and its a.e. derivative $\partial_t q(\cdot,x)$ satisfies, for a.a. $t \ge 0$,
	\begin{equation}\label{eq:ACqstar}
		|\partial_t q(t,x)| \le \frac{\Gamma\left(\alpha+\frac{1}{2}\right)}{\pi \Gamma(\alpha)}\sqrt{\frac{Kx}{t}}.
	\end{equation}
\end{prop}
\begin{proof}
For any $t > 0$, let $h \in \left(-\frac{t}{2},1\right)$ and consider $D^hq(t,x)=\frac{q(t+h,x)-q(t,x)}{h}$. Then we have, by \eqref{eq:selfsim},
\begin{equation*}
	|D^h q(t,x)| \le \int_0^1 \left|\frac{q_{\rm BS}(ts+hs,x)-q_{\rm BS}(ts,x)}{hs}\right|\frac{s^{\alpha}(1-s)^{-\alpha}}{\Gamma(\alpha)\Gamma(1-\alpha)}ds.
\end{equation*}
By Lagrange theorem, we know that there exists $\xi \in [\min\{ts,ts+hs\},\max\{ts,ts+hs\}]$ such that
\begin{equation*}
	\left|\frac{q_{\rm BS}(ts+hs,x)-q_{\rm BS}(ts,x)}{hs}\right|=\left|\partial_t q_{\rm BS}(x,\xi)\right|\le  \sqrt{\frac{Kx}{8\pi \xi}} \le \sqrt{\frac{Kx}{4\pi ts}},
\end{equation*}
where we also used \eqref{eq:der5}. Hence
\begin{equation*}
	|D^hq(t,x)| \le \sqrt{\frac{Kx}{4\pi t}}\int_0^1 \frac{s^{\alpha-\frac{1}{2}}(1-s)^{-\alpha}}{\Gamma(\alpha)\Gamma(1-\alpha)}ds=\frac{\Gamma\left(\alpha+\frac{1}{2}\right)}{\pi \Gamma(\alpha)}\sqrt{\frac{Kx}{t}}.
\end{equation*}
The right-hand side is independent of $h$ and belongs to $L^1(0,T)$ for all $T>0$, thus we can argue as in Proposition~\ref{prop:AC} and get the result.
\end{proof}
As claimed, the previous proposition implies that $q$ satisfies Item (i) of both Theorems \ref{thm:integrability2} and \ref{thm:Lapinside2}. We show this in the next result.
\begin{coro}\label{coro:diffest}
    Under the assumptions of Theorem \ref{thm:main2} we have that for all $t,x>0$
    \begin{equation*}
    \int_0^t |P_s q(t-s,x)-P_sq(t,x)|\frac{\alpha s^{-\alpha-1}}{\Gamma(1-\alpha)}ds<\infty
    \end{equation*}
    and, for all $\lambda>\frac{3}{8}$ and $x>0$,
    \begin{equation*}
    \int_0^\infty \int_0^t e^{-\lambda t}  |P_s q(t-s,x)-P_sq(t,x)|\frac{\alpha s^{-\alpha-1}}{\Gamma(1-\alpha)}ds\, dt<\infty.
    \end{equation*}
\end{coro}
\begin{proof}
    Fix $x>0$, and note that for $0 <s<t$,
    \begin{multline}
     \left|   q (t-s,x) - q (t,x) \right| \leq \int_0^s \left| \partial_t q(t-\tau,x) \right| d\tau \\
     \leq  \frac{\Gamma\left(\alpha+\frac{1}{2}\right)}{\pi \Gamma(\alpha)}\sqrt{Kx}\int_0^s(t-\tau)^{-1/2}d\tau
     =\frac{\Gamma\left(\alpha+\frac{1}{2}\right) s}{\pi \Gamma(\alpha)}\sqrt{\frac{Kx}{t-s}}.\label{623}
    \end{multline}
    This proves that $\left|   q (t-s,\cdot) - q (t,\cdot) \right| \in \cC_0(1/2)$. We get, by also employing Items $(1)$ and $(5)$ of Lemma \ref{lem:BSuniqueness},
    \begin{align*}
        \int_0^t &|P_s q(t-s,x)-P_sq(t,x)|\frac{\alpha s^{-\alpha-1}}{\Gamma(1-\alpha)}ds\\
        & \le \int_0^t P_s |q(t-s,\cdot)-q(t,\cdot)|(x)\frac{\alpha s^{-\alpha-1}}{\Gamma(1-\alpha)}ds \\
        &\le \frac{2 \alpha \Gamma\left(\alpha+\frac{1}{2}\right) \sqrt{K}}{\pi \Gamma(\alpha)\Gamma(1-\alpha)}e^{\frac{1}{2}|\log(x)|+\frac{3}{8}t}\int_0^t s^{-\alpha}(t-s)^{-\frac{1}{2}}ds\\
        &=\frac{2 \alpha \Gamma\left(\alpha+\frac{1}{2}\right) \sqrt{K}}{\Gamma(\alpha)\Gamma\left(\frac{3}{2}-\alpha\right)\sqrt{\pi} }t^{\frac{1}{2}-\alpha}e^{\frac{1}{2}|\log(x)|+\frac{3}{8}t},
    \end{align*}
    where we also used \cite[Lemma 1.11]{ascione2023fractional}. Furthermore, for any $\lambda>\frac{3}{8}$ it holds
    \begin{align*}
        \int_0^\infty &\int_0^t e^{-\lambda t}  |P_s q(t-s,x)-P_sq(t,x)|\frac{\alpha s^{-\alpha-1}}{\Gamma(1-\alpha)}ds\, dt \\
        &\le \frac{2 \alpha \Gamma\left(\alpha+\frac{1}{2}\right) \sqrt{K}}{\Gamma(\alpha)\Gamma\left(\frac{3}{2}-\alpha\right)\sqrt{\pi} }t^{\frac{1}{2}-\alpha}e^{\frac{1}{2}|\log(x)|}\int_0^\infty t^{\frac{1}{2}-\alpha}e^{-\left(\lambda-\frac{3}{8}\right)t}\, dt\\
        &=\frac{2 \alpha \Gamma\left(\alpha+\frac{1}{2}\right) \sqrt{K}}{\Gamma(\alpha)\sqrt{\pi} }t^{\frac{1}{2}-\alpha}e^{\frac{1}{2}|\log(x)|}\left(\lambda-\frac{3}{8}\right)^{\alpha-\frac{3}{2}}.
    \end{align*}
\end{proof}
Now we discuss the regularity of $q$ in $x$. This is done by means of the following proposition.
\begin{prop}\label{prop:regq}
Under the assumptions of Theorem \ref{thm:main2}, the following equalities hold for all $t,x>0$:
\begin{align}
    &\partial_x q (t,x) \, = \, \frac{1}{\Gamma (\alpha)\Gamma (1-\alpha)}\int_0^1 \partial_x q_{\rm BS} (\tau t, x) \, \tau^{\alpha-1} (1-\tau)^{-\alpha} d\tau , \label{derstar1}\\
    & \partial_x^2 q(t,x) \, = \, \frac{1}{\Gamma (\alpha)\Gamma (1-\alpha)}\int_0^1 \partial_x^2 q_{\rm BS} (\tau t, x) \, \tau^{\alpha-1} (1-\tau)^{-\alpha} d\tau , \label{derstar2}\\
    & G q (t,x) \, = \, \frac{1}{\Gamma (\alpha)\Gamma (1-\alpha)}\int_0^1 G q_{\rm BS} (\tau t, x) \, \tau^{\alpha-1} (1-\tau)^{-\alpha} d\tau . \label{derstar3}
\end{align}
Furthermore $q(t,\cdot) \in \cC_0(1)$, $\partial_x q(t,\cdot) \in \cC_0(0)$ and $Gq(t,\cdot) \in \cC_0(1/2)$ for all $t>0$, where
\begin{align}
|q(t,x)| &\le e^{|\log(x)|} \label{eq:est1q}\\
|\partial_x q(t,x)| &\le 1 \label{eq:est2q}\\
|G q(t,x)| & \le  \sqrt{\frac{K x}{2 t}}\frac{\Gamma\left(\alpha-\frac{1}{2}\right)}{2\pi \Gamma(\alpha)} \label{eq:est3q}.
\end{align}
In particular, $q \in \cC_{0, {\rm loc}}(\R_0^+)$.
\end{prop}
\begin{proof}
By \eqref{eq:qspecial} and \eqref{eq:der3.5} we already know that $q \in \cC_0(1)$ with $[q]_1 \le 1$. Moreover, by \eqref{eq:der4}, we know that we can apply the dominated convergence theorem to take the derivative inside the integral sign in \eqref{eq:qspecial} and then using \eqref{eq:selfsim}, getting \eqref{derstar1}. Furthermore, \eqref{eq:der4} also implies that $\partial_x q(t,\cdot) \in \cC_0(0)$ with $[\partial_x q(t,\cdot)]_0 \le 1$. 

To argue with the second derivative, fix $a > 0$ and notice that for all $x \ge a$ it holds, by \eqref{eq:der6},
\begin{align*}
    \int_0^1 \left|\partial_x^2 q_{\rm BS}(ts,x)\right|\frac{s^{\alpha-1}(1-s)^{-\alpha}}{\Gamma(1-\alpha)\Gamma(\alpha)}ds &\le \sqrt{\frac{K}{8\pi a^3 t}}\int_0^1 \frac{s^{\alpha-\frac{3}{2}}(1-s)^{-\alpha}}{\Gamma(1-\alpha)\Gamma(\alpha)}ds\\
    &=\frac{1}{2\pi}\sqrt{\frac{K}{2a^3 t}}\frac{\Gamma\left(\alpha-\frac{1}{2}\right)}{\Gamma(\alpha)}<\infty,
\end{align*}
since $\alpha>\frac{1}{2}$. Hence, by dominated convergence, we get \eqref{derstar2} and finally, multiplying both sides of \eqref{derstar2} by $\frac{x^2}{2}$, we also get \eqref{derstar3}. Finally, notice that \eqref{derstar3} and \eqref{eq:der5} imply that $Gq(t,\cdot) \in \cC_0(1/2)$ for all $t>0$ and in particular
\begin{equation*}
\left|Gq(t,x)\right| \le \sqrt{\frac{K x}{8\pi t}}\int_0^1 \frac{s^{\alpha-\frac{3}{2}}(1-s)^{-\alpha}}{\Gamma(1-\alpha)\Gamma(\alpha)}ds=\sqrt{\frac{K x}{2 t}}\frac{\Gamma\left(\alpha-\frac{1}{2}\right)}{2\pi \Gamma(\alpha)}.
\end{equation*}
\end{proof}
\begin{rmk}
Notice that Proposition \ref{prop:ACqstar}, Corollary \ref{coro:diffest}, \eqref{derstar1}, \eqref{eq:est1q} and \eqref{eq:est2q} hold for any $\alpha \in (0,1)$. The condition $\alpha>1/2$ comes into play only for \eqref{derstar2}, \eqref{derstar3} and \eqref{eq:est3q}. However, it is worth highlighting that, using a finer estimate in place of \eqref{eq:der6}, it is possible to prove that \eqref{derstar2} and \eqref{derstar3} still hold for all $t, x>0$ with $x \not = K$ if $\alpha \le 1/2$. However, notice that, for $h \in [0,1]$, by Lagrange's theorem,
\begin{align*}
\frac{\partial_x q_{\rm BS}(ts,K+h)-\partial_xq_{\rm BS}(ts,K)}{h}&=\partial_x^{2}q_{\rm BS}(ts,\xi)\\
&=\sqrt{\frac{K}{2\pi ts \xi^3}}\exp\left(-\frac{4\log^2(\xi/K)+t^2s^2}{8ts}\right)\\
&\ge \sqrt{\frac{K}{2\pi ts (K+1)^3}}\exp\left(-\frac{4\log^2\left(1+\frac{h}{K}\right)+t^2s^2}{8ts}\right)
\end{align*}
where $\xi \in [K,K+h]$. Hence, by \eqref{derstar1}, we have
\begin{align*}
&\frac{\partial_x q(t,K+h)-\partial_xq(t,K)}{h}\\
&\qquad =\int_0^1 \frac{\partial_x q_{\rm BS}(ts,K+h)-\partial_xq_{\rm BS}(ts,K)}{h}\frac{s^{\alpha-1}(1-s)^{-\alpha}}{\Gamma(\alpha) \Gamma(1-\alpha)}\, ds\\
&\qquad \ge \frac{\sqrt{K}}{\sqrt{2\pi t(K+1)^3}}\int_0^1 \exp\left(-\frac{4\log^2\left(1+\frac{h}{K}\right)+t^2s^2}{8ts}\right) \frac{s^{\alpha-\frac{3}{2}}(1-s)^{-\alpha}}{\Gamma(\alpha) \Gamma(1-\alpha)}\, ds.
\end{align*}
Taking the limit as $h \to 0^+$ we have, as a consequence of the monotone convergence theorem,
\begin{multline*}
\lim_{h \to 0^+}\frac{\partial_x q(t,K+h)-\partial_xq(t,K)}{h}\\
\ge \frac{\sqrt{K}}{\Gamma(\alpha) \Gamma(1-\alpha)\sqrt{2\pi t(K+1)^3}}\int_0^1 \exp\left(-\frac{ts}{8}\right) s^{\alpha-\frac{3}{2}}(1-s)^{-\alpha}\, ds=\infty,
\end{multline*}
if $\alpha \le \frac{1}{2}$. Hence, in such a case, $q(t,\cdot)$ cannot admit a second order derivative in $x=K$. Suppose that we have already shown that $q$ is solution of \eqref{eq:BSUnder}, then this lack of parabolic smoothing is also common in time-nonlocal equations, see for instance \cite[Proposition 3.3]{ascione2024time}.
\end{rmk}
Now we prove that $q$ satisfies Item (ii) of both Theorems \ref{thm:integrability2} and \ref{thm:Lapinside2}.
\begin{prop}
    Under the assumptions of Theorem \ref{thm:main2} it holds, for all $t,x>0$,
    \begin{equation*}
        \int_0^t |P_sq(t,x)-q(t,x)|\, \nu_\alpha(ds)<\infty,
    \end{equation*}
    and, for all $x>0$ and $\lambda>\frac{3}{8}$,
    \begin{equation*}
        \int_0^{\infty}\int_0^t e^{-\lambda t}|P_sq(t,x)-q(t,x)|\, \nu_\alpha(ds)\, dt<\infty
    \end{equation*}
    where $\nu_\alpha(ds)$ is the L\'evy measure corresponding to $\phi(\lambda) = \lambda^\alpha$, $\alpha \in (1/2,1)$, i.e., \begin{align}
        \nu_\alpha(ds)= \alpha s^{-\alpha-1}/\Gamma(1-\alpha)ds.
        \label{stablev}
    \end{align}
\end{prop}
\begin{proof}
By Proposition \ref{prop:regq} we know that $q(t,\cdot) \in \cC_0$ hence we can write
\begin{equation*}
P_sq(t,x)-q(t,x)=\int_0^s GP_\tau q(t,x) d\tau.
\end{equation*}
Recalling that, still by Proposition \ref{prop:regq}, $\partial_x q(t,\cdot), Gq(t,\cdot) \in \cC_0$, by Item $(3)$ of Lemma \ref{lem:BSuniqueness} we have
\begin{equation}\label{eq:Psestq}
{P_sq(t,x)-q(t,x)=\int_0^s P_\tau G q(t,x) d\tau.}
\end{equation}
In particular, $Gq(t,\cdot) \in \cC_0(1/2)$ and, by \eqref{eq:est3q}
\begin{equation*}
    [Gq(t,\cdot)]_{1/2} \le \sqrt{\frac{K}{2 t}}\frac{\Gamma\left(\alpha-\frac{1}{2}\right)}{2\pi \Gamma(\alpha)}.
\end{equation*}
Hence, by Item $(5)$ of Lemma \ref{lem:BSuniqueness} we get that $P_\tau Gq(t,\cdot) \in \cC_0(1/2)$ with
\begin{equation*}
    [P_\tau Gq(t,\cdot)]_{1/2} \le \sqrt{\frac{2K}{ t}}e^{\frac{3}{8}\tau}\frac{\Gamma\left(\alpha-\frac{1}{2}\right)}{2\pi \Gamma(\alpha)}.
\end{equation*}
Using this into \eqref{eq:Psestq} we have
\begin{align*}
    \left|P_sq(t,x)-q(t,x)\right|& {\le}\int_0^s \left|P_\tau Gvq(t,x)\right| d\tau \le \sqrt{\frac{2K}{ t}}e^{\frac{1}{2}|\log(x)|}\frac{\Gamma\left(\alpha-\frac{1}{2}\right)}{2\pi \Gamma(\alpha)}\int_0^s e^{\frac{3}{8}\tau} \, d\tau\\
    &\le \sqrt{\frac{2K}{ t}}e^{\frac{1}{2}|\log(x)|}\frac{\Gamma\left(\alpha-\frac{1}{2}\right)}{2\pi \Gamma(\alpha)}e^{\frac{3}{8}s}s.
\end{align*}
Integrating the latter against $\nu_\phi$, we have
\begin{align*}
\int_0^t \left|P_sq(t,x)-q(t,x)\right|\frac{\alpha s^{-\alpha-1}}{\Gamma(1-\alpha)}ds &\le \sqrt{\frac{2K}{ t}}e^{\frac{1}{2}|\log(x)|}\frac{\Gamma\left(\alpha-\frac{1}{2}\right)}{2\pi \Gamma(\alpha)}\int_0^t e^{\frac{3}{8}s}\frac{\alpha s^{-\alpha}}{\Gamma(1-\alpha)}\, ds\\
&\le e^{\frac{1}{2}|\log(x)|}\frac{\alpha \Gamma\left(\alpha-\frac{1}{2}\right) \sqrt{2K}}{2\pi \Gamma(\alpha)\Gamma(2-\alpha)}e^{\frac{3}{8}t}t^{\frac{1}{2}-\alpha}.
\end{align*}
Finally, integrating the latter against $e^{-\lambda t}$ for $\lambda>\frac{3}{8}$ we get
\begin{align*}
\int_0^{\infty} \int_0^t &e^{-\lambda t}\left|P_sq(t,x)-q(t,x)\right|\frac{\alpha s^{-\alpha-1}}{\Gamma(1-\alpha)}\, ds\, dt \\
&\le e^{\frac{1}{2}|\log(x)|}\frac{\alpha \Gamma\left(\alpha-\frac{1}{2}\right)\sqrt{2K}}{2\pi \Gamma(\alpha)\Gamma(2-\alpha)}\int_0^{\infty}e^{-\left(\lambda-\frac{3}{8}\right)t}t^{\frac{1}{2}-\alpha}\, dt\\
&=e^{\frac{1}{2}|\log(x)|}\frac{\alpha \Gamma\left(\alpha-\frac{1}{2}\right)\sqrt{2K}}{2\pi \Gamma(\alpha)\Gamma(2-\alpha)}\left(\lambda-\frac{3}{8}\right)^{\alpha-\frac{3}{2}}\Gamma\left(\frac{3}{2}-\alpha\right).
\end{align*}
\end{proof}
The last thing we need to check is Item (iii) of Theorem \ref{thm:Lapinside2}.
\begin{prop}
    Under the assumptions of Theorem \ref{thm:main2} for all $x,\lambda>0$ it holds
    \begin{equation*}
    \int_0^\infty e^{-\lambda t} \overline{\nu}_\alpha(t) |q(t,x)|\, dt <\infty,
    \end{equation*}
    where $\bar{\nu}_\alpha(t)$ is the tail of \eqref{stablev}.
\end{prop}
\begin{proof}
    Since $\overline{\nu}_\phi(t)=\frac{t^{-\alpha}}{\Gamma(1-\alpha)}$, we just use {the inequality} \eqref{eq:est1q} to get
    \textcolor{black}{\begin{align*}
    \int_0^\infty e^{-\lambda t} \frac{t^{-\alpha}}{\Gamma(1-\alpha)} |q(t,x)|\, dt \le \frac{e^{|\log(x)|}}{\Gamma(1-\alpha)}\int_0^\infty t^{-\alpha}e^{-\lambda t}dt=\frac{e^{|\log(x)|}}{\lambda^{1-\alpha}}.
    \end{align*}}
\end{proof}

\subsubsection{Existence of the solution}
We are now ready to prove a first bit of Theorem \ref{thm:main2}, which is summarized in the next proposition.
\begin{prop}\label{prop:qsolution}
Under the assumptions of Theorem \ref{thm:main2}, $q$ is a solution of \eqref{eq:BSUnder}.
\end{prop}
\begin{proof}
    In view of the results in Subsection \ref{regstar}, we know by Theorem \ref{thm:integrability2} that $(\partial_t-G)^\alpha q(t,x)$ is well-defined for all $x,t>0$. Moreover, for any $x>0$ we know by Theorem \ref{thm:Lapinside2} that $(\partial_t-G)^\alpha q(\cdot,x)$ admits Laplace transform for any $\lambda>\frac{3}{8}$.
    We divide the proof into two parts. First we check that for all $x>0$ and $\lambda>\frac{3}{8}$ we have
\begin{equation}\label{eq:Laptranseq}
\int_0^\infty e^{-\lambda t}(\partial_t-G)^\alpha q(t,x)dt=\lambda^{\alpha-1}(x-K)_+,
\end{equation}
which is the first equation in  \eqref{eq:BSUnder} after taking the Laplace transform in the $t$ variable on both sides. Then we use this to prove that
\begin{align}
    (\partial_t-G)^\alpha q(t,x) = \frac{t^{-\alpha}}{\Gamma(1-\alpha)} q(0,x).
\end{align}
We begin with \eqref{eq:Laptranseq}.
Consider $x>0$, $\lambda>\frac{3}{8}$ and denote
  \begin{equation*}
  \widetilde{q}(\lambda,x)=\int_0^\infty e^{-\lambda t}q(t,x) \, dt.
  \end{equation*}
  Let us first evaluate an alternative formula for $\widetilde{q}$. Set $u_K(x)=(x-K)_+$ and notice that, by \eqref{eq:qspecial}, we get
    \begin{align}
    \widetilde{q}(\lambda,x) \, = \, &\frac{1}{\Gamma(\alpha)\Gamma(1-\alpha)}\int_0^\infty e^{-\lambda t}\int_0^t s^{\alpha-1}(t-s)^{-\alpha}q_{\rm BS}(s,x)\, ds \, dt \notag \\
    = \, & \frac{\lambda^{\alpha -1}}{\Gamma(\alpha)} \int_0^{+\infty} e^{-\lambda s} s^{\alpha-1} q_{\rm BS}(s,x) \, ds \notag \\
    = \, & \frac{\lambda^{\alpha -1}}{\Gamma (\alpha)} \int_0^{+\infty} e^{-\lambda s} s^{\alpha -1} P_su_K(x) \, ds, \label{qtildastar}
    \end{align}
    where, in the first step, we used the convolution Theorem for Laplace transform (e.g., \cite[Proposition 1.6.4]{arendt2011vector}) and then we applied Tonelli's theorem since $q_{\rm BS}$ is non-negative.  As a consequence of \eqref{eq:postFubini}, we have that 
\begin{align}
		\int_0^\infty e^{-\lambda t} &\l  \partial_t -G \r^\alpha q(x,t) \, dt \, \nonumber\\= \,
  & -\int_0^\infty \int_0^\infty e^{-\lambda t} (P_s q(t-s,x)\mathds{1}_{(0,t]}(s)-q(t,x))\, dt \,\nu_\alpha(ds) \nonumber\\
  \, = \, & -\int_0^\infty \left( \int_s^\infty  \int_0^\infty e^{-\lambda t}q(t-s,y)p_{\rm GBM}(s,y;x)\, dy \, dt - \widetilde{q}(\lambda,x)\right) \, \nu_\alpha (ds) \nonumber\\
  \, = \, & -\int_0^\infty \left( e^{-\lambda s}\int_0^\infty    \widetilde{q} (\lambda,y)p_{\rm GBM}(s,y;x)\, dy - \widetilde{q}(\lambda,x)\right) \, \nu_\alpha (ds)\nonumber
  \\
  \, = \, & -\int_0^\infty \left( e^{-\lambda s}P_s(\widetilde{q} (\lambda,\cdot))(x) - \widetilde{q}(\lambda,x)\right) \, \nu_\alpha (ds), \nonumber
	\end{align}
 where in the third equality we used again Tonelli's theorem to exchange the order of the integrals (since $q$ and $p_{\rm GBM}$ are non-negative) after having made explicit the action of the operator $P_s$ in the second step.  Now we can replace $\widetilde{q}$ in the previous identity by making use of \eqref{qtildastar}. To do this, we first observe that
 \begin{align}
     P_s \l \widetilde{q} (\lambda, \cdot) \r (x) \, = \, & \frac{\lambda^{\alpha-1}}{\Gamma(\alpha)} \int_0^{+\infty}\int_0^{+\infty} e^{-\lambda \tau} \tau^{\alpha-1} P_\tau u_K(y) p_{\rm GBM}(s, y;x) d\tau \, dy \notag \\
     = \, & \frac{\lambda^{\alpha-1}}{\Gamma(\alpha)} \int_0^{+\infty} e^{-\lambda \tau} \tau^{\alpha-1} P_{s+\tau} u_K(x) \, d\tau
 \end{align}
 and then we have
 \begin{align}
		&\int_0^\infty e^{-\lambda t} \l  \partial_t -G \r^\alpha q(x,t) \, dt \notag \\  = \, &-\frac{\alpha\lambda^{\alpha-1}}{\Gamma(\alpha)\Gamma(1-\alpha)}\int_0^\infty   \int_0^{+\infty} \tau^{\alpha-1}s^{-\alpha-1}\left( e^{-\lambda (s+\tau)} P_{s+\tau} u_K(x) \,  -   e^{-\lambda \tau} P_{\tau} u_K(x) \,  \right)d\tau\, ds \notag \\
  = \, & -\frac{\alpha\lambda^{\alpha-1}}{\Gamma(\alpha)\Gamma(1-\alpha)}\int_0^\infty   \int_0^{+\infty} \tau^{\alpha-1}s^{-\alpha-1}e^{-\lambda \tau}P_\tau\left( e^{-\lambda s} q_{\rm BS}(s,x)\,  -  u_K(x) \,  \right)d\tau\, ds,
  \end{align}
 where we used the fact that $P_\tau$ is a semigroup action and the explicit form of $\nu_\alpha(\cdot)$ given in \eqref{stablev}. However, if we observe that 
  \begin{align}
     e^{-\lambda s} q_{\rm BS}(s,x)\,  -  u_K(x) \, = \, & \int_0^s \partial_w e^{-\lambda w} q_{\rm BS} (w,x) \, dw \notag \\ \,& \int_0^s e^{-\lambda w} \l \partial_w q_{\rm BS}(w,x) - \lambda q_{\rm BS} (w,x) \r \, dw \notag \\
     = \, & \int_0^s e^{-\lambda w}  \l G  - \lambda   \r \,  q_{\rm BS}(w,x) \, dw 
 \end{align}
and we explicitly write the integral in $P_\tau$, we get
 \begin{align}
     &\int_0^\infty e^{-\lambda t} \l  \partial_t -G \r^\alpha q(x,t) \, dt \notag \\  = \, & -\frac{\alpha\lambda^{\alpha-1}}{\Gamma(\alpha)\Gamma(1-\alpha)}\int_0^\infty   \int_0^{+\infty}\int_0^{+\infty} \int_0^s\tau^{\alpha-1}s^{-\alpha-1}e^{-\lambda (\tau+w)} \l G  - \lambda   \r \,  q_{\rm BS}(w,y) \,  p_{\rm GBM}(\tau,y;x) \, dw  \, dy \, d\tau\, ds \label{eq:preFub}.
 \end{align}
We want to use Fubini's theorem to change the order of the integrals. To do this, we notice that \eqref{eq:der3} implies $G q_{BS} \ge 0$ and then
 \begin{align}
 &   \int_0^{+\infty}\int_0^{+\infty}\int_0^{+\infty} \int_0^s \tau^{\alpha-1}s^{-\alpha-1}e^{-\lambda (\tau+w)} \left| (G-\lambda) q_{\rm BS} (w,y) \right| p_{\rm GBM} (\tau, y;x) \, dw \, dy \, d\tau \, ds \notag \\ &\leq \int_0^{+\infty}\int_0^{+\infty}\int_0^{s} \int_0^{+\infty} \tau^{\alpha-1} s^{-\alpha-1}e^{-\lambda (\tau+w)}(G+\lambda)q_{\rm BS}(w,y) p_{\rm GBM} (\tau, y;x) \, dy \, dw  \, d\tau \, ds\\
 &=\int_0^{+\infty}\int_0^{+\infty}\int_0^{s} \tau^{\alpha-1} s^{-\alpha-1}e^{-\lambda (\tau+w)}P_\tau (G+\lambda)q_{\rm BS}(w,x) \, dw  \, d\tau \, ds
 \label{bound646}
 \end{align}
Since $q_{\rm BS}(t,\cdot) \in \cC_0$, $\partial_x q_{\rm BS} \in \cC_0$ and $Gq_{\rm BS} \in \cC_0$, we can exchange the operator $(G+\lambda)$ with $P_\tau$, getting
 \begin{align}
 &   \int_0^{+\infty}\int_0^{+\infty}\int_0^{+\infty} \int_0^s \tau^{\alpha-1}s^{-\alpha-1}e^{-\lambda (\tau+w)} \left| (G-\lambda) q_{\rm BS} (w,y) \right| p_{\rm GBM} (\tau, y;x) \, dw \, dy \, d\tau \, ds \notag \\ &\leq \int_0^{+\infty}\int_0^{+\infty}\int_0^{s} \tau^{\alpha-1} s^{-\alpha-1}e^{-\lambda (\tau+w)}(G+\lambda)P_\tau q_{\rm BS}(w,x) \, dw  \, d\tau \, ds\\
 &= \int_0^{+\infty}\int_0^{+\infty}\int_0^{s} \tau^{\alpha-1} s^{-\alpha-1}e^{-\lambda (\tau+w)}(G+\lambda)q_{\rm BS}(\tau+w,x) \, dw  \, d\tau \, ds,
 \label{bound6462}
 \end{align}
where we also used the fact that $P_\tau q_{\rm BS}(w,x)=P_\tau P_w u_K(x)=P_{\tau+w}u_K(x)=q_{\rm BS}(w+\tau,x)$. By \eqref{eq:der3} we have
\begin{equation}\label{eq:control}
    Gq_{\rm BS}(\tau+w,x) \le \sqrt{\frac{Kx}{8\pi(\tau+w)}}\exp\left(-\frac{\log^2\left(\frac{x}{K}\right)}{2(\tau+w)}\right).
\end{equation}
Consider the function $r \in \R^+ \mapsto \sqrt{r}\exp\left(-r\frac{\log^2\left(\frac{x}{K}\right)}{2}\right) \in \R$. It is possible to check that such a function has a maximum at $r_\star=\log^{-2}\left(\frac{x}{K}\right)$ and then
\begin{equation*}
    \sqrt{r}\exp\left(-r\frac{\log^2\left(\frac{x}{K}\right)}{2}\right) \le \frac{1}{\sqrt{e}}\left|\log^{-1}\left(\frac{x}{K}\right)\right|.
\end{equation*}
Plugging the latter into \eqref{eq:control}, adding $\lambda q_{\rm BS}(\tau+w,x)$ and also using  \eqref{eq:der3.5} we get
\begin{equation*}
    (G+\lambda)q_{\rm BS}(\tau+w,x) \le \sqrt{\frac{Kx}{8\pi e}}\left|\log^{-1}\left(\frac{x}{K}\right)\right|+\lambda e^{|\log(x)|}.
\end{equation*}
Plugging this inequality into \eqref{bound6462} we obtain
 \begin{align}
&\int_0^{+\infty}\int_0^{+\infty}\int_0^{+\infty} \int_0^s \tau^{\alpha-1}s^{-\alpha-1}e^{-\lambda (\tau+w)} \left| (G-\lambda) q_{\rm BS} (w,y) \right| p_{\rm GBM} (\tau, y;x) \, dw \, dy \, d\tau \, ds \notag \\ \leq \,& \frac{1}{\lambda}\left(\sqrt{\frac{Kx}{8\pi e}}\left|\log^{-1}\left(\frac{x}{K}\right)\right|+\lambda e^{|\log(x)|}\right)\int_0^{+\infty}\int_0^{+\infty}  \tau^{\alpha-1} (1-e^{-\lambda s})s^{-\alpha-1}e^{-\lambda \tau}  \, d\tau \, ds \nonumber\\
= & \frac{1}{\lambda}\left(\sqrt{\frac{Kx}{8\pi e}}\left|\log^{-1}\left(\frac{x}{K}\right)\right|+\lambda e^{|\log(x)|}\right)\left(\int_{0}^{+\infty}\tau^{\alpha-1}e^{-\lambda \tau} \, d\tau \right)\left(\int_{0}^{+\infty}s^{-\alpha-1}(1-e^{-\lambda s}) \, ds\right)\nonumber\\
= &\frac{1}{\lambda}\left(\sqrt{\frac{Kx}{8\pi e}}\left|\log^{-1}\left(\frac{x}{K}\right)\right|+\lambda e^{|\log(x)|}\right)\frac{\Gamma(\alpha)\Gamma(1-\alpha)}{\alpha}<\infty \nonumber
\end{align}
where, in the last step, we used the integral representation of the Gamma function that
\begin{align*}
    \int_{0}^{+\infty}s^{-\alpha-1}(1-e^{-\lambda s}) \, ds = \frac{\Gamma(1-\alpha)}{\alpha} \lambda^\alpha.
\end{align*}
We can now safely use Fubini's theorem in \eqref{eq:preFub}, obtaining
 \begin{align}
     &\int_0^\infty e^{-\lambda t} \l  \partial_t -G \r^\alpha q(x,t) \, dt \notag \\  
     &=  -\frac{\alpha\lambda^{\alpha-1}}{\Gamma(\alpha)\Gamma(1-\alpha)}\int_0^\infty   \int_0^{+\infty}\int_0^{+\infty} \int_w^{+\infty}\tau^{\alpha-1}s^{-\alpha-1}e^{-\lambda (\tau+w)}  \l G  - \lambda   \r \,  q_{\rm BS}(w,y) \,  p_{\rm GBM}(\tau,y;x) \, ds  \, dy \, d\tau\, dw \nonumber \\
     &=  -\frac{\lambda^{\alpha-1}}{\Gamma(\alpha)\Gamma(1-\alpha)}\int_0^\infty   \int_0^{+\infty}\tau^{\alpha-1}w^{-\alpha}e^{-\lambda (\tau+w)}  P_\tau \l G  - \lambda   \r \,  q_{\rm BS}(w,x)   \, d\tau\, dw \notag\\
     &=  -\frac{\lambda^{\alpha-1}}{\Gamma(\alpha)\Gamma(1-\alpha)}\int_0^\infty   \int_0^{+\infty}\tau^{\alpha-1}w^{-\alpha}e^{-\lambda (\tau+w)}   \l G  - \lambda   \r \,  q_{\rm BS}(\tau+w,x)   \, d\tau\, dw \notag \\
     &=  -\frac{\lambda^{\alpha-1}}{\Gamma(\alpha)\Gamma(1-\alpha)}\int_0^\infty   \int_w^{+\infty}(z-w)^{\alpha-1}w^{-\alpha}   \l G  - \lambda   \r \,  e^{-\lambda z} q_{\rm BS}(z,x)   \, dz\, dw \notag\\
     &=  -\frac{\lambda^{\alpha-1}}{\Gamma(\alpha)\Gamma(1-\alpha)}\int_0^\infty   \l G  - \lambda   \r \,  e^{-\lambda z} q_{\rm BS}(z,x)\left(\int_0^{z}(z-w)^{\alpha-1}w^{-\alpha}   \, dw\right)\, dz \notag\\
     &= -\lambda^{\alpha-1}\int_0^\infty   \l G  - \lambda   \r \,  e^{-\lambda z} q_{\rm BS}(z,x) dz, \label{eq:tosubs}
 \end{align}
where in the third equality we used that $P_\tau q_{\rm BS} (w,y) = P_\tau P_w u_K(x) = P_{\tau+w}u_K(x) = q_{\rm BS} (\tau+w,x)$, in the fourth equality we used the change of variables $\tau=z-w$, in the fifth equality we used again Fubini's theorem and in the last equality we used the integral representation of the Beta function. Now notice that \eqref{eq:der3.5}, \eqref{eq:der4} and \eqref{eq:der5} guarantee that we can take the Laplace transform in the variable $t$ on both sides of the first equation of \eqref{eq:BSclass1}, leading to
\begin{equation*}
\lambda \int_0^{+\infty} e^{-\lambda z}q_{\rm BS}(z,x)dz-u_K(x)=\int_0^{+\infty} e^{-\lambda z}Gq_{\rm BS}(z,x)dz.
\end{equation*}
This equality can be rearranged to obtain
\begin{equation*}
\int_0^{+\infty} (\lambda-G)e^{-\lambda z} q_{\rm BS}(z,x)dz=u_K(x)
\end{equation*}
which, replaced into \eqref{eq:tosubs} gives \eqref{eq:Laptranseq}.

By injectivity of the Laplace transform, we know that for all $x>0$ there exists a measurable set $\mathcal{I}_x \subset \R^+$ such that $|\mathcal{I}_x|=0$ and for all $t \in \R^+ \setminus \mathcal{I}_x$ one has
\begin{equation*}
    (\partial_t-G)^\alpha q(t,x)=\frac{t^{-\alpha}}{\Gamma(1-\alpha)}q(0,x).
\end{equation*}
We want to prove that $\mathcal{I}_x=\varnothing$ for all $x>0$. To do this, we first show that $q$ is continuous. Indeed, let $t,x>0$ and consider a sequence $(t_n,x_n) \to (t,x)$. By \eqref{eq:qspecial} and \eqref{eq:selfsim} we have
\begin{equation*}
    q(t_n,x_n)=\int_0^1 q_{\rm BS}(t_ns,x_n)\frac{(1-s)^{-\alpha}s^{\alpha-1}}{\Gamma(1-\alpha)\Gamma(\alpha)}ds.
\end{equation*}
Since $x_n \to x$, there exist $0<a<b$ such that $x_n \in [a,b]$ for all $n \in \N$ and then $|q_{\rm BS}(t_ns,x_n)| \le e^{\max\{|\log(a)|, |\log(b)|\}}$. Hence, we can use the dominated convergence theorem to get
\begin{equation*}
    \lim_{n \to +\infty}q(t_n,x_n)=\int_0^1 q_{\rm BS}(ts,x)\frac{(1-s)^{-\alpha}s^{\alpha-1}}{\Gamma(1-\alpha)\Gamma(\alpha)}ds=q(t,x).
\end{equation*}
Now we prove that for fixed $x>0$ the function $(\partial_t-G)^\alpha q(\cdot,x)$ is right-continuous. Indeed, let $t_n \downarrow t$ for some $t>0$ and observe that 
\begin{align*}
    &|(\partial_t-G)^\alpha q(t_n,x)-(\partial_t-G)^\alpha q(t,x)| \\
    &\le \frac{\alpha}{\Gamma(1-\alpha)}\int_0^t |P_sq(t-s,x)-q(t,x)-P_sq(t_n-s,x)+q(t_n,x)| s^{-1-\alpha}ds\\
    &+\frac{\alpha}{\Gamma(1-\alpha)}\int_t^{t_n} |P_sq(t_n-s,x)-q(t_n,x)| s^{-1-\alpha}ds\\
    &+\frac{t^{-\alpha}-t_n^{-\alpha}}{\Gamma(1-\alpha)}|q(t,x)|+|q(t,x)-q(t_n,x)|\frac{t_n^{-\alpha}}{\Gamma(1-\alpha)}\\
    &:=\frac{\alpha}{\Gamma(1-\alpha)}\left(I_n(x)+J_n(x)+\frac{t^{-\alpha}-t_n^{-\alpha}}{\alpha}|q(t,x)|+|q(t,x)-q(t_n,x)|\frac{t_n^{-\alpha}}{\alpha}\right).
\end{align*}
Let us first consider $I_n(x)$. We have
\begin{align*}
    |P_sq(t-s,x)&-q(t,x)-P_sq(t_n-s,x)+q(t_n,x)|\\
    & \le |P_sq(t-s,x)-P_sq(t_n-s,x)| + |q(t,x)-q(t_n,x)|\\
    &\le P_s|q(t_n-s,\cdot)-q(t-s,\cdot)|(x) + |q(t,x)-q(t_n,x)|
\end{align*}
Arguing as in the proof of Corollary \ref{coro:diffest}, we have that $|q(t-s,\cdot)-q(t_n-s,\cdot)| \in \cC_0(1/2)$ with $t_n-s-(t_n-t)$
\begin{equation*}
[|q(t-s,\cdot)-q(t_n-s,\cdot)|]_{1/2} \le \frac{\Gamma\left(\alpha+\frac{1}{2}\right)(t_n-t)}{\pi \Gamma(\alpha)}\sqrt{\frac{K}{t-s}}.
\end{equation*}
Hence, by Item $(5)$ of Lemma \ref{lem:BSuniqueness} we have
\begin{align*}
P_s|q(t_n-s,\cdot)-q(t-s,\cdot)|(x) \le 2\frac{\Gamma\left(\alpha+\frac{1}{2}\right)(t_n-t)}{\pi \Gamma(\alpha)}\sqrt{\frac{K}{t-s}}e^{\frac{3}{8}s+\frac{|\log(x)|}{2}}.
\end{align*}
Combining this inequality with the fact that $q$ is continuous, for fixed $s \in [0,t)$, we find
\begin{equation*}
    \lim_{n \to \infty}|P_sq(t-s,x)-q(t,x)-P_sq(t_n-s,x)+q(t_n,x)|=0.
\end{equation*}
Hence, it is sufficient to show that we can use the dominated convergence theorem to take the limit inside the integral sign in $I_n(x)$. To do this, observe that
\begin{align}\label{eq:In0}
\begin{split}
    |P_sq(t-s,x)&-q(t,x)-P_sq(t_n-s,x)+q(t_n,x)|\\
    & \le |P_sq(t-s,x)-P_sq(t,x)| + |P_sq(t,x)-q(t,x)|\\
    & \qquad +|P_sq(t_n-s,x)-P_sq(t_n,x)|+|P_sq(t_n,x)-q(t_n,x)|\\
    &\le P_s|q(t-s,\cdot)-q(t,\cdot)|(x) + |P_sq(t,x)-q(t,x)|\\
    & \qquad + P_s|q(t_n-s,\cdot)-q(t_n,\cdot)|(x)+|P_sq(t_n,x)-q(t_n,x)|.
    \end{split}
\end{align}
Again recall that $|q(t-s,\cdot)-q(t,\cdot)| \in \cC_0(1/2)$ with
\begin{equation*}
[|q(t-s,\cdot)-q(t,\cdot)|]_{1/2} \le \frac{\Gamma\left(\alpha+\frac{1}{2}\right)s}{\pi \Gamma(\alpha)}\sqrt{\frac{K}{t-s}},
\end{equation*}
hence, by Item $(5)$ of Lemma \ref{lem:BSuniqueness},
\begin{align}
    P_s|q(t-s,\cdot)-q(t,\cdot)|(x) &\le \frac{2\Gamma\left(\alpha+\frac{1}{2}\right)s}{\pi \Gamma(\alpha)}\sqrt{\frac{K}{t-s}}e^{\frac{3}{8}t+\frac{|\log(x)|}{2}}\nonumber\\
    & \le \frac{2\Gamma\left(\alpha+\frac{1}{2}\right)s}{\pi \Gamma(\alpha)}\sqrt{\frac{K}{t-s}}e^{\frac{3}{8}t_1+\frac{|\log(x)|}{2}} \label{eq:In1}\\
    P_s|q(t_n-s,\cdot)-q(t_n,\cdot)|(x) &\le \frac{2\Gamma\left(\alpha+\frac{1}{2}\right)s}{\pi \Gamma(\alpha)}\sqrt{\frac{K}{t_n-s}}e^{\frac{3}{8}t_n+\frac{|\log(x)|}{2}} \nonumber \\
    &\le \frac{2\Gamma\left(\alpha+\frac{1}{2}\right)s}{\pi \Gamma(\alpha)}\sqrt{\frac{K}{t-s}}e^{\frac{3}{8}t_1+\frac{|\log(x)|}{2}}. \label{eq:In2}
\end{align}
Furthermore, notice that
\begin{equation*}
    P_sq(t,x)-q(t,x)=\int_0^s G P_\tau q(t,x) \, d\tau=\int_0^s  P_\tau G q(t,x) \, d\tau
\end{equation*}
where we used Item $(3)$ of Lemma \ref{lem:BSuniqueness} since $q(t,\cdot), \partial_x q(t,\cdot)$ and $Gq(t,\cdot)$ belong to $\cC_0$. In particular, recall that, by \eqref{eq:est3q}, that $Gq(t,\cdot) \in \cC_0(1/2)$ and
\begin{equation*}
    [Gq(t,\cdot)]_{\frac{1}{2}} \le \sqrt{\frac{K}{2t}}\frac{\Gamma\left(\alpha-\frac{1}{2}\right)}{2\pi \Gamma(\alpha)}.
\end{equation*}
Hence, by Item $(5)$ of Lemma \ref{lem:BSuniqueness} we get, since $\tau \le s \le t$
\begin{align*}
    |P_\tau Gq(t,x)| &\le 2\sqrt{\frac{K}{2t}}\frac{\Gamma\left(\alpha-\frac{1}{2}\right)}{2\pi \Gamma(\alpha)}e^{\frac{3}{8}\tau +\frac{|\log(x)|}{2}}\\
    &\le 2\sqrt{\frac{K}{2t}}\frac{\Gamma\left(\alpha-\frac{1}{2}\right)}{2\pi \Gamma(\alpha)}e^{\frac{3}{8}t_1+\frac{|\log(x)|}{2}}
\end{align*}
and then
\begin{equation}\label{eq:In3}
    |P_sq(t,x)-q(t,x)| \le 2\sqrt{\frac{K}{2t}}\frac{\Gamma\left(\alpha-\frac{1}{2}\right)}{2\pi \Gamma(\alpha)}e^{\frac{3}{8}t_1+\frac{|\log(x)|}{2}}s
\end{equation}
Analogously,
\begin{equation}\label{eq:In4}
    |P_sq(t_n,x)-q(t_n,x)| \le 2\sqrt{\frac{K}{2t}}\frac{\Gamma\left(\alpha-\frac{1}{2}\right)}{2\pi \Gamma(\alpha)}e^{\frac{3}{8}t_1+\frac{|\log(x)|}{2}}s.
\end{equation}
Combining \eqref{eq:In1}, \eqref{eq:In2}, \eqref{eq:In3} and \eqref{eq:In4} with \eqref{eq:In0} we have
\begin{align*}
\begin{split}
    |P_sq(t-s,x)&-q(t,x)-P_sq(t_n-s,x)+q(t_n,x)|\\
    & \le \frac{4\Gamma\left(\alpha+\frac{1}{2}\right)s}{\pi \Gamma(\alpha)}\sqrt{\frac{K}{t-s}}e^{\frac{3}{8}t_1+\frac{|\log(x)|}{2}}+4\sqrt{\frac{K}{2t}}\frac{\Gamma\left(\alpha-\frac{1}{2}\right)}{2\pi \Gamma(\alpha)}e^{\frac{3}{8}t_1+\frac{|\log(x)|}{2}}s.
    \end{split}
\end{align*}
Integrating the latter against $s^{-1-\alpha}$ we get
\begin{align*}
\begin{split}
    \int_0^t |P_sq(t-s,x)&-q(t,x)-P_sq(t_n-s,x)+q(t_n,x)|s^{-1-\alpha}\, ds\\
    & \le \frac{4\Gamma\left(\alpha+\frac{1}{2}\right)}{\pi \Gamma(\alpha)}\sqrt{K}e^{\frac{3}{8}t_1+\frac{|\log(x)|}{2}}\int_0^t (t-s)^{-\frac{1}{2}}s^{-\alpha}\, ds \\
    &\qquad +4\sqrt{\frac{K}{2t}}\frac{\Gamma\left(\alpha-\frac{1}{2}\right)}{2\pi \Gamma(\alpha)}e^{\frac{3}{8}t_1+\frac{|\log(x)|}{2}}\int_0^t s^{-\alpha}\, ds\\
    &=\frac{4\Gamma\left(\alpha+\frac{1}{2}\right)\Gamma(1-\alpha)}{ \Gamma(\alpha)\Gamma\left(\frac{3}{2}-\alpha\right)}\sqrt{\frac{K}{\pi}}e^{\frac{3}{8}t_1+\frac{|\log(x)|}{2}}t^{\frac{1}{2}-\alpha}\\
    &\qquad +4\sqrt{\frac{K}{2}}\frac{\Gamma\left(\alpha-\frac{1}{2}\right)}{2\pi (1-\alpha)\Gamma(\alpha)}e^{\frac{3}{8}t_1+\frac{|\log(x)|}{2}}t^{\frac{1}{2}-\alpha}
    \end{split}
\end{align*}
where we also used \cite[Lemma 1.11]{ascione2023fractional}. Hence, we can use the dominated convergence theorem to guarantee that
\begin{equation*}
    \lim_{n \to \infty}I_n(x)=0.
\end{equation*}
Next, we move to $J_n(x)$. We have, arguing as before and recalling that this time $s \in [t_n,t]$,
\begin{align*}
    |P_sq(t_n-s,x)-q(t_n,x)| &\le |P_sq(t_n-s,x)-P_s q(t_n,x)|+|P_sq(t_n,x)-q(t_n,x)|\\
    &\le P_s|q(t_n-s,\cdot)-q(t_n,\cdot)|+|P_sq(t_n,x)-q(t_n,x)|\\
    &\le \frac{2\Gamma\left(\alpha+\frac{1}{2}\right)s}{\pi \Gamma(\alpha)}\sqrt{\frac{K}{t_n-s}}e^{\frac{3}{8}t_1+\frac{|\log(x)|}{2}}+2\sqrt{\frac{K}{2t}}\frac{\Gamma\left(\alpha-\frac{1}{2}\right)}{2\pi \Gamma(\alpha)}e^{\frac{3}{8}t_1+\frac{|\log(x)|}{2}}s.
\end{align*}
Hence,
\begin{align*}
   & J_n(x)\notag \\ &\le \frac{2\Gamma\left(\alpha+\frac{1}{2}\right)}{\pi \Gamma(\alpha)}\sqrt{K}e^{\frac{3}{8}t_1+\frac{|\log(x)|}{2}}\int_{t}^{t_n} (t_n-s)^{-\frac{1}{2}}s^{-\alpha}\, ds+2\sqrt{\frac{K}{2t}}\frac{\Gamma\left(\alpha-\frac{1}{2}\right)}{2\pi \Gamma(\alpha)}e^{\frac{3}{8}t_1+\frac{|\log(x)|}{2}}\int_{t}^{t_n}s^{-\alpha}ds\\
    &=\frac{2\Gamma\left(\alpha+\frac{1}{2}\right)}{\pi \Gamma(\alpha)}\sqrt{K}e^{\frac{3}{8}t_1+\frac{|\log(x)|}{2}}\int_{0}^{t_n-t} w^{-\frac{1}{2}}(t_n-w)^{-\alpha}\, dw+2\sqrt{\frac{K}{2t}}\frac{\Gamma\left(\alpha-\frac{1}{2}\right)}{2\pi (1-\alpha)\Gamma(\alpha)}e^{\frac{3}{8}t_1+\frac{|\log(x)|}{2}}(t_n^{1-\alpha}-t^{1-\alpha})\\
    &=\frac{2\Gamma\left(\alpha+\frac{1}{2}\right)}{\pi \Gamma(\alpha)}\sqrt{K}e^{\frac{3}{8}t_1+\frac{|\log(x)|}{2}}t_n^{\frac{1}{2}-\alpha}\int_{0}^{1-\frac{t}{t_n}} z^{-\frac{1}{2}}(1-z)^{-\alpha}\, dz+2\sqrt{\frac{K}{2t}}\frac{\Gamma\left(\alpha-\frac{1}{2}\right)}{2\pi (1-\alpha)\Gamma(\alpha)}e^{\frac{3}{8}t_1+\frac{|\log(x)|}{2}}(t_n^{1-\alpha}-t^{1-\alpha}).
\end{align*}
Hence, taking the limit as $n \to \infty$, we get
\begin{equation*}
    \lim_{n \to \infty}J_n(x)=0.
\end{equation*}
{Combined with the fact that also the remaining terms converge to 0 as $n \to \infty$,} this finally shows that
\begin{equation*}
    \lim_{n \to \infty}(\partial_t-G)^\alpha q(t_n,x)=(\partial_t-G)^\alpha q(t,x).
\end{equation*}
Now we prove that $\mathcal{I}_x = \varnothing$ for all $x>0$. Indeed, fix $x>0$ and assume that $\mathcal{I}_x \not = \varnothing$. Let $t \in \mathcal{I}_x$ and consider an interval of the form $[t,t+\delta]$ for some $\delta>0$. Since $|\mathcal{I}_x \cap [t,t+\delta]|=0$, we know that there exists a sequence $t_n \downarrow t$ such that $t_n \not \in \mathcal{I}_x$ for all $n \in \N$. For such a sequence, we know that
\begin{equation*}
    (\partial_t-G)^{\alpha}q(t_n,x)=\frac{t_n^{-\alpha}}{\Gamma(1-\alpha)}q(0,x).
\end{equation*}
Taking the limit as $n \to \infty$ we get
\begin{equation*}
    (\partial_t-G)^{\alpha}q(t,x)=\frac{t^{-\alpha}}{\Gamma(1-\alpha)}q(0,x)
\end{equation*}
that is absurd. Hence $\mathcal{I}_x=\varnothing$ and thus we get that
\begin{equation}
    (\partial_t-G)^{\alpha}q(t,x)=\frac{t^{-\alpha}}{\Gamma(1-\alpha)}q(0,x), \ \qquad  \forall t,x>0.
\end{equation}
Finally, we notice that $q \in \cC_{0,{\rm loc}}(\R_0^+)$ by Proposition \ref{prop:regq}.
\end{proof}
\begin{rmk}
Actually, the first equality in \eqref{eq:BSUnder} guarantees that $(\partial_t-G)^\alpha q$ is continuous in both variables.
\end{rmk}
\subsubsection{Uniqueness of the solution}
In order to prove that \eqref{eq:BSUnder} has a unique solution, we need to consider an auxiliary problem. For $\beta \ge 0$ consider the function $f_\beta(x)=e^{\beta|\log(x)|}$ and, by Item $(4)$ of Lemma \ref{lem:BSuniqueness},
\begin{align*}
    u_\beta(t,x)=P_tf_\beta(x)&=e^{\beta \log(x)+\frac{\beta(\beta-1)}{2}t}\Phi\left(-\left(\beta-\frac{1}{2}\right)\sqrt{t}-\frac{\log(x)}{\sqrt{t}}\right)\\
    &\qquad +e^{-\beta \log(x)+\frac{\beta(\beta+1)}{2}t}\Phi\left(-\left(\beta+\frac{1}{2}\right)\sqrt{t}+\frac{\log(x)}{\sqrt{t}}\right).
\end{align*}
Let us give some preliminary estimates on $u_\beta$, that can be verified by direct evaluation.
\begin{lem}\label{lem:estubeta}
    For all $\beta \ge 0$ and $t,x>0$ we have:
    \begin{align}
    \partial_x u_\beta(t,x)&=\beta e^{\frac{\beta(\beta-1)}{2}t+(\beta-1)\log(x)}\Phi\left(-\left(\beta-\frac{1}{2}\right)\sqrt{t}-\frac{\log(x)}{\sqrt{t}}\right) \nonumber\\
    &\qquad -\beta e^{\frac{\beta(\beta+1)}{2}t-(\beta+1)\log(x)}\Phi\left(-\left(\beta+\frac{1}{2}\right)\sqrt{t}+\frac{\log(x)}{\sqrt{t}}\right) \label{eq:derub1}\\
    \partial^2_x u_\beta(t,x)&=\beta(\beta-1) e^{\frac{\beta(\beta-1)}{2}t+(\beta-2)\log(x)}\Phi\left(-\left(\beta-\frac{1}{2}\right)\sqrt{t}-\frac{\log(x)}{\sqrt{t}}\right) \nonumber\\
    &\qquad +\beta(\beta+1) e^{\frac{\beta(\beta+1)}{2}t-(\beta+2)\log(x)}\Phi\left(-\left(\beta+\frac{1}{2}\right)\sqrt{t}+\frac{\log(x)}{\sqrt{t}}\right)\nonumber \\
    &\qquad -\sqrt{\frac{2 \beta^2}{\pi t}}e^{-\frac{t}{8}-\frac{3}{2}\log(x)-\frac{\log^2(x)}{2t}}\label{eq:derub2}\\
    \partial_t u_\beta(t,x)&=Gu_\beta(t,x)=\frac{\beta(\beta-1)}{2} e^{\frac{\beta(\beta-1)}{2}t+\beta\log(x)}\Phi\left(-\left(\beta-\frac{1}{2}\right)\sqrt{t}-\frac{\log(x)}{\sqrt{t}}\right) \nonumber\\
    &\qquad +\frac{\beta(\beta+1)}{2} e^{\frac{\beta(\beta+1)}{2}t-\beta\log(x)}\Phi\left(-\left(\beta+\frac{1}{2}\right)\sqrt{t}+\frac{\log(x)}{\sqrt{t}}\right)\nonumber \\
    &\qquad -\frac{\beta}{\sqrt{2\pi t}}e^{-\frac{t}{8}+\frac{\log(x)}{2}-\frac{\log^2(x)}{2t}}\label{eq:derub3}.
    \end{align}
In particular,
\begin{align}
|u_\beta(t,x)| &\le 2e^{\frac{\beta(\beta+1)}{2}t+\beta|\log(x)|} \label{eq:derub3.5}\\
|\partial_x u_\beta(t,x)| &\le 2 \beta e^{\frac{\beta(\beta+1)}{2}t+(\beta+1)|\log(x)|} \label{eq:derub4}\\
|\partial^2_x u_\beta(t,x)| &\le 2 \beta(\beta+1) e^{\frac{\beta(\beta+1)}{2}t+(\beta+2)|\log(x)|}+\frac{\beta \sqrt{2}}{\sqrt{\pi t}}e^{\frac{3}{2}|\log(x)|} \label{eq:derub5}\\
|\partial_t u_\beta(t,x)|&=|G u_\beta(t,x)| \le \beta(\beta+1) e^{\frac{\beta(\beta+1)}{2}t+\beta|\log(x)|}+\frac{\beta}{\sqrt{2\pi t}}e^{\frac{|\log(x)|}{2}}. \label{eq:derub6}
\end{align}
\end{lem}
Now consider the function
\begin{align}\label{eq:qbeta}
    q_\beta(t,x)=\int_0^t u_\beta(s,x)g_H(s;t)\, ds=\int_0^t u_\beta(s,x)\frac{s^{\alpha-1}(t-s)^{-\alpha}}{\Gamma(\alpha)\Gamma(1-\alpha)}\, ds
    =\int_0^1 u_\beta(st,x)\frac{s^{\alpha-1}(1-s)^{-\alpha}}{\Gamma(\alpha)\Gamma(1-\alpha)}\, ds.
\end{align}
Similarly to what we did for $q$, we need to prove some regularity results for $q_\beta$. We start with the following proposition, concerning the regularity in the $t$ variable.
\begin{prop}\label{prop:qbetatimereg}
Let ${\phi(\lambda)}=\lambda^\alpha$ for some $\alpha \in (0,1)$ and $\beta \ge 0$. Then for all $x>0$ the function $q_\beta(\cdot,x)$ belongs to ${\rm AC}(\R_0^+)$ and its a.e. derivative $\partial_tq_\beta(\cdot,x)$ satisfies
\begin{equation}\label{eq:qbetaderest1}
    |\partial_t q_\beta(t,x)| \le \alpha\beta(\beta+1) e^{\frac{\beta(\beta+1)}{2}(t+1)+\beta|\log(x)|}+\frac{2\beta\Gamma\left(\frac{1}{2}+\alpha\right)}{\pi\sqrt{t}\Gamma(\alpha)}e^{\frac{|\log(x)|}{2}}.
\end{equation}
In particular, for all $t,x>0$
\begin{equation*}
    \frac{\alpha}{\Gamma(1-\alpha)}\int_0^t|P_sq_\beta(t-s,x)-P_sq_\beta(t,x)|s^{-1-\alpha}\, ds <\infty
\end{equation*}
and for all $x>0$ and $\lambda>\max\left\{\beta(\beta+1), \frac{3}{4}\right\}$
\begin{equation}\label{eq:LaptransPsdiff}
    \frac{\alpha}{\Gamma(1-\alpha)}\int_0^{+\infty}\int_0^te^{-\lambda t}|P_sq_\beta(t-s,x)-P_sq_\beta(t,x)|s^{-1-\alpha}\, ds\, dt <\infty.
\end{equation}
\end{prop}
\begin{proof}
Arguing as in Proposition \ref{prop:ACqstar}, let $t>0$ and $h \in \left(-\frac{t}{2},1\right)$. We have
\begin{equation}\label{eq:incratiob}
    \left|\frac{q_\beta(t+h,x)-q_\beta(t,x)}{h}\right| \le \int_0^1\left|\frac{u_\beta(st+sh,x)-u_\beta(st,x)}{sh}\right|\frac{s^{\alpha}(1-s)^{-\alpha}}{\Gamma(\alpha)\Gamma(1-\alpha)}\, ds.
\end{equation}
By Lagrange's theorem we know that there exists $\xi \in [\min\{ts,ts+hs\}, \max\{ts,ts+hs\}] \subseteq \left[\frac{ts}{2},(t+1)s\right] $ such that
\begin{align*}
\left|\frac{u_\beta(st+sh,x)-u_\beta(st,x)}{sh}\right|&=|\partial_t u_\beta(\xi,x)| \\
&\le \beta(\beta+1) e^{\frac{\beta(\beta+1)}{2}\xi+\beta|\log(x)|}+\frac{\beta}{\sqrt{2\pi \xi}}e^{\frac{|\log(x)|}{2}} \\
&\le \beta(\beta+1) e^{\frac{\beta(\beta+1)}{2}(t+1)s+\beta|\log(x)|}+\frac{\beta}{\sqrt{\pi ts}}e^{\frac{|\log(x)|}{2}},
\end{align*}
where we used \eqref{eq:derub6}. Using this bound into \eqref{eq:incratiob} we get
\begin{align*}
    &\left|\frac{q_\beta(t+h,x)-q_\beta(t,x)}{h}\right| \\
    &\qquad \le \beta(\beta+1) e^{\frac{\beta(\beta+1)}{2}(t+1)+\beta|\log(x)|}\int_0^1\frac{s^{\alpha}(1-s)^{-\alpha}}{\Gamma(\alpha)\Gamma(1-\alpha)}\, ds\\
    &\qquad \qquad +\frac{\beta}{\sqrt{\pi t}}e^{\frac{|\log(x)|}{2}}\int_0^1\frac{s^{\alpha-\frac{1}{2}}(1-s)^{-\alpha}}{\Gamma(\alpha)\Gamma(1-\alpha)}\, ds\\
    &\qquad =\alpha\beta(\beta+1) e^{\frac{\beta(\beta+1)}{2}(t+1)+\beta|\log(x)|}+\frac{2\beta\Gamma\left(\frac{1}{2}+\alpha\right)}{\pi\sqrt{t}\Gamma(\alpha)}e^{\frac{|\log(x)|}{2}}.
\end{align*}
Arguing as in Proposition \ref{prop:AC}, this implies that $q_\beta(\cdot,x)$ belongs to ${\rm AC}(\R_0^+)$ and that its a.e. derivative $\partial_t q_\beta(t,x)$ satisfies \eqref{eq:qbetaderest1}. 

Next, observe that
\begin{align*}
    |q_\beta(t-s,x)-q_\beta(t,x)| &\le \int_0^s |\partial_t q(t-\tau,x)|\, d\tau\\
    &\le \alpha\beta(\beta+1)e^{\frac{\beta(\beta+1)}{2}(t+1)+\beta|\log(x)|}s+\frac{2s\beta\Gamma\left(\frac{1}{2}+\alpha\right)}{\pi\sqrt{t-s}\Gamma(\alpha)}e^{\frac{|\log(x)|}{2}}.
\end{align*}
Setting $\widetilde{\beta}=\max\left\{\frac{1}{2},\beta\right\}$, we have that $|q_\beta(t-s,\cdot)-q_\beta(t,\cdot)| \in \cC_0(\widetilde{\beta})$ with
\begin{align*}
    [q_\beta(t-s,\cdot)-q_\beta(t,\cdot)]_{\widetilde{\beta}} \le \alpha\beta(\beta+1)e^{\frac{\beta(\beta+1)}{2}(t+1)}s+\frac{2s\beta\Gamma\left(\frac{1}{2}+\alpha\right)}{\pi\sqrt{t-s}\Gamma(\alpha)}
\end{align*}
and then, by Item $(5)$ of Lemma \ref{lem:BSuniqueness}
\begin{multline}\label{eq:estPsqbeta}
    P_s|q_\beta(t-s,\cdot)-q_\beta(t,\cdot)| \\ \le \left(2\alpha\beta(\beta+1)e^{\widetilde{\beta}(\widetilde{\beta}+1)(t+1)}s+\frac{4s\beta\Gamma\left(\frac{1}{2}+\alpha\right)}{\pi\sqrt{t-s}\Gamma(\alpha)}e^{\frac{\widetilde{\beta}(\widetilde{\beta}+1)}{2}t}\right)e^{\widetilde{\beta}|\log(x)|}.
\end{multline}
Thus we have
\begin{multline*}
    \frac{\alpha}{\Gamma(1-\alpha)}\int_0^t|P_sq_\beta(t-s,x)-P_sq_\beta(t,x)| s^{-1-\alpha}\, ds \\
    \le \left(\frac{2\alpha^2\beta(\beta+1)}{\Gamma(2-\alpha)}e^{\widetilde{\beta}(\widetilde{\beta}+1)(t+1)}t^{1-\alpha}+\frac{4\alpha \beta\Gamma\left(\frac{1}{2}+\alpha\right)}{\sqrt{\pi}\Gamma(\alpha)\Gamma\left(\frac{3}{2}-\alpha\right)}e^{\frac{\widetilde{\beta}(\widetilde{\beta}+1)}{2}t}t^{\frac{1}{2}-\alpha}\right)e^{\widetilde{\beta}|\log(x)|}
\end{multline*}
that also implies \eqref{eq:LaptransPsdiff} with $\lambda>\widetilde{\beta}(\widetilde{\beta}+1)$.
\end{proof}
Next, we need to study the regularity of $q_\beta$ in the $x$ variable. Let us first focus on the first derivative.
\begin{prop}\label{prop:qbetareg2}
   For all $t,x>0$, we have
    \begin{equation}\label{eq:derinside}
        \partial_x q_\beta(t,x)=\frac{1}{\Gamma(\alpha)\Gamma(1-\alpha)}\int_0^1 \partial_x u_\beta(t\tau,x)\tau^{\alpha-1}(1-\tau)^{-\alpha}d\tau.
    \end{equation}
    Furthermore, $q_\beta \in \cC_0(\beta)$ and $\partial_x q_\beta \in \cC_0(\beta+1)$, with
    \begin{align}
        [q_\beta(t,\cdot)]_\beta \le 2e^{\frac{\beta(\beta+1)}{2}t} \label{eq:qbetaest1}\\
        [\partial_xq_\beta(t,\cdot)]_\beta \le 2\beta e^{\frac{\beta(\beta+1)}{2}t}. \label{eq:qbetaest2}
    \end{align}
\end{prop}
\begin{proof}
    The identity \eqref{eq:derinside} follows by \eqref{eq:qbeta} by a simple application of the dominated convergence theorem, which is justified by \eqref{eq:derub4}. Furthermore, \eqref{eq:derub3.5} and \eqref{eq:derub4} imply respectively \eqref{eq:qbetaest1} and \eqref{eq:qbetaest2}.
\end{proof}
For the second derivative, we need a further assumption.
\begin{prop}\label{prop:qbetareg3}
    Let $\alpha \in \left(\frac{1}{2},1\right)$ and $\beta \ge 0$. Then, for all $t,x>0$, we have
    \begin{align}\label{eq:derinside2}
        \partial^2_x q_\beta(t,x)=\frac{1}{\Gamma(\alpha)\Gamma(1-\alpha)}\int_0^1 \partial^2_x u_\beta(t\tau,x)\tau^{\alpha-1}(1-\tau)^{-\alpha}d\tau.\\
        G q_\beta(t,x)=\frac{1}{\Gamma(\alpha)\Gamma(1-\alpha)}\int_0^1 G u_\beta(t\tau,x)\tau^{\alpha-1}(1-\tau)^{-\alpha}d\tau. \label{eq:opinside}
    \end{align}
    Moreover, $Gq_\beta \in \cC_0(\widetilde{\beta})$ with $\widetilde{\beta}=\max\left\{\beta, \frac{1}{2}\right\}$ and
    \begin{align}
        [Gq_\beta(t,\cdot)]_{\widetilde{\beta}} \le \beta(\beta+1)e^{\frac{\beta(\beta+1)}{2}t}+\frac{\beta\Gamma\left(\alpha-\frac{1}{2}\right)\Gamma(1-\alpha)}{\pi\sqrt{2t}}.
        \label{eq:qbetaest3}
    \end{align}
    As a consequence, for all $t,x>0$,
    \begin{equation*}
        \frac{\alpha}{\Gamma(1-\alpha)}\int_0^t |P_sq(t,x)-q(t,x)|s^{-1-\alpha}\, ds < \infty,
    \end{equation*}
    and for all $x>0$ and $\lambda>\widetilde{\beta}(\widetilde{\beta}+1)$,
    \begin{equation}\label{eq:LapestGqbeta1}
        \frac{\alpha}{\Gamma(1-\alpha)}\int_0^\infty \int_0^t e^{-\lambda t}|P_sq(t,x)-q(t,x)|s^{-1-\alpha}\, ds\, dt < \infty.
    \end{equation}
\end{prop}
\begin{proof}
    Fix $x,t>0$ and let $0<a<x<b$. For $y \in [a,b]$ and $\tau \in [0,1]$, by \eqref{eq:derub5}, we have
    \begin{equation*}
        |\partial_x^2 u_\beta(t\tau,y)| \le 2\beta(\beta+1)e^{\frac{\beta(\beta+1)}{2}t+(\beta+2)M}+\frac{\beta \sqrt{2}}{\sqrt{\pi t\tau}}e^{\frac{3M}{2}},
    \end{equation*}
    where $M=\max\{|\log(a)|, |\log(b)|\}$. Notice that the right-hand side is integrable against $g_H(\tau,1)$. Indeed,
    \begin{align*}
        \int_0^1 &\left(2\beta(\beta+1)e^{\frac{\beta(\beta+1)}{2}t+(\beta+2)M}+\frac{\beta}{\sqrt{2\pi t\tau}}e^{\frac{M}{2}}\right)\tau^{\alpha-1}(1-\tau)^{-\alpha}\, d\tau\\
        &=2\beta(\beta+1)e^{\frac{\beta(\beta+1)}{2}t+(\beta+2)M}+\frac{\beta\Gamma\left(\alpha-\frac{1}{2}\right)\Gamma(1-\alpha)\sqrt{2}}{\pi\sqrt{t}}e^{\frac{3M}{2}}.
    \end{align*}
    Hence, an application of the dominated convergence theorem to \eqref{eq:derinside} leads to \eqref{eq:derinside2}. Furthermore, multiplying both sides of \eqref{eq:derinside2} by $\frac{x^2}{2}$ we get \eqref{eq:opinside} and \eqref{eq:qbetaest3} follows by \eqref{eq:opinside} and \eqref{eq:derub6}.

    Now observe that by Proposition \ref{prop:qbetareg3} we know that $q_\beta(t,\cdot), \partial_x q_\beta(t,\cdot) \in \cC_0$, while we also proved that $G q_\beta(t,\cdot) \in \cC_0$. Hence we can write
    \begin{equation*}
    P_sq_\beta(t,x)-q_\beta(t,x)=\int_0^s GP_\tau q_\beta(t,x)\, d\tau=\int_0^s P_\tau Gq_\beta(t,x)\, d\tau.
    \end{equation*}
    Furthermore, by \eqref{eq:qbetaest3} and Item $(5)$ of Lemma \ref{lem:BSuniqueness} we get for $\tau \in [0,s] \subseteq [0,t]$
    \begin{align}
        P_\tau |Gq_\beta(t,\cdot)| \le \left(\beta(\beta+1)e^{\widetilde{\beta}(\widetilde{\beta}+1)t}+\frac{\beta\Gamma\left(\alpha-\frac{1}{2}\right)\Gamma(1-\alpha)}{\pi\sqrt{2t}}e^{\frac{\widetilde{\beta}(\widetilde{\beta}+1)}{2}t}\right)e^{\widetilde{\beta}|\log(x)|}.
        \label{eq:PtauGbetaest}
    \end{align}
    Thus, we have
    \begin{multline}
        |P_sq_\beta(t,x)-q(t,x)| \\ \le \left(\beta(\beta+1)e^{\widetilde{\beta}(\widetilde{\beta}+1)t}+\frac{\beta\Gamma\left(\alpha-\frac{1}{2}\right)\Gamma(1-\alpha)}{\pi\sqrt{2t}}e^{\frac{\widetilde{\beta}(\widetilde{\beta}+1)}{2}t}\right)se^{\widetilde{\beta}|\log(x)|}.
        \label{eq:Ptaudiffest2}
    \end{multline}
    and then
    \begin{multline*}
        \frac{\alpha}{\Gamma(1-\alpha)}\int_0^t |P_sq_\beta(t,x)-q(t,x)|s^{-1-\alpha}\, ds \\
        \le \frac{\alpha}{\Gamma(2-\alpha)}\left(\beta(\beta+1)e^{\widetilde{\beta}(\widetilde{\beta}+1)t}+\frac{\beta\Gamma\left(\alpha-\frac{1}{2}\right)\Gamma(1-\alpha)}{\pi\sqrt{2t}}e^{\frac{\widetilde{\beta}(\widetilde{\beta}+1)}{2}t}\right)t^{1-\alpha}e^{\widetilde{\beta}|\log(x)|}.
    \end{multline*}
    Eventually, the latter inequality implies \eqref{eq:LapestGqbeta1} for $\lambda>\widetilde{\beta}(\widetilde{\beta}+1)$.
\end{proof}
Now we check that $q_\beta$ satisfies Item (iii) of Theorem \ref{thm:Lapinside2}.
\begin{prop}\label{prop:qbetareg4}
Let $\alpha \in (0,1)$ and $\beta \ge 0$. Then, for all $\lambda>\frac{\beta(\beta+1)}{2}$ and $x>0$ we have
\begin{equation*}
    \frac{1}{\Gamma(1-\alpha)}\int_0^\infty t^{-\alpha}e^{-\lambda t}|q_\beta(t,x)|\, dt<\infty.
\end{equation*}
\end{prop}
\begin{proof}
By \eqref{eq:qbetaest1} we have
\begin{align*}
    \frac{1}{\Gamma(1-\alpha)}\int_0^\infty t^{-\alpha}e^{-\lambda t}|q_\beta(t,x)|\, dt &\le \frac{2e^{\beta|\log(x)|}}{\Gamma(1-\alpha)}\int_0^\infty t^{-\alpha}e^{-\left(\lambda-\frac{\beta(\beta+1)}{2}\right)t}\, dt\\
    &=2e^{\beta|\log(x)|}\left(\lambda-\frac{\beta(\beta+1)}{2}\right)^{\alpha-1}.
\end{align*}
\end{proof}
We can now prove that $q_\beta$ is a solution of a special fully nonlocal equation. 
\begin{prop}\label{prop:testsol}
    Let $\alpha \in \left(\frac{1}{2},1\right)$ and $\beta \ge 0$. Then $q_\beta$ is a solution of
    \begin{equation}\label{eq:specialeq}
    \begin{cases}
    (\partial_t-G)^\alpha q_\beta(t,x)=\frac{t^{-\alpha}}{\Gamma(1-\alpha)}q_\beta(0,x) & t,x>0 \\
    q_\beta(0,x)=e^{\beta|\log(x)|} & x>0.
    \end{cases}
    \end{equation}
\end{prop}
\begin{proof}
First notice that $q_\beta(0,x)=u_\beta(0,x)=e^{\beta|\log(x)|}$. By Propositions \ref{prop:qbetatimereg}, \ref{prop:qbetareg3} and \ref{prop:qbetareg4} we know that $(\partial_t-G)^\alpha q_\beta(t,x)$ is well-defined for all $t,x>0$ and for fixed $x>0$ the function $(\partial_t-G)^\alpha q_\beta(\cdot,x)$ is Laplace transformable for $\lambda>\widetilde{\beta}(\widetilde{\beta}+1)$, where $\widetilde{\beta}=\max\left\{\beta,\frac{1}{2}\right\}$. Now we show that
\begin{equation}\label{eq:Laplaceequality}
    \int_0^\infty e^{-\lambda t}(\partial_t-G)^\alpha q_\beta(t,x)\, dt=\lambda^{\alpha-1}e^{\beta|\log(x)|}
\end{equation}
for $\lambda>\frac{1}{16}\left(\frac{1}{4}+\widetilde{\beta}(\widetilde{\beta}+1)\right)$. This is done exactly as in Proposition \ref{prop:qsolution}, the only difference being the bound we prove to use Fubini's theorem in \eqref{eq:preFub}. Precisely, recall that $f_\beta(x)=e^{\beta|\log(x)|}$ and set
\begin{equation*}
\overline{u}(t,x)=e^{-\frac{\log^2(x)}{2t}}f_{\widetilde{\beta}}(x)
\end{equation*}
and notice that, by \eqref{eq:derub3.5} and \eqref{eq:derub6}
\begin{equation*}
    \left| (G-\lambda) u_{\beta} (w,y) \right| \le \left((\beta(\beta+1)+2\lambda)e^{\frac{\beta(\beta+1)}{2}w}f_{\widetilde{\beta}}(y)+\frac{\beta}{\sqrt{2\pi w}}\overline{u}(w,y)\right).
\end{equation*}
Hence, we have
\begin{align}
 &   \int_0^{+\infty}\int_0^{+\infty}\int_0^{+\infty} \int_0^s \tau^{\alpha-1}s^{-\alpha-1}e^{-\lambda (\tau+w)} \left| (G-\lambda) u_{\beta} (w,y) \right| p_{\rm GBM} (\tau, y;x) \, dw \, dy \, d\tau \, ds \notag \\
 & \le \int_0^{+\infty}\int_0^{+\infty}\int_0^{s} (\beta(\beta+1)+2\lambda)e^{\frac{\beta(\beta+1)}{2}w}\tau^{\alpha-1} s^{-\alpha-1}e^{-\lambda (\tau+w)}u_{\widetilde{\beta}}(\tau,x) \, dw  \, d\tau \, ds \nonumber \\
 &\qquad +\int_0^{+\infty}\int_0^{+\infty}\int_0^{s} \frac{\beta}{\sqrt{2\pi w}}\tau^{\alpha-1} s^{-\alpha-1}e^{-\lambda (\tau+w)}P_\tau \overline{u}(w,x) \, dw  \, d\tau \, ds \nonumber \\
 &=I_1+I_2.
 \label{Fubnew1}
 \end{align}
Concerning $I_1$, we have, by Item $(5)$ of Lemma \ref{lem:BSuniqueness}
 \begin{align*}
 I_1& \le 2(\beta(\beta+1)+2\lambda)\int_0^{+\infty}\int_0^{+\infty}\int_0^{s} e^{-\left(\lambda-\frac{\widetilde{\beta}(\widetilde{\beta}+1)}{2}\right)(\tau+w)}\tau^{\alpha-1} s^{-\alpha-1} \, dw  \, d\tau \, ds\\
 &=\frac{4(\beta(\beta+1)+2\lambda)}{2\lambda-\widetilde{\beta}(\widetilde{\beta}+1)}\left(\int_0^{+\infty}\left(1-e^{-\left(\lambda-\frac{\widetilde{\beta}(\widetilde{\beta}+1)}{2}\right)s}\right)s^{-\alpha-1}\, ds\right)\\
 &\qquad \times \left(\int_0^{+\infty}e^{-\left(\lambda-\frac{\widetilde{\beta}(\widetilde{\beta}+1)}{2}\right)\tau}\tau^{\alpha-1} \, d\tau\right)\\
 &=\frac{4(\beta(\beta+1)+2\lambda)\Gamma(1-\alpha)\Gamma(\alpha)}{2\lambda-\widetilde{\beta}(\widetilde{\beta}+1)}.
 \end{align*}
For $I_2$, we first need to evaluate $P_\tau \overline{u}(w,x)$. We have
\begin{align*}
    P_\tau \overline{u}(w,x)&=\int_0^{+\infty} e^{-\frac{\log^2(y)}{2w}+\widetilde{\beta}|\log(y)|-\frac{\tau}{8}-\frac{(\log(y)-\log(x))^2}{2\tau}}\sqrt{\frac{x}{y^3 2\pi \tau}}\, dy\\
    &=\sqrt{x}e^{-\frac{\tau}{8}-\frac{\log^2(x)}{2\tau}}\int_0^{+\infty} e^{-\frac{\log^2(y)}{2w}-\frac{\log^2(y)}{2\tau}+\frac{\log(y)\log(x)}{\tau}+\widetilde{\beta}|\log(y)|}\frac{1}{\sqrt{y^3 2\pi \tau}}\, dy.
\end{align*}
Now let $z=\log(y)$ to achieve
\begin{align*}
    P_\tau \overline{u}(w,x)&=\sqrt{\frac{x}{2\pi \tau}}e^{-\frac{\tau}{8}-\frac{\log^2(x)}{2\tau}}\int_{-\infty}^{+\infty} e^{-\frac{z^2}{2w}-\frac{z^2}{2\tau}+\frac{\log(x)}{\tau}z-\frac{z}{2}+\widetilde{\beta}|z|}\, dz.
\end{align*}
Define
\begin{equation}
    \mathcal{H}(\tau,w)=\left(\frac{1}{w}+\frac{1}{\tau}\right)^{-1}
\end{equation}
so that
\begin{align*}
    P_\tau \overline{u}(w,x)&=\sqrt{\frac{x \mathcal{H}(\tau,w)}{\tau}}e^{-\frac{\tau}{8}-\frac{\log^2(x)}{2\tau}}\int_{-\infty}^{+\infty} \frac{1}{\sqrt{2\pi \mathcal{H}(\tau,w)}}e^{-\frac{z^2}{2\mathcal{H}(\tau,w)}+\left(\frac{\log(x)}{\tau}-\frac{1}{2}\right)z+\widetilde{\beta}|z|}\, dz\\
    &=\sqrt{\frac{x \mathcal{H}(\tau,w)}{\tau}}e^{-\frac{\tau}{8}-\frac{\log^2(x)}{2\tau}}\E^{(0,0)}[e^{\widetilde{\beta} |B(\mathcal{H}(\tau,w))|+\left(\frac{\log(x)}{\tau}-\frac{1}{2}\right)B(\mathcal{H}(\tau,w))}]\\
    &=\sqrt{\frac{x \mathcal{H}(\tau,w)}{\tau}}e^{-\frac{\tau}{8}-\frac{\log^2(x)}{2\tau}}e^{\frac{1}{2}\left(\widetilde{\beta}+\frac{\log(x)}{\tau}-\frac{1}{2}\right)^2\mathcal{H}(\tau,w)}\Phi\left(-\left(\widetilde{\beta}+\frac{\log(x)}{\tau}-\frac{1}{2}\right)\sqrt{t}\right)\\
    &\qquad + \sqrt{\frac{x \mathcal{H}(\tau,w)}{\tau}}e^{-\frac{\tau}{8}-\frac{\log^2(x)}{2\tau}}e^{\frac{1}{2}\left(\widetilde{\beta}-\frac{\log(x)}{\tau}+\frac{1}{2}\right)^2\mathcal{H}(\tau,w)}\Phi\left(-\left(\widetilde{\beta}-\frac{\log(x)}{\tau}+\frac{1}{2}\right)\sqrt{t}\right)\\
    &=\sqrt{\frac{x \mathcal{H}(\tau,w)}{\tau}}e^{-\frac{\tau}{8}-\frac{\log^2(x)}{2(\tau+w)}+A_1\mathcal{H}(\tau,w)+\left(\widetilde{\beta}-\frac{1}{2}\right)\frac{w}{\tau+w}\log(x)}\Phi\left(-A_2\sqrt{t}\right)\\
    &\qquad + \sqrt{\frac{x \mathcal{H}(\tau,w)}{\tau}}e^{-\frac{\tau}{8}-\frac{\log^2(x)}{2(\tau+w)}+B_1\mathcal{H}(\tau,w)-\left(\widetilde{\beta}+\frac{1}{2}\right)\frac{w}{\tau+w}\log(x)}\Phi\left(-B_2\sqrt{t}\right),
\end{align*}
where
\begin{align*}
A_1=\frac{\widetilde{\beta}^2}{2}+\frac{1}{8}-\frac{\widetilde{\beta}}{2} && A_2=\widetilde{\beta}+\frac{\log(x)}{\tau}-\frac{1}{2}\\
B_1=\frac{\widetilde{\beta}^2}{2}+\frac{1}{8}+\frac{\widetilde{\beta}}{2} && B_2=\beta-\frac{\log(x)}{\tau}+\frac{1}{2}
\end{align*}
and we used Lemma \ref{lem:momgen}. Furthermore, we have
\begin{align*}
    \frac{\beta}{\sqrt{2\pi w}}P_\tau \overline{u}(w,x)&=\beta\sqrt{\frac{x}{2\pi(\tau+w)}}e^{-\frac{\tau}{8}-\frac{\log^2(x)}{2(\tau+w)}+A_1\mathcal{H}(\tau,w)+\left(\widetilde{\beta}-\frac{1}{2}\right)\frac{w}{\tau+w}\log(x)}\Phi\left(-A_2\sqrt{t}\right)\\
    &\qquad +2\beta\sqrt{\frac{x}{2\pi(\tau+w)}}e^{-\frac{\tau}{8}-\frac{\log^2(x)}{2(\tau+w)}+B_1\mathcal{H}(\tau,w)-\left(\widetilde{\beta}+\frac{1}{2}\right)\frac{w}{\tau+w}\log(x)}\Phi\left(-B_2\sqrt{t}\right)\\
    &\le \frac{\beta}{\sqrt{2\pi(\tau+w)}}e^{-\frac{\tau}{8}-\frac{\log^2(x)}{2(\tau+w)}+B_1\mathcal{H}(\tau,w)+\left(\widetilde{\beta}+1\right)|\log(x)|},
\end{align*}
where we used $B_1>A_1$ and $\frac{w}{\tau+w} \le 1$. Now recall that the function $r \in \R^+ \mapsto \sqrt{r}e^{-\frac{\log^2(x)}{2r}} \in \R^+$ admits as maximum point $r_\star=\log^{-2}(x)$, so that
\begin{align*}
    \frac{\beta}{\sqrt{2\pi w}}P_\tau \overline{u}(w,x)&\le \frac{\beta}{\sqrt{2\pi e}}|\log^{-1}(x)|e^{-\frac{\tau}{8}+B_1\mathcal{H}(\tau,w)+\left(\widetilde{\beta}+1\right)|\log(x)|}.
\end{align*}
Next, we observe that $2\mathcal{H}(\tau,w)$ is the harmonic mean of $\tau$ and $w$, hence, we can use the fact that the harmonic mean is smaller than or equal to the arithmetic mean to write
\begin{equation*}
    \mathcal{H}(\tau,w) \le \frac{\tau+w}{4}
\end{equation*}
and thus we get 
\begin{align*}
    \frac{\beta}{\sqrt{2\pi w}}P_\tau \overline{u}(w,x)&\le \frac{\beta}{\sqrt{2\pi e}}|\log^{-1}(x)|e^{\frac{B_1}{4}(\tau+w)+\left(\widetilde{\beta}+1\right)|\log(x)|}.
\end{align*}
Now we can go back to $I_2$ to get 
\begin{align*}
    I_2 &\le \frac{\beta}{\sqrt{2\pi e}}|\log^{-1}(x)|e^{\left(\widetilde{\beta}+1\right)|\log(x)|}\int_0^{+\infty}\int_0^{+\infty}\int_0^{s} \tau^{\alpha-1} s^{-\alpha-1}e^{-\left(\lambda-\frac{B_1}{4}\right) (\tau+w)} \, dw  \, d\tau \, ds\\
    &=\frac{\beta}{\left(\lambda-\frac{B_1}{4}\right)\sqrt{2\pi e}}|\log^{-1}(x)|e^{\left(\widetilde{\beta}+1\right)|\log(x)|}\left(\int_0^{+\infty}s^{-\alpha-1}\left(1-e^{-\left(\lambda-\frac{B_1}{4}\right)s}\right)\, ds\right)\\
    &\qquad \times \left(\int_0^{+\infty} \tau^{\alpha-1}e^{-\left(\lambda-\frac{B_1}{4}\right)\tau} \, d\tau \right)\\
    &=\frac{\beta\Gamma(1-\alpha)\Gamma(\alpha)}{\left(\lambda-\frac{B_1}{4}\right)\sqrt{2\pi e}}|\log^{-1}(x)|e^{\left(\widetilde{\beta}+1\right)|\log(x)|}.
\end{align*}
 Once we have \eqref{eq:Laplaceequality}, we know that for all $x>0$ there exists $\mathcal{I}_x \subset \R^+$ such that $|\mathcal{I}_x|=0$ and for all $t \in \R^+ \setminus \mathcal{I}_x$ we have
\begin{equation}\label{eq:equalitysol}
    (\partial_t-G)^\alpha q_\beta(t,x)=\frac{t^{-\alpha}}{\Gamma(1-\alpha)}q_\beta(0,x).
\end{equation}
Analogously to what we did in Proposition \ref{prop:qsolution}, we need to prove that $\mathcal{I}_x=\varnothing$. First, notice that a simple application of the dominated convergence theorem to \eqref{eq:qbeta}, justified by \eqref{eq:derub3.5}, shows that $q_\beta$ is continuous. Furthermore, arguing exactly as in Proposition \ref{prop:qsolution}, by means of the estimates \eqref{eq:estPsqbeta} and \eqref{eq:Ptaudiffest2}, one can show that if $t_n \downarrow t$, then $(\partial_t-G)^\alpha q_\beta(t_n,x) \to (\partial_t-G)^\alpha q_\beta(t,x)$. This shows, with the same argument as in Proposition \ref{prop:qsolution}, that $\mathcal{I}_x=\varnothing$ for all $x>0$ and then $q_\beta$ is solution of \eqref{eq:specialeq}.
\end{proof}
Now we want to use the auxiliary function $q_\beta$ to show that \eqref{eq:BSUnder} has a unique solution. To do this, we need also the following lower bound.
\begin{prop}\label{prop:lobound}
    Let $\alpha \in (0,1)$ and $\beta \ge 0$. Then there exists a non-increasing function $C_\beta:\R^+ \to \R$ with $C_\beta(t)>0$ for all $t>0$ and such that
    \begin{equation*}
        q_\beta(t,x) \ge C_\beta(t)e^{\beta|\log(x)|}.
    \end{equation*}
\end{prop}
\begin{proof}
    First of all, let us observe that if $x \ge 1$ then $\log(x) \ge 0$ and we have
    \begin{equation*}
    u_\beta(t,x) \ge e^{\beta \log(x)+\frac{\beta(\beta+1)}{2}t}\Phi\left(-\left(\beta+\frac{1}{2}\right)\sqrt{t}\right) \ge e^{\beta \log(x)+\frac{\beta(\beta-1)}{2}t}\Phi\left(-\left(\beta-\frac{1}{2}\right)\sqrt{t}\right).
    \end{equation*}
    If instead $x \in (0,1]$, then $\log(x) \le 0$ and
    \begin{equation*}
    u_\beta(t,x) \ge e^{-\beta \log(x)+\frac{\beta(\beta-1)}{2}t}\Phi\left(-\left(\beta-\frac{1}{2}\right)\sqrt{t}\right).
    \end{equation*}
    Hence, in general, we have
    \begin{equation*}
    u_\beta(t,x) \ge e^{\beta |\log(x)|+\frac{\beta(\beta-1)}{2}t}\Phi\left(-\left(\beta-\frac{1}{2}\right)\sqrt{t}\right).
    \end{equation*}
    Next, by \eqref{eq:qbeta}, we get
    \begin{align*}
        q_\beta(t,x)& \ge e^{\beta|\log(x)|}\int_0^1e^{\frac{\beta(\beta-1)}{2}ts}\Phi\left(-\left(\beta-\frac{1}{2}\right)\sqrt{ts}\right)\frac{s^{\alpha-1}(1-s)^{-\alpha}}{\Gamma(\alpha)\Gamma(1-\alpha)} \\
        &\ge \begin{cases} \frac{ e^{\beta|\log(x)|+\frac{\beta(\beta-1)}{2}t}}{2} & \beta \le \frac{1}{2} \\ e^{\beta|\log(x)|+\frac{\beta(\beta-1)}{2}t}\Phi\left(-\left(\beta-\frac{1}{2}\right)\sqrt{t}\right) & \beta \in \left(\frac{1}{2},1\right]\\
        e^{\beta|\log(x)|}\Phi\left(-\left(\beta-\frac{1}{2}\right)\sqrt{t}\right) & \beta > 1,
        \end{cases}
    \end{align*}
    that ends the proof.
\end{proof}
Now we shall show that the solutions of the fully {nonlocal} Black-Scholes equation belonging to $\cC_{0, {\rm loc}}(\R_0^+)$ are unique. 
\begin{prop}\label{prop:uniquenessBSnonloc}
    Let $\alpha \in \left(\frac{1}{2},1\right)$ and $f \in \cC_0$. Then, the equation
    \begin{equation}\label{eq:BSnonlocgen}
    \begin{cases}
    (\partial_t-G)^\alpha q_f(t,x)=\frac{t^{-\alpha}}{\Gamma(1-\alpha)}q_f(0,x) & t,x>0 \\
    q_f(0,x)=f(x) & x>0\\
    q_f \in \cC_{0,{\rm loc}}(\R_0^+)
    \end{cases}
    \end{equation}
    at most has one solution.
\end{prop}
\begin{proof}
    Since the equation is linear, it is sufficient to verify that if $f \equiv 0$ then the solution $q_f \equiv 0$. Fix $T>0$ and recall that since $q_f \in \cC_{0,{\rm loc}}(\R_0^+)$ then
    \begin{equation*}
        |q_f(t,x)| \le C_Te^{\beta_T|\log(x)|}.
    \end{equation*}
    For simplicity, set $C_T=C$ and $\beta_T:=\beta_0$. Now, for any $A \ge 1$ and $\beta>\beta_0$, consider the function
    \begin{equation*}
        w_\beta(t,x;A):=\frac{C}{C_\beta(T)}e^{(\beta_0-\beta)|\log(A)|}q_\beta(t,x).
    \end{equation*}
    Notice that $w_\beta(0,x;A)>0=q_f(0,x)$ for all $x>0$. Next, by Proposition \ref{prop:lobound} we have that, for $t \in (0,T]$,
    \begin{equation*}
        w_\beta(t,x;A) \ge \frac{C_\beta(t)}{C_\beta(T)}Ce^{(\beta_0-\beta)|\log(A)|+\beta|\log(x)|} \ge Ce^{(\beta_0-\beta)|\log(A)|+\beta|\log(x)|}.
    \end{equation*}
    If $x \not \in \left[A^{-1},A\right]$ then $|\log(x)| \ge |\log(A)|$ and we have
    \begin{equation*}
        w_\beta(t,x;A) \ge Ce^{\beta_0|\log(A)|} \ge |q_f(t,x)|.
    \end{equation*}
    This shows that for all $x \not \in \left[A^{-1},A\right]$ and $t \in [0,T]$
    \begin{equation}\label{eq:inequalitytoextend}
        w_\beta(t,x;A)+q_f(t,x) \le 0 \le w_\beta(t,x;A)-q_f(t,x).
    \end{equation}
    We want to show that this inequality holds for all $x>0$. To do this, assume by contradiction that there exists a point $(t_\star,x_\star) \in [0,T] \times \R^+$ such that 
    \begin{equation*}
    w_\beta(t_\star,x_\star;A)-q_\beta(t_\star,x_\star)<0.
    \end{equation*}
    It must hold $x_\star \in [A^{-1},A]$. Furthermore, since all the involved functions are continuous, we can assume without loss of generality that
    \begin{equation*}
        w_\beta(t_\star,x_\star;A)-q_f(t_\star,x_\star)=\min_{(t,x) \in [0,T] \times [A^{-1},A]}(w_\beta(t,x;A)-q_f(t,x)).
    \end{equation*}
    First of all, notice that for all $x \in \R^+$
    \begin{equation*}
        w_\beta(t_\star,x_\star;A)-q_f(t_\star,x_\star)<0<w_\beta(0,x;A)-q_f(0,x).
    \end{equation*}
    Moreover, using \eqref{eq:inequalitytoextend} for $x \not \in [A^{-1},A]$ {and} the definition of $(t_\star,x_\star)$, if $x \in [A^{-1},A]$, we get
    \begin{equation*}
        w_\beta(t_\star,x_\star;A)-q_f(t_\star,x_\star) \le w_\beta(t,x;A)-q_f(t,x)
    \end{equation*}
    for all $t \in [0,T]$ and $x>0$. By Proposition \ref{prop:maxprin} we must have
    \begin{equation*}
    -(\partial_t-G)^\alpha (w_\beta(\cdot,\cdot;A)-q_f(\cdot,\cdot))(x_\star,t_\star)+\frac{t_\star^{-\alpha}}{\Gamma(1-\alpha)}w_\beta(0,x_\star;A)>0.
    \end{equation*}
    However, by Proposition \ref{prop:testsol} we know that, up to a multiplicative constant, $w_\beta(\cdot,\cdot;A)$ solves \eqref{eq:specialeq}, while $q_f$ solves \eqref{eq:BSnonlocgen}, hence
    \begin{equation*}
    -(\partial_t-G)^\alpha (w_\beta(\cdot,\cdot;A)-q_f(\cdot,\cdot))(x_\star,t_\star)+\frac{t_\star^{-\alpha}}{\Gamma(1-\alpha)}w_\beta(0,x_\star;A)=0,
    \end{equation*}
    {which is a contradiction}. Thus, we find
    \begin{equation*}
        w_\beta(t,x;A)-q_f(t,x) \ge 0, \ \forall x>0, \ t \in [0,T]
    \end{equation*}
    With a similar argument one can prove that
    \begin{equation*}
        w_\beta(t,x;A)+q_f(t,x) \le 0, \ \forall x>0, \ t \in [0,T]
    \end{equation*}
    and then, since \eqref{eq:inequalitytoextend} holds for all $x>0$ and $t \in [0,T]$, we have
    \begin{equation*}
    |q_f(t,x)| \le w_\beta(t,x;A)=\frac{C}{C_\beta(T)}e^{(\beta_0-\beta)|\log(A)|}q_\beta(t,x)
    \end{equation*}
    taking the limit as $A \to \infty$ we get that $|q_f(t,x)|=0$ for all $x>0$ and $t \in [0,T]$. Since $T>0$ is arbitrary, we end the proof.
\end{proof}
\subsection{Proof of Theorem \ref{thm:main2}}
Before proving the main theorem of the section, we must {state} the last preliminary result whose proof is in Appendix \ref{appendix628}.
\begin{prop}\label{prop:generator1}
    Let ${\phi(\lambda)}=\lambda^{\alpha}$ for some $\alpha \in (0,1)$. Then, on $f \in C_c^\infty(\R\times \R^+)$ the generator $A^\phi$ of $(B^\phi,{S})$ acts as follows:
    \begin{align*}
    A^\phi f(x,v)&=\int_{\R}\int_{\R}(f(x+h,s+v)-f(x,v)-h\partial_xf(x,v)\mathds{1}_{[-1,1]}(h))\mathcal{K}(dh;ds) \, dh,
\end{align*}
where $\mathcal{K}$ is defined in \eqref{eq:K}. 
\end{prop}
Now we are finally ready to prove Theorem \ref{thm:main2}.
\begin{proof}[Proof of Theorem \ref{thm:main2}]
We already know by \eqref{prezzofinbsnonloca} that $q_\star(T-t,x,0)=q(t,x)$. Furthermore, by Proposition \ref{prop:qsolution} we know that the function $q$ solves \eqref{eq:BSUnder} and by Proposition \ref{prop:uniquenessBSnonloc} that it is actually the unique solution. We only need to prove \eqref{eq:possoj}. To do this, denote by $(\mathcal{Q}_t)_{t \ge 0}$ the transition semigroup of $(X_{\rm e}(t),\gamma(t))$ and observe that
\begin{align*}
q_\star(t,X(t),\gamma(t))&=\widetilde{Q}_{T-t}u(X(t),\gamma(t))\\
&=\widetilde{\E}^{(x,v)}[(X(T)-K)_+ \mid X(t), \gamma(t)]\\
&=\E^{(x,v)}[(e^{X_{\rm e}(T)-\frac{T-v-\gamma(T)}{2}}-K)_+ \mid X(t), \gamma(t)]\\
&=\mathcal{Q}_{T-t}g_{T,v}(X_{\rm e}(t),\gamma(t))=\mathcal{Q}_{T-t}g_{T,v}(\log(X(t)),\gamma(t)),
\end{align*}
where
\begin{equation*}
    g_{T,v}(x,w)=(e^{x-\frac{T-v-w}{2}}-K)_+.
\end{equation*}
By Proposition \ref{prop:generator1} and \cite[Theorem 4.1]{meerschaert2014semi}, we have, for $w>0$
\begin{equation*}
    \mathcal{Q}_{t}g(x,w)=g(x,w+t)\mathcal{K}_w(x;\R^d \times [w+t,\infty))+ \int_{\R \times [w,w+t)}\mathcal{Q}_{t+w-\tau}g(x+y,0)\mathcal{K}_w(d\tau,dy),
\end{equation*}
In our case, this becomes
\begin{multline}\label{eq:presubstitut}
    q_\star(t,x,w)=\mathcal{Q}_{T-t}g_{T,v}(\log(x),w)=g_{T,v}(\log(x),w+T-t)\mathcal{K}_w(x;\R^d \times [w+T-t,\infty))\\
    +\int_{\R \times [w,w+T-t)}\mathcal{Q}_{T-t+w-\tau}g_{T,v}(\log(x)+y,0)\mathcal{K}_w(d\tau,dy).
\end{multline}
{Now, we just have} to evaluate the inner term. We have, setting for {brevity} $\widetilde{x}=\log(x)+y+t-w+\tau-v$ and $\widetilde{t}=T-t+w-\tau$,
\begin{align*}
&\mathcal{Q}_{T-t+w-\tau}g_{T,v}(\log(x)+y,0)\\
&=\E^{(\log(x)+y,0)}\left[\left(\exp\left(X_{\rm e}(\widetilde{t})-\frac{T-v-\gamma(\widetilde{t})}{2}\right)-K\right)_+\right]\\
&=\E^{(\widetilde{x},0)}\left[\left(\exp\left(X_{\rm e}(\widetilde{t})-\frac{\widetilde{t}-\gamma(\widetilde{t})}{2}\right)-K\right)_+\right]\\
&=\E^{(\widetilde{x},0)}[(X(\widetilde{t})e^{-\frac{\widetilde{t}-\gamma(\widetilde{t})}{2}}-K)_+]=q(\widetilde{t},\widetilde{x}).
\end{align*}
Substituting the latter expression into \eqref{eq:presubstitut} we get \eqref{eq:possoj}.
\end{proof}

	\section*{Acknowledgements}
The authors E. Scalas and B. Toaldo acknowledge financial support under the National Recovery and Resilience Plan (NRRP), Mission 4, Component 2, Investment 1.1, Call for tender No. 104 published on 2.2.2022 by the Italian Ministry of University and Research (MUR), funded by the European Union – NextGenerationEU– Project Title “Non–Markovian Dynamics and Non-local Equations” – 202277N5H9 - CUP: D53D23005670006 - Grant Assignment Decree No. 973 adopted on June 30, 2023, by the Italian Ministry of Ministry of University and Research (MUR).

The author B. Toaldo would like to thank the Isaac Newton Institute for Mathematical Sciences, Cambridge, for support and hospitality during the programme Stochastic Systems for Anomalous Diffusion, where work on this paper was undertaken. This work was supported by EPSRC grant EP/Z000580/1.

The authors would like to thank Prof. I. Bio\v{c}i\'{c} for fruitful discussions on several technical points that improved a previous version of this manuscript.
\appendix

\section{A simple property of conditional expectations}
\begin{lem}\label{lem:condexpprop}
Let $X,Y,Z$ be three random variables with values respectively in $E_X,E_Y,E_Z$. Assume that $Z$ is independent of both $X$ and $Y$. Then, for any bounded measurable function $F:E_X \times E_Z \to \R$
\begin{equation*}
\E[F(X,Z)\mid X,Y]=\E[F(X,Z) \mid X]
\end{equation*}
\end{lem}
\begin{proof}
    First consider two sets $B_X \in \cE_X$ and $B_Z \in \cE_Z$. Then
    \begin{align*}
    \E\left[\mathds{1}_{B_X}(X)\mathds{1}_{B_Z}(Z) \mid X,Y\right]&=\mathds{1}_{B_X}(X)\E\left[\mathds{1}_{B_Z}(Z) \mid X,Y\right]\\
    &=\mathds{1}_{B_X}(X)\E\left[\mathds{1}_{B_Z}(Z)\right]\\
    &=\mathds{1}_{B_X}(X)\E\left[\mathds{1}_{B_Z}(Z) \mid X\right]\\
    &=\E\left[\mathds{1}_{B_X}(X)\mathds{1}_{B_Z}(Z) \mid X\right].
    \end{align*}
    Now fix $F \in \sigma(X,Y)$ and define for any $B \in \cE_X \otimes \cE_Z$
    \begin{equation*}
    \mu_F(B)=\E\left[\mathds{1}_F\E\left[\mathds{1}_{B}(X,Z)\mid X,Y\right]\right]
    \end{equation*}
    and
    \begin{equation*}
    \nu_F(B)=\E\left[\mathds{1}_F\E\left[\mathds{1}_{B}(X,Z)\mid X\right]\right].
    \end{equation*}
    Notice that on $B=B_X \times B_Z$ for some $B_X \in \cE_X$ and $B_Z \in \cE_Z$ it holds $\mu_F(B_X \times B_Z)=\nu_F(B_X \times B_Z)$. Since the sets of the form $B_X \times B_Z$ constitute a $\pi$-system generating $\cE_X \otimes \cE_Z$, this is enough to guarantee that $\mu_F=\nu_F$. Furthermore, since $F \in \sigma(X,Y)$ is arbitrary, this leads to
    \begin{equation*}
        \E\left[\mathds{1}_{B}(X,Z)\mid X,Y\right]=\E\left[\mathds{1}_{B}(X,Z)\mid X\right]
    \end{equation*}
    for all $B \in \cE_X \times \cE_Z$. Once this has been shown for indicator functions, a standard argument leads to the statement.
\end{proof}
\section{Proof of Lemma \ref{lem:momgen}}\label{app:B}
{The proof is by direct calculation}. Just notice that
\begin{align*}
\int_{-\infty}^{+\infty}e^{\lambda_1 |y|+\lambda_2 y}p(t,x-y)dy&=\frac{1}{\sqrt{2\pi t}}\int_{-\infty}^{+\infty}e^{\lambda_1|y|+\lambda_2 y-\frac{y^2}{2t}-\frac{x^2}{2t}+\frac{xy}{t}}dy\\
&=\frac{1}{\sqrt{2\pi t}}\int_{0}^{+\infty}e^{\lambda_1 y+\lambda_2 y-\frac{y^2}{2t}-\frac{x^2}{2t}+\frac{xy}{t}}dy\\
&+\frac{1}{\sqrt{2\pi t}}\int_{-\infty}^{0}e^{-\lambda_1 y+\lambda_2y-\frac{y^2}{2t}-\frac{x^2}{2t}+\frac{xy}{t}}dy\\
&=F(t,x;\lambda_1+\lambda_2)+F(t,-x;\lambda_1-\lambda_2)
\end{align*}
where
\begin{equation*}
F(t,x;\lambda)=\frac{1}{\sqrt{2\pi t}}\int_{0}^{+\infty}e^{\lambda y-\frac{y^2}{2t}-\frac{x^2}{2t}+\frac{xy}{t}}dy.
\end{equation*}
We only need to evaluate $F$. To do this, notice that
\begin{align*}
\lambda y-\frac{y^2}{2t}-\frac{x^2}{2t}+\frac{xy}{t}
&=-\frac{(y-(\lambda t+x))^2}{2t}+\frac{\lambda (\lambda t+2x)}{2}
\end{align*}
hence
\begin{equation*}
F(t,x;\lambda)=\frac{1}{\sqrt{2\pi t}}e^{\frac{\lambda (\lambda t+2x)}{2}}\int_{0}^{+\infty}e^{-\frac{(y-(\lambda t+x))^2}{2t}}dy.
\end{equation*}
Using the change of variables $v=\frac{y-(\lambda t+x)}{\sqrt{t}}$, we have
\begin{equation*}
F(t,x;\lambda)=\frac{1}{\sqrt{2\pi}}e^{\frac{\lambda (\lambda t+2x)}{2}}\int_{-\frac{\lambda t+x}{\sqrt{t}}}^{+\infty}e^{-v^2}{2}dv=e^{\frac{\lambda (\lambda t+2x)}{2}}\Phi\left(-\frac{\lambda t+x}{\sqrt{t}}\right).
\end{equation*}
This ends the proof.
\qed
\section{Proof of Lemma \ref{lem:BSuniqueness}}\label{app:C}	
We first recall that, by \cite[Theorems~6.20~and~8.6]{pascucci2011pde}, the equation
\begin{equation}\label{eq:heataux}
	\begin{cases}
		\partial_t v(t,y)=\frac{1}{2}\partial_y^2 v(t,y), & t>0, \ x \in \R \\
		v(0,y)=g(y), & x \in \R
	\end{cases}
\end{equation}
admits a solution $v_g$ whenever $g \in C(\R)$ satisfies $|g(y)| \le Ce^{\beta |y|}$ for suitable constants $C,\beta>0$ and such a solution is unique among functions $v(t,y)$ satisfying, for all $T>0$, $|v(t,y)| \le C_1e^{C_2y^2}$ for some constants $C_1,C_2>0$ and $t \in [0,T]$. Furthermore, still by \cite[Theorem~8.6]{pascucci2011pde}, we know that,
\begin{equation*}
	v_g(t,y)=\E^{(y,0)}[g(B(t))]
\end{equation*}
Hence, for any $T>0$ and $t \in [0,T]$,
\begin{equation*}
	|v_g(t,y)| \le \E^{(y,0)}[|g(B(t))|] \le C\E^{(y,0)}\left[e^{\beta|B(t)|}\right]\le 2Ce^{\beta |y|+\frac{\beta^2t}{2}},
\end{equation*}
where we also used Lemma \ref{lem:momgen}.
Now consider the equation
\begin{equation}\label{eq:BSclass}
	\begin{cases}
		\partial_t u(t,x)=Gu(t,x), & t>0, \ x \in \R^+ \\
		u(0,x)=f(x) & x \in \R^+,
	\end{cases}
\end{equation}
where $f \in \cC_0$ and let $\beta>0$ be such that $f \in \cC_0(\beta)$. Then if we set
\begin{equation*}
	g(y)=e^{-\frac{y}{2}}f(e^y), \ x \in \R
\end{equation*}
we do have $|g(y)| \le [f]_\beta e^{\left(\beta+\frac{1}{2}\right)|y|}$ and $g \in C(\R)$, hence we can solve the heat equation \eqref{eq:heataux}, providing the unique solution $v_g(t,y)$ that satisfies 
\begin{equation*}
|v_g(t,y)| \le 2[f]_\beta e^{\frac{1}{2}\left(\beta+\frac{1}{2}\right)^2 t+\left(\beta+\frac{1}{2}\right)|y|}.
\end{equation*}
According to \cite[Proposition~7.9]{pascucci2011pde}, we know that
\begin{equation*}
	u_f(t,x)=e^{\frac{\log(x)}{2}-\frac{t}{8}}v_g(t,\log(x))	
\end{equation*} 
is a classical solution of \eqref{eq:BSclass}. Hence, in particular, $u_f \in C(\R_0^+ \times \R^+) \cap C^1(\R^+ \times \R^+)$ and for all $t>0$ it holds $u_f(t,\cdot) \in C^2(\R^+)$. Moreover, for $x>0$ we get
\begin{equation}
\label{eq:estimateItem}	|u_f(t,x)| \le e^{\frac{|\log(x)|}{2}}|v_g(t,\log(x))| \le 2[f]_\beta e^{(\beta+1)|\log(x)|+\frac{\beta(\beta+1)}{2}t}
\end{equation}
hence $u_f \in \cC_{\rm sol}$. To show uniqueness, notice that if we consider any other solution of \eqref{eq:BSclass} in $\cC_{\rm sol}$, then we have, still by \cite[Proposition~7.9]{pascucci2011pde}, that $v^\prime_g(t,x)=e^{-\frac{x}{2}+\frac{t}{8}}u(t,e^x)$ satisfies \eqref{eq:heataux} with
\begin{equation*}
	|v^\prime_g(t,x)| \le \widetilde{C}_Te^{\frac{T}{8}}e^{\left(\widetilde{\beta}_T+\frac{1}{2}\right)|x|}
\end{equation*}
for  all $T>0$, $t \in [0,T]$, $x \in \R$. By \cite[Theorem~6.20]{pascucci2011pde}, we know that $v^\prime_g=v_g$ hence $u^\prime_f=u_f$.

Denote by $(P_s)_{s \ge 0}$ the semigroup action on $\cC_0$ such that for all $f \in \cC_0$ the function $u_f(s,x)=P_sf(x)$ is the unique solution of \eqref{eq:BSclass1} in $\cC_{\rm sol}$.

Notice that for the special initial data $f_{\beta_1,\beta_2}(x)=e^{\beta_1|\log(x)|+\beta_2\log(x)}$, that belongs to $\cC_0(|\beta_1|+|\beta_2|)$, we have $g(y)=e^{\left(\beta_2-\frac{1}{2}\right)y+\beta_1|y|}$, hence, by Lemma \ref{lem:momgen}, it holds
\begin{align*}
v_g(t,y)&=\E^{(y,0)}[g(B(t))]=e^{\left(\beta_1+\beta_2-\frac{1}{2}\right)y+\frac{1}{2}\left(\beta_1+\beta_2-\frac{1}{2}\right)^2t}\Phi\left(-\left(\beta_1+\beta_2-\frac{1}{2}\right)\sqrt{t}-\frac{y}{\sqrt{t}}\right)\\
&\qquad + e^{-\left(\beta_1-\beta_2+\frac{1}{2}\right)y+\frac{1}{2}\left(\beta_1-\beta_2+\frac{1}{2}\right)^2t}\Phi\left(-\left(\beta_1-\beta_2+\frac{1}{2}\right)\sqrt{t}+\frac{y}{\sqrt{t}}\right).
\end{align*}
As a consequence, we get, in such a case
\begin{align*}
    P_tf_{\beta_1,\beta_2}(x)&=e^{(\beta_1+\beta_2) \log(x)+\frac{(\beta_1+\beta_2)(\beta_1+\beta_2-1)}{2}t}\Phi\left(-\left(\beta_1+\beta_2-\frac{1}{2}\right)\sqrt{t}-\frac{\log(x)}{\sqrt{t}}\right)\\
&\qquad + e^{-(\beta_1-\beta_2)\log(x)+\frac{(\beta_1-\beta_2)(\beta_1-\beta_2+1)}{2}t}\Phi\left(-\left(\beta_1-\beta_2+\frac{1}{2}\right)\sqrt{t}+\frac{\log(x)}{\sqrt{t}}\right).
\end{align*}
This shows Item $(4)$.

Now we prove Item $(1)$. Recall that for $f \in \cC_0$ it holds $P_tf(x)=e^{\frac{\log(x)}{2}-\frac{t}{8}}v_g(t,\log(x))$, where $v_g$ satisfies \eqref{eq:heataux} with $g(y)=e^{-\frac{y}{2}}f(e^y)$. As before, we also recall that
\begin{equation*}
	v_g(t,y)=\int_{\R} g(z)p(t,z-y)dz=\int_{\R}e^{-\frac{z}{2}}f(e^z)p(t,z-y)dz,
\end{equation*}
where 
\begin{equation*}
p(t,z-y)=\frac{1}{\sqrt{2\pi t}}e^{-\frac{(z-y)^2}{2t}}
\end{equation*}
is the heat kernel, as in \eqref{heatker}. Hence,
\begin{equation*}
	P_tf(x)=e^{\frac{\log(x)}{2}-\frac{t}{8}}\int_{\R}e^{-\frac{z}{2}}f(e^z)p(t,z-\log(x))dz.
\end{equation*}
Using the change of variables $z=\log(w)$, we achieve
\begin{align*}
	P_tf(x)&=e^{\frac{\log(x)}{2}-\frac{t}{8}}\int_{0}^{+\infty}\frac{e^{-\frac{\log(w)}{2}}}{w}f(w)p(t,\log(w)-\log(x))dw\\
	&=\int_{0}^{+\infty}f(w)p_{\rm GBM}(t,w;x)dw.
\end{align*}
Since $p_{\rm GBM}(t,\cdot;x)$ is a probability density function, this also guarantees that $P_t$ is positivity preserving and sub-Markov. In particular, if $f \in \cC_0(\beta)$, then $|f| \le f_{\beta,0}=:f_\beta$ and then
\begin{align*}
|P_tf(x)| &\le P_t|f|(x) \le [f]_\beta P_tf_{\beta,0}(x) \le [f]_\beta e^{\beta \log(x)+\frac{\beta(\beta-1)}{2}t}\Phi\left(-\left(\beta-\frac{1}{2}\right)\sqrt{t}-\frac{\log(x)}{\sqrt{t}}\right)\\
&+[f]_\beta e^{-\beta \log(x)+\frac{\beta(\beta+1)}{2}t}\Phi\left(-\left(\beta+\frac{1}{2}\right)\sqrt{t}+\frac{\log(x)}{\sqrt{t}}\right)\le 2[f]_\beta e^{\frac{\beta(\beta+1)}{2}t+\beta|\log(x)|},
\end{align*}
that shows Item $(5)$.

To prove Item $(2)$ observe that
\begin{equation*}
	p_{\rm GBM}(t,w;x)=\frac{1}{w\sqrt{2\pi}}\exp\left(-\frac{\left(2\log(w)-2\log(x)+t\right)^2}{8t}\right)
\end{equation*}
and then
\begin{equation*}
	\partial_x p_{\rm GBM}(t,w;x)=\frac{2\log(w)-2\log(x)+t}{2txw\sqrt{2\pi}}\exp\left(-\frac{\left(2\log(w)-2\log(x)+t\right)^2}{8t}\right).
\end{equation*}
Fix $0<a<b$ and let $x \in [a,b]$. Let also $\beta \ge 0$ be such that $f \in \cC_0(\beta)$ and set $M=\max\{|\log(a)|,|\log(b)|\}$. Then we have
\begin{align*}
	|f(w)&\partial_x p_{\rm GBM}(t,w;x)| \\
        &\le [f]_\beta\frac{|2\log(w)-2\log(x)+t|}{2txw\sqrt{2\pi}}\exp\left(\beta|\log(w)|-\frac{\left(2\log(w)-2\log(x)+t\right)^2}{8t}\right)\\
	&\le \frac{[f]_\beta}{2at\sqrt{2\pi}}g_t(\log(w))\exp\left(-\log(w)-\frac{\log^2(w)}{16t}\right)\exp\left(-\frac{\log^2(w)}{32t}\right)\\
	&\quad \times\exp\left(\beta|\log(w)|-\frac{\log^2(w)}{8t}\right) \exp\left(\frac{-\log^2(w)+2\log(w)(2\log(x)-t)}{4t}\right),
\end{align*}
where
\begin{equation*}
    g_t(z)=(2|z|+2M+t)e^{-\frac{z^2}{32t}}.
\end{equation*}
First, notice that $g_t \in C(\R)$, $g_t(z) \ge 0$ for all $z \in \R$ and $\lim_{z \to \pm \infty}g_t(z)=0$. Hence, there exists $M_1(t)=\max_{z \in \R}g_t(z)$. Next, consider the parabola $r \in \R \mapsto -\frac{r^2}{8t}+\beta r \in \R$ and we notice that its vertex, which is its maximum, lies in the point $\left(4\beta t, 2\beta^2t\right)$. Similarly, the vertex of the parabola $r \in \R \mapsto -\frac{r^2}{4t}+\frac{r(2\log(x)-t)}{2t} \in \R$ lies in the point $\left(2\log(x)-t,\frac{(2\log(x)-t)^2}{4t}\right)$, which is its maximum. Finally, consider the parabola $r \in \R \mapsto -\frac{r^2}{16 t}-r \in \R$ and notice that its vertex, that is its maximum, lies in the point $\left(-8t, 4t\right)$. Hence, if we set
\begin{equation*}
    C_t:=\frac{[f]_\beta M_1(t)}{2at\sqrt{2\pi}}\exp\left(2(\beta^2+2)t\right)\max_{x \in [a,b]}\exp\left(\frac{(2\log(x)-t)^2}{4t}\right)
\end{equation*}
we get
\begin{align*}
	|f(w)\partial_x p_{\rm GBM}(t,w;x)| \le C_t \exp\left(-\frac{\log^2(w)}{32t}\right),
\end{align*}
where the right-hand side is integrable and independent of $x$. Hence, by dominated convergence,
\begin{equation*}
    \partial_x P_tf(x)=\int_{0}^{+\infty}f(w)\partial_x p_{\rm GBM}(t,w;x)\, dw.
\end{equation*}
To prove the bound, we argue as follows: first we notice that
\begin{align*}
    |\partial_x P_tf(x)| &\le \int_{0}^{+\infty}|f(w)\partial_x p_{\rm GBM}(t,w;x)|\, dw\\
    &=\frac{[f]_\beta}{2tx\sqrt{2\pi}} \int_{0}^{+\infty}\frac{\left|2\log(w)-2\log(x)+t\right|}{w}\\
    &\qquad \times \exp\left(\beta|\log(w)|-\frac{(2\log(w)-2\log(x)+t)^2}{8t}\right)\, dw.
\end{align*}
Next, we apply the change of variables $z=\frac{2\log(w)-2\log(x)+t}{2\sqrt{t}}$, obtaining
\begin{align*}
    |\partial_x P_tf(x)| &\le \frac{[f]_\beta}{x\sqrt{2\pi}} \int_{-\infty}^{+\infty}|z|\exp\left(\beta\left|z\sqrt{t}+\log(x)-\frac{t}{2}\right|-\frac{z^2}{2}\right)\, dz\\
    &\le \frac{[f]_\beta}{2x\sqrt{2\pi}}\exp\left(\beta\left|\log(x)\right|+\beta\frac{t}{2}\right)\int_{0}^{+\infty}z\exp\left(\beta \sqrt{t} z-\frac{z^2}{2}\right)\, dz\\
    &=\frac{[f]_\beta}{2x\sqrt{2\pi}}\exp\left(\beta\left|\log(x)\right|+\beta\frac{t}{2}\right)\int_{0}^{+\infty}(z-\beta\sqrt{t})\exp\left(\beta \sqrt{t} z-\frac{z^2}{2}\right)\, dz\\
    &+\frac{[f]_\beta \beta \sqrt{t}}{2x\sqrt{2\pi}}\exp\left(\beta\left|\log(x)\right|+\beta\frac{t}{2}\right)\int_{0}^{+\infty}\exp\left(\beta \sqrt{t} z-\frac{z^2}{2}\right)\, dz.\\
    &:=\frac{2[f]_\beta}{x\sqrt{2\pi}}\exp\left(\beta\left|\log(x)\right|+\beta\frac{t}{2}\right)(I_1+I_2).
\end{align*}
Concerning $I_1$, it is easy to check that $I_1=1$. For $I_2$, we rewrite it as
\begin{equation*}
    I_2=\sqrt{\frac{\pi}{2}}\E^{(0,0)}\left[e^{\beta \sqrt{t}|B(1)|}\right]=\sqrt{2\pi}e^{\frac{\beta^2 t^2}{2}}\Phi\left(-\beta \sqrt{t}\right).
\end{equation*}
Hence we get
\begin{align*}
    |\partial_x P_tf(x)| &\le \frac{2[f]_\beta}{x\sqrt{2\pi}}\exp\left(\beta\left|\log(x)\right|+\beta\frac{t}{2}\right)\left(1+\sqrt{2\pi}e^{\frac{\beta^2 t^2}{2}}\Phi\left(-\beta \sqrt{t}\right)\right)\\
    &\le 4[f]_\beta\exp\left((\beta+1)\left|\log(x)\right|+\beta\frac{t}{2}+\frac{\beta^2 t^2}{2}\right),
\end{align*}
that proves Item $(2)$. 

Finally, we prove Item $(3)$. Let $f$ as required, in particular with $f \in \cC_0(\beta)$, $\partial_x f \in \cC_0(\beta^\prime)$ and $Gf \in \cC_0(\beta^{\prime \prime}$. Recall that 
\begin{align}
		\partial_x P_t f(x) \, = \, \int_0^\infty f(w) \, \partial_x p_{\rm GBM}(t,w;x) \, dw.
\end{align}
We want to prove a similar property for the second derivative. To do this, first notice that
\begin{multline}
	 \partial_x^2p_{\rm GBM}(t,w;x)=\frac{(-8t+(2\log(w)-2\log(x)+t)^2-4t(2\log(w)-2\log(x)+t))}{8t^2w\sqrt{2\pi}x^2}\\
	 \times \exp\left(-\frac{(2\log(w)-2\log(x)+t)^2}{8t}\right).
	\end{multline}
        By a similar argument as the one adopted in Item $(2)$, we obtain, for fixed $x \in [a,b]$, $0<a<b$, and $t>0$,
	\begin{align}
	|f(x)\partial_y^2 p(t,x;x)| \le C_t\exp\left(-\frac{\log^2(w)}{32t}\right) 
	\end{align}
        for some constant $C_t$ depending on $t$, $a$, $b$, $\beta$ and $[f]_\beta$. Again, by dominated convergence, we have
        \begin{align}
		\partial^2_x P_t f(x) \, = \, \int_0^\infty f(w) \, \partial^2_x p_{\rm GBM}(t,w;x) \, dw.
    \end{align}
    Next, notice that
    \begin{align}\label{eq:adjoint}
		\frac{x^2}{2}\partial_x^2 p_{\rm GBM}(t,w;x) \, = \,  \partial_w^2 \l \frac{w^2}{2}p_{\rm GBM}(t,w;x) \r.
\end{align}
	Hence we get
    \begin{align}\label{eq:preintbp}
		G P_t f(x) \, = \, \int_0^\infty f(w) \, \partial^2_w\left( \frac{w^2}{2}p_{\rm GBM}(t,w;x)\right) \, dw.
    \end{align}
    Now we want to integrate by parts. To do this, we first notice that, again with the same exact argument, we have
    \begin{equation*}
        \left|f(w) \, \partial_w\left( \frac{w^2}{2}p_{\rm GBM}(t,w;x)\right)\right| \le C\exp\left(-\frac{\log^2(w)}{32t}\right)
    \end{equation*}
    for some constant $C$ depending on $t,x,\beta$ and $[f]_\beta$. Hence, in particular,
    \begin{equation*}
    \lim_{w \to 0}\left|f(w) \, \partial_w\left( \frac{w^2}{2}p_{\rm GBM}(t,w;x)\right)\right|=\lim_{w \to +\infty}\left|f(w) \, \partial_w\left( \frac{w^2}{2}p_{\rm GBM}(t,w;x)\right)\right|=0
    \end{equation*}
    and then, integrating by parts in \eqref{eq:preintbp},
    \begin{align}\label{eq:preintbp2}
		G P_t f(x) \, = \, -\int_0^\infty \partial_w f(w) \, \partial_w\left( \frac{w^2}{2}p_{\rm GBM}(t,w;x)\right) \, dw.
    \end{align}
    To integrate by parts a second time, we notice again that
    \begin{equation}
    \frac{w^2}{2}\left|\partial_w f(w) \, p_{\rm GBM}(t,w;x)\right| \le C \exp\left(-\frac{\log^2(w)}{32t}\right)
    \end{equation}
    for some constant $C$ depending on $t,x,\beta^\prime$ and $[\partial_xf]_{\beta^\prime}$. Thus, again
     \begin{equation*}
    \lim_{w \to 0}\frac{w^2}{2}\left|\partial_w f(w) \, p_{\rm GBM}(t,w;x)\right|=\lim_{w \to +\infty}\frac{w^2}{2}\left|\partial_w f(w) \, p_{\rm GBM}(t,w;x)\right|=0
    \end{equation*}
    and then, integrating by parts in \eqref{eq:preintbp2},
    \begin{align}
		G P_t f(x) \, = \, -\int_0^\infty Gf(w) p_{\rm GBM}(t,w;x) \, dw=P_tGf(x),
    \end{align}
    where the integral on the right-hand side is finite since $Gf \in \cC_0$.
\qed



\section{Proof of Proposition \ref{prop:generator1}} \label{appendix628}
Let us consider the process $(B(t),\tau(t))$, where $B(t)$ is a Brownian motion such that $\bP^{(x,v)}(B(0)=x)=1$ and $\bP^{(x,v)}(\tau(t)=t+v)=1$. Notice that $(B^\phi,{S})$ is obtained from $(B,\tau)$ by means of subordination. Furthermore, the semigroup $(\mathcal{T}_t)_{t \ge 0}$ of $(B(t),\tau(t))$ is given by
\begin{equation*}
    \mathcal{T}_tf(x,v)=\int_{-\infty}^{+\infty}f(y,t+v)p(t,x-y)dy, \forall f \in C_b(\R \times \R^+_0).
\end{equation*}
where $p$ is given by \eqref{heatker}. By Phillips' Theorem, we know that at least on $C_c^\infty(\R \times \R^+)$ the generator $A^\phi$ of $(B^\phi,\sigma)$ is given by
\begin{align*}
    A^\phi f(x,v)&=\frac{\alpha}{\Gamma(1-\alpha)}\int_0^{+\infty}(\mathcal{T}_tf(x,v)-f(x,v))t^{-1-\alpha}\, dt\\
    &=\frac{\alpha}{\Gamma(1-\alpha)}\int_0^{+\infty}\int_{-\infty}^{+\infty}(f(y,t+v)-f(x,v))p(t,x-y)t^{-1-\alpha}\, dy \, dt\\
    &=\frac{\alpha}{\Gamma(1-\alpha)}\int_0^{+\infty}\int_{-\infty}^{+\infty}(f(x+h,t+v)-f(x,v))p(t,h)t^{-1-\alpha}\, dh \, dt,
\end{align*}
where we applied the change of variables $x-y=h$. Now notice that {by symmetry}
\begin{equation*}
    \int_{-1}^{1}h\partial_xf(x,v)p(t,h)\, dh=0
\end{equation*}
hence we can rewrite
\begin{align}\label{eq:generator1}
    A^\phi f(x,v) &=\frac{\alpha}{\Gamma(1-\alpha)}\int_0^{+\infty}\int_{-\infty}^{+\infty}(f(x+h,t+v)-f(x,v)-h\partial_xf(x,v)\mathds{1}_{[-1,1]}(h))p(t,h)t^{-1-\alpha}\, dh \, dt.
\end{align}
Now let us first show that the integral is absolutely convergent. Indeed, we have
\begin{align*}
    \int_0^{+\infty}\int_{-\infty}^{+\infty}&|f(x+h,t+v)-f(x,v)-h\partial_xf(x,v)|p(t,h)t^{-1-\alpha}\, dh \, dt \\
    &\le \int_0^{+\infty}\int_{-\infty}^{+\infty}|f(x+h,t+v)-f(x+h,v)|p(t,h)t^{-1-\alpha}\, dh \, dt\\
    &\qquad +\int_0^{+\infty}\int_{-\infty}^{+\infty}|f(x+h,v)-f(x,v)-h\partial_xf(x,v)\mathds{1}_{[-1,1]}(h)|p(t,h)t^{-1-\alpha}\, dh \, dt\\
    &=I_1+I_2.
\end{align*}
As for $I_1$, we split the integral in two parts:
\begin{align*}
    I_1&=\int_0^{1}\int_{-\infty}^{+\infty}|f(x+h,t+v)-f(x+h,v)|p(t,h)t^{-1-\alpha}\, dh \, dt\\
    &+\int_1^{+\infty}\int_{-\infty}^{+\infty}|f(x+h,t+v)-f(x+h,v)|p(t,h)t^{-1-\alpha}\, dh \, dt=I_3+I_4.
\end{align*}
Since $f \in C_c^\infty(\R \times \R^+)$ and $\int_{-\infty}^{+\infty}p(t,h)dh=1$ we have
\begin{equation*}
    I_3 \le \max_{(x,v) \in \R \times \R^+}|\partial_v f(x,v)|\int_0^{1} t^{-\alpha}\, dt=\frac{\max_{(x,v) \in \R \times \R^+}|\partial_v f(x,v)|}{1-\alpha}
\end{equation*}
while
\begin{equation*}
    I_4 \le 2\max_{(x,v) \in \R \times \R^+}|f(x,v)|\int_1^{+\infty}t^{-1-\alpha}\, dt=\frac{\max_{(x,v) \in \R \times \R^+}|f(x,v)|}{\alpha}.
\end{equation*}
To handle $I_2$, we first change the order of the integral, defining
\begin{equation*}
    j(h):=\int_{0}^{+\infty}p(t,h)t^{-1-\alpha}dt=C_\alpha|h|^{1+2\alpha}
\end{equation*}
for some normalizing constant $C_\alpha>0$. Then we split the integral in two parts:
\begin{align*}
I_2&=\int_{h \not \in [-1,1]}|f(x+h,v)-f(x,v)|j(h)\, dh \, dt\\
&+\int_{h \in [-1,1]}|f(x+h,v)-f(x,v)-h\partial_xf(x,v)|j(h)\, dh \, dt=I_5+I_6.
\end{align*}
For $I_5$, we have
\begin{equation*}
    I_5 \le 4\max_{(x,v) \in \R \times \R^+}|f(x,v)|\int_1^{+\infty}h^{-1-2\alpha}\, dh=\frac{4\max_{(x,v) \in \R \times \R^+}}{2\alpha},
\end{equation*}
while to handle $I_6$ we observe that
\begin{equation*}
I_6 \le \max_{(x,v) \in \R \times \R^+}|\partial_x^2f(x,v)|\int_{0}^{1}h^{1-2\alpha}\, dh=\frac{\max_{(x,v) \in \R \times \R^+}|\partial_x^2f(x,v)|}{2(1-\alpha)}.
\end{equation*}
Now we want to find an alternative way of writing \eqref{eq:generator1}. To do this, we define the following family of measures on $\R$:
\begin{equation*}
    A \in \cB(\R) \mapsto \widetilde{\mathcal{K}}(A,h)=\int_{0}^{+\infty}\mathds{1}_A(t)p(t,h)\frac{\alpha t^{-1-\alpha}}{\Gamma(1-\alpha)}\, dt.
\end{equation*}
We can then define a further measure on $\R^2$ by setting for any $A_1,A_2 \in \cB(\R) $
\begin{equation*}
    \mathcal{K}(A_1 \times A_2)=\int_{-\infty}^{+\infty}\mathds{1}_{A_1}(h)\widetilde{\mathcal{K}}(A_2,h)\, dh.
\end{equation*}
On the other hand, notice that \eqref{eq:generator1} can be rewritten as
\begin{equation*}
    A^\phi f(x,v) =\int_{-\infty}^{+\infty}\int_0^{+\infty}\int_{\R}\frac{\alpha t^{-1-\alpha}}{\Gamma(1-\alpha)}(f(x+h,s+v)-f(x,v)-h\partial_xf(x,v)\mathds{1}_{[-1,1]}(h))p(t,h)\, \delta_t(ds)\,  dt \, dh.
\end{equation*}
However, observe that for any $A \in \cB(\R)$ it holds
\begin{align*}
    \widetilde{\mathcal{K}}(A,h)&=\int_{0}^{+\infty}\mathds{1}_A(t)p(t,h)\frac{\alpha t^{-1-\alpha}}{\Gamma(1-\alpha)}\, dt=\int_{0}^{+\infty}\delta_t(A)p(t,h)\frac{\alpha t^{-1-\alpha}}{\Gamma(1-\alpha)}\, dt\\
    &=\int_{0}^{+\infty}\int_{\R} \mathds{1}_A(s)p(t,h)\frac{\alpha t^{-1-\alpha}}{\Gamma(1-\alpha)}\,\delta_t(ds)\, dt.
\end{align*}
By a standard measure theory argument we can conclude that
\begin{align*}
    A^\phi f(x,v)&=\int_{-\infty}^{+\infty}\int_{\R}(f(x+h,s+v)-f(x,v)-h\partial_xf(x,v)\mathds{1}_{[-1,1]}(h))\widetilde{\mathcal{K}}(ds;h) \, dh\\
    &=\int_{\R}\int_{\R}(f(x+h,s+v)-f(x,v)-h\partial_xf(x,v)\mathds{1}_{[-1,1]}(h))\mathcal{K}(dh;ds) \, dh,
\end{align*}
where the latter follows by definition of the measure $K$.
\qed

	\bibliographystyle{abbrv}	
	\bibliography{biblio}
	
\end{document}